\documentclass[10pt, A4, leqno, showkeys]{amsart} 

\usepackage{graphicx} 

      \usepackage{mathtools,amssymb} 
      \usepackage{amsthm}
      \usepackage{amsfonts}
  \usepackage{url} 
  \usepackage{bm} 
  \usepackage{mathrsfs} 
  \mathtoolsset{showmanualtags} 
    \usepackage{tikz}
    \usetikzlibrary{intersections,calc,arrows}
    \usepackage[inline]{enumitem} 
    \setlist[enumerate]{font=\normalfont, leftmargin = *}
    \usepackage{comment} 
    \usepackage{ifthen} 
    \usepackage{array} 
    \usepackage{cite} 
    \usepackage{hhline} 
    \usepackage{makecell} 

  \newcommand{\paren}[1]{(#1)}
    \newcommand{\cparen}[1]{{\left\{#1\right\}}}
    \newcommand{\sparen}[1]{{\left[#1\right]}}
    \newcommand{\abs}[1]{{\left\lvert#1\right\rvert}}
    
    \bmdefine{\bzero}{0} 
    \bmdefine{\be}{e} 
    \bmdefine{\bv}{v} 
    \bmdefine{\bu}{v} 
    \bmdefine{\bw}{w} 
    \newcommand{\pfrac}[2]{{\frac{\partial #1}{\partial #2}}}
    
    \newcommand{\sfrac}[2]{{\!\left. #1 \middle\slash #2 \right.\!}}
    \newcommand{\lup}[2]{{{{\vphantom{#2}}^{#1}}{#2}}}
    \newcommand{\trans}[1]{{\lup{t}{#1}}}

    \newcommand{\rel}[1]{\mathrel{}#1\mathrel{}} 
    \newcommand{\relmiddle}[1]{\rel{\middle#1}} 
    \newcommand{\mvert}{\relmiddle{\vert}} 
    \newcommand{\smallo}[1]{{o (#1)}}
    
    \newcommand{\Span}[2][]{\operatorname{span}_{#1} (#2)}
    \newcommand{\notni}{\not\ni}
    
    \newcommand{\sgn}{\operatorname{sgn}}

    \newcommand{\id}[1][]{{\mathrm{id}_{#1}}}
    \newcommand{\GL}[2][]{{\operatorname{GL}_{#1} \paren{#2}}}
    
    \newcommand{\restrict}[2]{{\! {\left. {#1} \right\rvert}_{#2}}}
    
      \renewcommand{\phi}{\varphi}
      \renewcommand{\epsilon}{\varepsilon}

    \newcommand{\quatext}[1]{{\quad\text{#1}\quad}}
    \usepackage{mleftright}
    \mleftright
    \newcommand{\mathemph}[1]{{\text{\emph{$#1$}}}}

    \newcommand{\Cusp}[1]{\ensuremath{C_{(#1)}}}
    \newcommand{\A}{\ensuremath{\mathcal{A}}}
    \newcommand{\cL}{\ensuremath{\mathcal{L}}}
    \newcommand{\Aqui}{\A}
      \newcommand{\deriv}[2]{{#1}^{\paren{#2}}}
      \newcommand{\sderiv}[2]{{#1}^{\sparen{#2}}}
      
        \bmdefine{\zero}{0}
        
        \newcommand{\R}{\bm{R}}
        \newcommand{\Z}{\bm{Z}}
        \bmdefine{\V}{V}
        \bmdefine{\bbR}{R}
        \bmdefine{\bbZ}{Z}
        \newcommand{\poZ}{\Z_{> 0}}
        \newcommand{\nnZ}{\Z_{\geq 0}}
        \newcommand{\poR}{\R_{> 0}}
        
        \newcommand{\geZ}[1]{\Z_{\geq #1}}
    \renewcommand{\hat}[1]{\widehat{#1}}
    \renewcommand{\tilde}[1]{\widetilde{#1}}
    
    \newcommand{\nscite}[1]{\!\!\cite{#1}}
    \newcommand{\rscite}[2]{\cite[#1]{#2}}
    \newcommand{\rcite}[2]{\!\!\rscite{#1}{#2}}
        
        \newcommand{\Fcns}[3]{\mathcal{F}^{#1}_{#2} (#3)}
        \newcommand{\cR}[2]{\mathcal{R}^{#1} (#2)}
        \newcommand{\cNR}[2]{\mathcal{NR}^{#1} (#2)}

      \newcommand{\Rpdf}[1][N]{\texorpdfstring{\ensuremath{\R^{#1}}}{\ensuremath{R^{#1}}}}
      \newcommand{\Apdf}{\texorpdfstring{\ensuremath{\mathcal{A}}}{\ensuremath{A}}}
      \renewcommand{\emptyset}{\varnothing}
      \newcommand{\denu}[2]{d_{#1} \paren{#2}}
      \newcommand{\Matrices}[2]{M_{#1} \paren{#2}}
      \newcommand{\parts}[1]{\mathcal{P}_{#1}}
      
    \newcounter{num}
    
    \newcommand{\rnum}[1]{{\setcounter{num}{#1}\roman{num}}}
      \usepackage{caption} 
      \usepackage{subcaption}
      \usepackage{graphicx}
      \usepackage{fancybox} 
    \makeatletter
    \renewcommand{\p@enumii}{}
    \renewcommand{\p@enumiii}{}
    \makeatother

  \usepackage{hyperref}
  \hypersetup{
    colorlinks=true,       
    linkcolor=blue,        
    citecolor=blue,       
    filecolor=blue,     
    urlcolor=magenta          
  }
  \hypersetup{
    linktoc=all,         
    bookmarksnumbered=true 
  }

    \usepackage[capitalize,nameinlink,nosort]{cleveref}
    \usepackage{autonum}
    \usepackage{aliascnt}
    \theoremstyle{plain} 
    \newtheorem{thm}{Theorem}[section] 
        \newaliascnt{defn}{thm}
        \newaliascnt{lem}{thm}
        \newaliascnt{cor}{thm}
        \newaliascnt{prop}{thm}
        \newaliascnt{exm}{thm}
        \newaliascnt{rem}{thm}
        \newaliascnt{fact}{thm}
          \newtheorem{fact}[fact]{Fact}
          \newtheorem{cor}[cor]{Corollary}
          \newtheorem{lem}[lem]{Lemma}
          \newtheorem{prop}[prop]{Proposition}
        \theoremstyle{definition}
          \newtheorem{defn}[defn]{Definition}
          
          \newtheorem{rem}[rem]{Remark}
          \newtheorem*{acknowledgements}{Acknowledgements}
        \theoremstyle{remark}
        \aliascntresetthe{defn}
        \aliascntresetthe{lem}
        \aliascntresetthe{cor}
        \aliascntresetthe{prop}
        \aliascntresetthe{exm}
        \aliascntresetthe{rem}
        \aliascntresetthe{fact}

    \crefname{equation}{}{}
    \crefname{enumi}{}{}
    \crefname{enumii}{}{}
    \crefname{enumiii}{}{}
    \crefname{thm}{Theorem}{Theorems}
    \crefname{defn}{Definition}{Definitions}
    \crefname{lem}{Lemma}{Lemmas}
    \crefname{cor}{Corollary}{Corollaries}
    \crefname{prop}{Proposition}{Propositions}
    \crefname{exm}{Example}{Examples}
    \crefname{rem}{Remark}{Remarks}
    \crefname{prblm}{Problem}{Problems}
    \crefname{footnote}{Footnote}{Footnotes}
    \crefname{claim}{Claim}{Claims}
    \crefname{fact}{Fact}{Facts}
    \crefname{section}{Section}{Sections}
    \crefname{subsection}{Subsection}{Subsections}
      \crefname{figure}{Figure}{Figures}
      \crefname{table}{Table}{Tables}
      \crefrangeformat{equation}{#3#1#4--#5#2#6}
      \crefrangeformat{enumi}{#3#1#4--#5#2#6}
      \crefrangeformat{enumii}{#3#1#4--#5#2#6}
      \crefrangeformat{enumiii}{#3#1#4--#5#2#6}

  \numberwithin{equation}{section}

\title[Curves of Finite Multiplicities in $\R^N$]
{Criteria and Curvatures for Singularities of Finite Multiplicities of Curves in $\boldsymbol{R}^N$}

\author{Jun Matsumoto}
\address[Jun Matsumoto]{ 
Department of Mathematics, \endgraf
Institute of Science Tokyo, \endgraf
O-okayama, Meguro, Tokyo, 152-8551, Japan
}
\email{j.matsumoto.517@gmail.com}

\author{Shuki Sano}
\address[Shuki Sano]{
Department of Mathematics, \endgraf
Institute of Science Tokyo, \endgraf
O-okayama, Meguro, Tokyo, 152-8551, Japan
}
\email{sano.s.7465@m.isct.ac.jp}

\author{Kiyoto Yanagida}
\address[Kiyoto Yanagida]{
Department of Mathematics, \endgraf
Institute of Science Tokyo, \endgraf
O-okayama, Meguro, Tokyo, 152-8551, Japan
}
\email{kiyo1087gida@gmail.com}

\date{\today}



\begin{document}

\maketitle

\begin{abstract}
  First, this paper presents a systematic procedure for constructing criteria 
  for singularities of curves of finite multiplicities in $\R^N$.
  Based on this method, we provide explicit criteria for  
  singularities of multiplicities two, three, and four, 
  including specific cusps appearing only in dimensions three or higher.
  Furthermore, we generalize the normalized curvature functions 
  and the cuspidal curvature to singular curves in $\R^N$. 
  Using these generalized curvatures, we reinterpret the existence and uniqueness theorem
  given by Fukui for curves in $\R^N$ of finite multiplicities.
    
\end{abstract}

\section{Introduction} \label{sec:Introduction}

This paper is devoted to studying some criteria and geometric properties for curves in $\R^N$
with singularities---points where the first derivative of a curve vanishes.
When a curve admits singularities, evaluating its geometric properties requires specialized approaches. 
A fundamental problem in singularity theory is the classification of map germs under $\mathcal{A}$-equivalence 
(see \cref{ssec:Taylor_coefficients} for the precise definition). 
In particular, 
it is useful for geometric applications to establish easily computable criteria
for standard cusps in $\R^N$ expressed by 
$t \mapsto (t^{n_1}, t^{n_2}, \dots, t^{n_k})$.
It is known that the planar $(m,n)$-cusp ($t \mapsto (t^m, t^n)$)
has an application in optics (\!\!\cite{Zhang_Pei_oprik}).

For plane curves, various criteria for specific singularities 
of finite multiplicity (\cref{defn:multiplicity}) 
have been extensively investigated and utilized to study geometric invariants
(\cref{rem:criteria_known}). 
Recently, Fukui and Hoshino \cite{Fukui_Hoshino_curvature_criteria} showed criteria for 
$\A$-simple real singularities, which Bruce and Gaffney \cite{BG82_simple_sing} classified
for the complex case,  
by taking a special parameter of curves: the curvature parameter.
Despite these developments, a systematic method to construct criteria for singularities 
of finite multiplicities for curves in general higher-dimensional space $\R^N$ 
has not yet been fully established.
While criteria for singularities have thus far been established primarily for plane curves, 
singular space curves naturally arise as the singular sets of surfaces
(see \cref{fig:sing_surface}). 
Therefore, it is a natural geometric progression 
to extend criteria of singularities for curves in higher-dimensional space.

\begin{figure}[h]
  \centering
  \begin{tabular}{c@{\hspace{2cm}}c}
      \includegraphics[width=50mm]{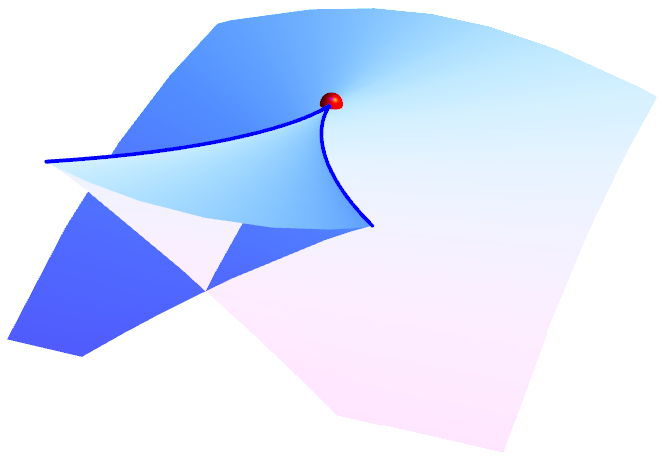} &
      \includegraphics[width=35mm]{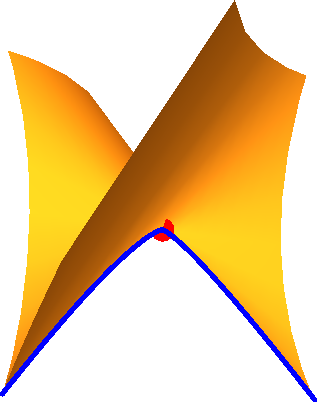}\\
      Swallowtail & Cuspidal butterfly
  \end{tabular}
  \caption{Singular curves appearing on surfaces in $\R^3$} 
  \label{fig:sing_surface}
  
  \vspace{1ex} 
  \begin{minipage}{0.9\linewidth} 
    {
    \small
    The blue curve in each figure above is the image of the singular set 
    of the swallowtail and the cuspidal butterfly.
    The left curve in the above figure is a $(2, 3)$-cusp, and 
    the right curve is a $(3, 4, 5)$-cusp in $\R^3$.
    The red point denotes the singular point of each curve.
    }
  \end{minipage}
\end{figure}

On the other hand, from a differential geometric perspective, 
analyzing the behavior of curvature near singularities is essential. 
For curves with an ordinary cusp (e.g., $(2,3)$-cusp), Shiba and Umehara 
\cite{Umehara_2011_simplification, Shiba_Umehara_2012_curvature_functions}
clarified the behavior of the curvature function at the singularity, 
introducing the normalized curvature function, the half-arclength parameter,
and the cuspidal curvature 
to formulate an existence and uniqueness theorem. 
Fukui \cite{Fukui_2017_multiplicities} later generalized this differential geometric approach, and
established a fundamental theorem for curves with singularities of finite multiplicity. 
Moreover, Honda--Saji \cite{Honda_Saji_25_invariant}
and Martins--Saji--Santos--Teramoto \cite{Martins_Saji_Santos_Teramoto_2024_bdd_geom_inv}
introduced some geometric invariants---secondary curvature, bias, further in general, 
$(m,n)$-bias, and $(m, n)$-cuspidal curvature---for planar $(m, n)$-cusp. 

Building upon these foundational works, there is a natural demand to further generalize these curvature functions 
to curves in $\R^N$ 
and to clarify their differential geometric meanings in connection with specific $\mathcal{A}$-equivalence classes.

The purposes of this paper are as follows:
\begin{itemize}
  \item 
   We give a systematic procedure for constructing criteria for singularities appearing on curves in arbitrary dimension $\R^N$. 
  \item
    We introduce some curvatures for curves in $\R^N$ of finite multiplicities.                                                 
\end{itemize}
\noindent
This paper is organized as follows.

\cref{sec:preliminaries} introduces various preliminary notions needed for the discussions in \cref{ssec:RN.construct_suff} and beyond.
In \cref{ssec:Taylor_coefficients}, we introduce notation and conventions we use in this paper, and we review the known criteria for singularities of plane curves in terms of Taylor coefficients (\cref{rem:criteria_known}).
\cref{ssec:FdB_formula} introduces Fa\`{a} di Bruno's formula for curves in $\R^N$, which we use in the computation of derivatives.
In \cref{ssec:multiplicity}, we introduce the concept of multiplicity of curves defined by Fukui \cite{Fukui_2017_multiplicities}.
\cref{sec:Frobenius.numbers} introduces concepts from number theory, such as the Frobenius problem.
While this material deviates from the discussion of singularities, we present it there because the notation established in \cref{defn:denumerant}, particularly $\cNR{2}{A}$, is utilized in \cref{thm:A.equivalence.suff.cond.1}.

In \cref{ssec:RN.construct_suff,ssec:RN.examples}, we provide examples of the construction of criteria for several \A-equivalence classes of singularities of curves in $\R^N$ based on the following procedure:
\begin{enumerate}[label={\textbf{Step~\arabic*}.},ref={\textbf{Step~\arabic*}}]
  \item \label{item:construct.procedure:step.1}%
        Select a curve $\varGamma$ for which the criteria are to be determined, and determine the set $\cNR{2}{A}$ or $\cNR{2}{\varGamma}$ (see \cref{defn:denumerant,defn:space.F.of.functions}).
        Although sufficient conditions for being \A-equivalent to $\varGamma$ can be obtained directly from \cref{thm:general.property.A.equivalence,thm:A.equivalence.suff.cond.1}, these conditions are typically too strong for practical applications.
  \item \label{item:construct.procedure:step.2}%
        Construct weaker sufficient conditions by employing parameter changes and coordinate transformations of $\R^N$.
  \item \label{item:construct.procedure:step.3}%
        Reformulate the constructed conditions into those based on the linear independence of Taylor coefficients of a curve by utilizing linear transformations of $\R^N$.
  \item \label{item:construct.procedure:step.4}%
        Prove, via explicit computation of derivatives, that the constructed sufficient conditions are also necessary.
        If the constructed conditions fail to be invariant under \A-equivalence, return to \ref{item:construct.procedure:step.2}.
\end{enumerate}
\cref{sec:Singularities.multiplicity.2,sec:Singularities.multiplicity.3,sec:Singularities.multiplicity.4} provide criteria for classifying singularities of multiplicity $2$, $3$, and $4$, respectively (\cref{thm:criteria.2.3.and.2.5.and.2.7,conj:criteria.3.4.5.and.3.4.and.3.5.7.and.3.5,conj:criteria.4.5.and.other}).
All the known criteria mentioned in \cref{rem:criteria_known} can be obtained as the $N = 2$ cases of these theorems.
Moreover, \cref{conj:criteria.3.7.etc} gives another perspective to 
\cite[Theorem~3.4]{Fukui_Hoshino_curvature_criteria}.
Furthermore,
    $\paren{3, 4, 5}$-cusps
    and $\paren{3, 5, 7}$-cusps
    in \cref{conj:criteria.3.4.5.and.3.4.and.3.5.7.and.3.5},
    $\paren{3, 7, 8}$-cusps,
    $\paren{3, 7, 11}$-cusps, and
    $\paren{3, 7_{+8}, 11}$-cusps
    in \cref{conj:criteria.3.7.etc},
    and $\paren{4, 5, 6}$-cusps,
    $\paren{4, 5, 7}$-cusps,
    $\paren{4, 5, 11}$-cusps,
    and $\paren{4, 5_{\pm 7}, 11}$-cusps
    in \cref{conj:criteria.4.5.and.other}, appear only when $N \geq 3$.
Also, $\paren{4, 5, 6, 7}$-cusps in \cref{conj:criteria.4.5.and.other} appear only when $N \geq 4$.

In \cref{sec:curvature}, we begin by recalling basic notions of regular curves in $\R^N$. 
The purpose of the latter part of \cref{sec:curvature} is to extend the notion of the 
normalized curvature function defined by Shiba and Umehara \cite{Shiba_Umehara_2012_curvature_functions} to curves in $\R^N$ with a singularity of finite multiplicity (\cref{defn:m.nom.curv.fcn}).
In addition, for curves in $\R^N$ with a singularity of finite multiplicity, we define another kind of curvature, which is a generalization of the cuspidal curvature introduced by Umehara \cite{Umehara_2011_simplification} for plane curves of multiplicity $2$ (\cref{defn:kappa.m.i}), and give a generalization of the relation between the cuspidal curvature and the normalized curvature function due to Shiba and Umehara \cite{Shiba_Umehara_2012_curvature_functions} (\cref{thm:smoothness.of.nom.curv.fcn} \ref{item:existence.of.sing.curv.}).
Furthermore, we reinterpret the existence and uniqueness theorem given by Fukui \cite{Fukui_2017_multiplicities} for curves in $\R^N$ with singularities of finite multiplicity, using the normalized curvature functions generalized to space curves (\cref{thm:existence.uniqueness.with.sing.}).

\begin{acknowledgements}
The authors would like to thank 
Toshizumi Fukui, Saiki Hoshino, Yoshiki Matsushita, and 
Kotaro Yamada for their helpful advice and comments on this research.
In addition, the authors wish to express their gratitude to Kenta Hayano
for pointing out the connection between \cref{thm:general.property.A.equivalence}
and \cite[Corollary~2.5.1]{Bruce_1987_determinacy_and_unipotency}.
The first author was supported by JST SPRING, Japan Grant Number JPMJSP2180.
\end{acknowledgements}

\section{Preliminaries}
\label{sec:preliminaries}%

  \subsection{Notation and conventions}\label{ssec:Taylor_coefficients}
      
      For $m, n \in \poZ$,
      we denote an $\R^n$-valued $C^{\infty}$-map germ $f$
      defined on an open neighborhood of $\bm{a} \in \R^m$ by
      $f \colon \paren{\R^m, \bm{a}} \to \R^n$,
      and by
      $f \colon \paren{\R^m, \bm{a}} \to \paren{\R^n, \bm{b}}$
      when $f \paren{\bm{a}} = \bm{b} \in \R^n$ is given.
      Throughout this paper,
      for $N \in \geZ{2}$,
      a curve in $\R^N$ is a map germ
      $\gamma \colon \paren{\R, 0} \to (\R^N, \bzero)$.
      For $k \in \cparen{2, \dots, N}$ and $\paren{n_1, \dots, n_k} \in \paren{\geZ{2}}^k$, we denote by $\Cusp{n_1, \dots, n_k}$ the curve $\Cusp{n_1, \dots, n_k} \paren{t} \coloneqq \paren{t^{n_1}, \dots, t^{n_k}, 0, \dots, 0}$ in $\R^N$.
      We denote by
      \begin{alignat}{2}
      \Cusp{3, 7_{+ 8}} (t) & \coloneqq \paren{t^3, t^7 + t^8, 0, \dots, 0},
      \quad &
      \Cusp{3, 7_{+ 8}, 11} (t) & \coloneqq \paren{t^3, t^7 + t^8, t^{11}, 0, \dots, 0}, \\
      \Cusp{4, 5_{+7}} \paren{t} & \coloneqq \paren{t^4, t^5 + t^7, 0, \dots, 0},
      \quad &
      \Cusp{4, 5_{-7}} \paren{t} & \coloneqq \paren{t^4, t^5 - t^7, 0, \dots, 0}, \\
      \Cusp{4, 5_{+7}, 11} \paren{t} & \coloneqq \paren{t^4, t^5 + t^7, t^{11}, 0, \dots, 0},
      \quad &
      \Cusp{4, 5_{-7}, 11} \paren{t} & \coloneqq \paren{t^4, t^5 - t^7, t^{11}, 0, \dots, 0}.
      \end{alignat}
      We also abbreviate
      $\Cusp{3, 7}$ and $\Cusp{3, 7_{+8}}$ as $\Cusp{3, 7_{8 \sigma}}$ ($\sigma = 0, 1$),
      $\Cusp{3, 7, 11}$ and $\Cusp{3, 7_{+8}, 11}$ as $\Cusp{3, 7_{8 \sigma}, 11}$ ($\sigma = 0, 1$),
      and $\Cusp{4, 5}$, $\Cusp{4, 5_{+7}}$, and $\Cusp{4, 5_{-7}}$ as $\Cusp{4, 5_{7 \sigma}}$ ($\sigma = 0, \pm 1$).
          
      We say that two curves
      $\gamma_1, \colon \paren{\bbR, 0} \to \paren{\bbR^N, \bzero}$
      are \emph{\A-equivalent} (or \emph{right-left equivalent})
      to $\gamma_2 \colon \paren{\bbR, 0} \to \paren{\bbR^N, \bzero}$
      if there exist diffeomorphism germs 
      $\phi \colon \paren{\R, 0} \to \paren{\R, 0}$
      and $\varPhi \colon \paren{\R^N, \bzero} \to \paren{\R^N, \bzero}$ 
      such that $\varPhi \circ \gamma_1 \circ \phi = \gamma_2$. 
      In particular, if $\phi = \text{id}$, then $\gamma_1$ and $\gamma_2$ 
      are said to be \emph{\cL-equivalent} (or \emph{left-equivalent}).
      A curve in $\R^N$ is said to have an $X$-cusp (e.g., a $(2, 3)$-cusp) at $t = 0$ if it is \A-equivalent to the model curve $C_X$ (e.g., $\Cusp{2, 3}$).
      For convenience,
      $\phi \colon \paren{\R, 0} \to \paren{\R, 0}$ will be referred to as a \emph{parameter change}, 
      and $\varPhi \colon \paren{\R^N, \bzero} \to \paren{\R^N, \bzero}$ as a \emph{coordinate transformation}.
      We denote by $J_{\varPhi}$ the Jacobian matrix of a coordinate transformation $\varPhi$.
      Unless confusion arises, partial derivatives are abbreviated; for example, 
      $J_{\varPhi} \paren{\bzero}$ is simply written as $J_{\varPhi}$.

      For a curve $\gamma \colon \paren{\R, 0} \to \paren{\R^N, \bzero}$, 
      let us denote by $\sderiv{\gamma}{i}$ the vector
      \begin{equation}
            \mathemph{\sderiv{\gamma}{i}}
            \coloneqq \frac{1}{i!} \,
              \restrict{\frac{d^i \gamma}{d t^i}}{t = 0}
            = \frac{1}{i!} \,
              \deriv{\gamma}{i} \paren{0}
            \in \R^N
            \quad
            (i \in \nnZ),
            \label{eq:defn.divided.derivative}
      \end{equation}
      which is the $i$-th Taylor coefficient of $\gamma$ at $t = 0$. 
      Similarly, we use the notation $f^{[i]}$ for a function germ 
      $f \colon \paren{\R, 0} \to \paren{\R, 0}$. 
      
      For an $n$-dimensional multi-index 
      $\alpha = \paren{\alpha_1, \dots, \alpha_n} \in \paren{\nnZ}^n$, 
      we use the following notation: 
      \begin{itemize}
          \item sum of components
          $\mathemph{\abs{\alpha}}
            \coloneqq \sum_{i = 1}^{n} \alpha_i$,
          \item factorial
          $\mathemph{\alpha!}
            \coloneqq \prod_{i = 1}^{n} \alpha_i !$, and
          \item partial derivative
          $\mathemph{\partial^{\alpha}}
            \coloneqq \frac{\partial^{\abs{\alpha}}}{\partial x_1^{\alpha_1} \cdots \partial x_n^{\alpha_n}}$.
       \end{itemize}
    
    Let $o(t^n)$ be the Landau small $o$ notation as $t \to 0$, and let $\Span{\bv_1, \dots, \bv_k}$ be the linear subspace spanned by vectors $\bv_1, \dots, \bv_k \in \R^N$.
    For brevity, we identify column vectors with their row-vector representations.
    Finally, $\sgn x$ denotes the sign of $x \in \R$, i.e., $\sgn x \in \{0, \pm 1\}$ according to the sign of $x$.

    \begin{rem} \label{rem:criteria_known}
        By using the notation $\gamma^{[i]}$ for the Taylor coefficients, we can prevent the coefficients appearing in the criteria for curves from becoming large. 
        With this notation, known criteria can be rewritten as follows.
        Let $\gamma \colon \paren{\R, 0} \to \paren{\R^2, \bzero}$ be a plane curve.
        \begin{enumerate}[label={(\arabic*)}]
            \item \label{item:cor.known.criteria:2.3}%
                  $\gamma$ is \A-equivalent to
                  $\mathemph{\Cusp{2, 3}}$
                  if and only if
                  $\sderiv{\gamma}{1}
                    = \bzero$
                  and
                  $
                    \det \paren{%
                      \sderiv{\gamma}{2},
                      \sderiv{\gamma}{3}
                    }
                    \neq 0.
                  $
            \item \label{item:cor.known.criteria:2.5}%
                  $\gamma$ is \A-equivalent to
                  $\mathemph{\Cusp{2, 5}}$
                  if and only if there exists $\lambda \in \R$
                  such that
                  $\sderiv{\gamma}{1}
                    = \bzero$,
                  $
                    \sderiv{\gamma}{3}
                    = \lambda \, \sderiv{\gamma}{2},
                  $ 
                  and
                  $
                    \det \paren{%
                      \sderiv{\gamma}{2},
                      \sderiv{\gamma}{5}
                      - 2 \lambda \, \sderiv{\gamma}{4}
                    }
                    \neq 0.
                  $
            \item \label{item:cor.known.criteria:2.7}%
                  $\gamma$ is \A-equivalent to
                  $\mathemph{\Cusp{2, 7}}$
                  if and only if there exist $\lambda, \mu \in \R$
                  such that
                  $\sderiv{\gamma}{1}
                    = \bzero$,
                  $
                    \sderiv{\gamma}{3}
                    = \lambda \, \sderiv{\gamma}{2}$,
                  $\sderiv{\gamma}{5}
                    - 2 \lambda \, \sderiv{\gamma}{4}
                    = \mu \, \sderiv{\gamma}{2},
                  $
                  and
                  \begin{equation}
                    \det \paren{%
                      \sderiv{\gamma}{2},
                      \sderiv{\gamma}{7}
                      - 3 \lambda \, \sderiv{\gamma}{6}
                      + \paren{3 \lambda^3 - 2 \mu} \sderiv{\gamma}{4}
                    }
                    \neq 0.
                  \end{equation}
            \item \label{item:cor.known.criteria:3.4}%
                  $\gamma$ is \A-equivalent to
                  $\mathemph{\Cusp{3, 4}}$
                  if and only if
                  $\sderiv{\gamma}{1}
                    = \sderiv{\gamma}{2}
                    = \bzero$
                  and
                  $
                    \det \paren{%
                      \sderiv{\gamma}{3},
                      \sderiv{\gamma}{4}
                    }
                    \neq 0.
                  $
            \item \label{item:cor.known.criteria:3.5}%
                  $\gamma$ is \A-equivalent to
                  $\mathemph{\Cusp{3, 5}}$
                  if and only if
                  $\sderiv{\gamma}{1}
                    = \sderiv{\gamma}{2}
                    = \bzero,
                    \det \paren{%
                      \sderiv{\gamma}{3},
                      \sderiv{\gamma}{4}
                    }
                    = 0,
                    $
                    and
                    $
                    \det \paren{%
                      \sderiv{\gamma}{3},
                      \sderiv{\gamma}{5}
                    }
                    \neq 0.
                  $
            \item \label{item:cor.known.criteria:4.5}%
                  For each $\sigma \in \cparen{0, \pm 1}$,
                  $\gamma$ is \A-equivalent to 
                  $\Cusp{4, 5}$ $(\sigma = 0)$,
                  $\Cusp{4, 5_{+7}}$ $(\sigma = 1)$, or
                  $\Cusp{4, 5_{-7}}$ $(\sigma = -1)$
                  if and only if
                  $
                    \sderiv{\gamma}{1}
                    = \sderiv{\gamma}{2}
                    = \sderiv{\gamma}{3}
                    = \bzero,
                    \det \paren{%
                      \sderiv{\gamma}{4},
                      \sderiv{\gamma}{5}
                    }
                    \neq 0,
                  $
                  and respectively
                  \begin{equation}
                    \sgn \left(\mu_2 - \frac{5}{4} \, \lambda_2 - \frac{11}{10} \, \mu_1^2\right)
                    = \mathemph{\sigma},
                    \label{eq:cirteria_45etc_sgn}%
                  \end{equation}
                  where
                  $\lambda_2, \lambda_3, \mu_1, \mu_2
                    \in \R$
                  are such that
                  $
                    \sderiv{\gamma}{6}
                    = \lambda_2 \, \sderiv{\gamma}{4}
                    + \mu_1 \, \sderiv{\gamma}{5},
                    \sderiv{\gamma}{7}
                    = \lambda_3 \, \sderiv{\gamma}{4}
                    + \mu_2 \, \sderiv{\gamma}{5}.
                    \label{eq:thm.criteria.4.5.like:defn.lambda23.mu12}%
                  $
        \end{enumerate}
        
        These criteria are collected from the following sources:
        \begin{enumerate*}[label={(\arabic*)}]
            \item \cite[p.12]{Porteous_2001_curves_and_surfaces},
            \item \cite[Theorem~1.23]{Porteous_2001_curves_and_surfaces},
            \item \cite[Theorem~A.1]{Hattori_Honda_Morimoto_2024_Bours_theorem_singularities_arXiv},
            \item \cite[p.12]{Porteous_2001_curves_and_surfaces},
            \item \cite[Fact~2.1]{Martins_Saji_Santos_Teramoto_2019_singular_surfaces}, and
            \item \cite[Theorems~4.1 and 4.13]{Matsushita_2024_classifications_4_5_cusps_arXiv}.
        \end{enumerate*}
    \end{rem}
    
  \subsection{Fa\`{a} di Bruno's Formula for Curves in \Rpdf[N]}
  \label{ssec:FdB_formula}
    \begin{defn}
      \label{defn:partition_of_integer}
      Let $n \in \poZ$.
      A \emph{partition} of $n$ is an $n$-dimensional multi-index
      $\beta \coloneqq \paren{\beta_1, \beta_2, \dots, \beta_n}
        \in \paren{\nnZ}^n$
      satisfying $\sum_{i = 1}^{n} i \, \beta_i = n$.
      We denote by $\mathemph{\parts{n}}$
      the set of partitions of $n$.
     
    \end{defn}
    
    \begin{fact}[Fa\`{a} di Bruno's formula, see \rscite{Theorem~2.1}{Okano_2001_Bruno_applications}]
        \label{fact:Bruno.formula}
        Let $n \in \poZ$.
        We consider two functions
        $g \colon \paren{\R, 0} \to \R$
        and $f \colon \paren{\R, g \paren{0}} \to \R$.
        Then, the composite function
        $f \circ g \colon \paren{\R, 0} \to \R$
        satisfies
        \begin{equation}
            \deriv{\paren{f \circ g}}{n} \paren{s}
            = \sum_{d = 1}^{n}
                \sum_{\substack{%
                    \beta \in \parts{n} \\
                    \abs{\beta} = d
                  }}
                  \frac{n!}{%
                    \beta !
                  }
                  \paren{%
                    \deriv{f}{d} \left(g \paren{s}\right)
                  }
                  \prod_{i = 1}^{n}
                    \left(%
                      \frac{\deriv{g}{i} \paren{s}}{i !}
                 \right)^{\beta_i}.
            \label{eq:Bruno_formula.scalar}
        \end{equation}
       
    \end{fact}

    \cref{fact:Bruno.formula}
    can be extended as follows.

    \begin{fact}[\rcite{Theorem 2.1}{Mishkov_2000_Generalization_Bruno}]
        \label{fact:Bruno.extension}%

        Let $n \in \poZ$ and $r \in \geZ{2}$.
        We consider an $\R^r$-valued function
        $\bm{x} \paren{t} = \paren{x_1 \paren{t}, \dots, x_r \paren{t}}
          \colon \paren{\R, 0} \to \R^r$
        and a function
        $f \colon \paren{\R^r, \bm{x} \paren{0}} \to \R$.
        Then, the composite function
        $f \circ \bm{x} \colon \paren{\R, 0} \to \R$
        satisfies
        \begin{equation}
          \deriv{\paren{f \circ \bm{x}}}{n} \paren{t}
          = \sum_{\beta \in \parts{n}}
              \sum_{\paren{q_{i j}} \in \mathcal{Q}_{\beta}}
                \frac{n!}{%
                  \displaystyle
                  \prod_{i = 1}^{n}
                    \paren{i!}^{\beta_i} \,
                  \prod_{i = 1}^{n}
                    \prod_{j = 1}^{r}
                      q_{i j} !
                } \,
                \partial^{\alpha} f
                  \paren{\bm{x} \paren{t}} \,
                \prod_{i = 1}^{n}
                  \prod_{j = 1}^{r}
                    \paren{\deriv{x}{i}_j \paren{t}}^{q_{i j}},
          \label{eq:Bruno_formula.vector}
        \end{equation}
        where
        \begin{equation}
          \mathcal{Q}_{\beta}
          \coloneqq
            \cparen{%
              \paren{q_{i j}}_{1 \leq i \leq n, \, 1 \leq j \leq r} \in \Matrices{n \times r}{\nnZ}
              \mvert
              \text{%
                $\sum_{j = 1}^{r} q_{i j} = \beta_i$
                for each $i = 1, \dots, n$
              }
            }
          \label{eq:Bruno_formula.vec.sum}
        \end{equation}
        and the $r$-dimensional multi-index
        $\alpha = \paren{\alpha_1, \dots, \alpha_r}
          \in \paren{\nnZ}^r$
        is given by
        \begin{equation}
          \alpha_j \coloneqq \sum_{i = 1}^{n} q_{i j}
            \quad (j = 1, 2, \dots, r).
          \label{eq:Bruno_formula.vec.notation}
        \end{equation}
    \end{fact}

    \begin{rem}
        Note that
        \begin{equation}
          \abs{\alpha}
          = \sum_{j = 1}^{r} \alpha_j
          = \sum_{i = 1}^{n} \sum_{j = 1}^{r} q_{i j}
          = \sum_{i = 1}^{n} \beta_i
          = \abs{\beta}.
        \end{equation}
        \cref{tab:qij_alpha_and_beta} summarizes the relationships among $(q_{i j})$, $\alpha$, and $\beta$ involved in \cref{fact:Bruno.extension}.
        \begin{table}[htbp]
            \centering
            \begin{tabular}{c||ccc|c}
                  & $1$ & $\cdots$ & $r$ & sum \\ \hline \hline
                 $1$ & $q_{1 1}$ & $\cdots$ & $q_{1 r}$ & $\beta_{1}$ \\
                 $\vdots$ & $\vdots$ & $\ddots$ & $\vdots$ & $\vdots$ \\
                 $n$ & $q_{n 1}$ & $\cdots$ & $q_{n r}$ & $\beta_{n}$ \\ \hline
                 sum & $\alpha_{1}$ & $\cdots$ & $\alpha_{r}$ & $\abs{\alpha} = \abs{\beta}$
            \end{tabular}
            \caption{Relation of $(q_{i j})$, $\alpha$, and $\beta$}
            \label{tab:qij_alpha_and_beta}%
        \end{table}
    \end{rem}

    \begin{cor}
      \label{cor:Bruno_for_curves}%
      Let $\gamma \colon \paren{\R, 0} \to \paren{\R^N, \bzero}$
      be a curve in $\R^N$.
      Then, we get the following:
      \begin{enumerate}[label={(\alph*)}, leftmargin = *]
        \item \label{item:cor_Bruno_curve.para}%
          If $\phi \colon \paren{\R, 0} \to \paren{\R, 0}$
          is a parameter change, then the curve
          $\tilde{\gamma} \coloneqq \gamma \circ \phi
            \colon \paren{\R, 0} \to \paren{\R^N, \bzero}$
          satisfies
          \begin{equation}
            \sderiv{\tilde{\gamma}}{n}
            = \sum_{d = 1}^{n}
                \sum_{\substack{%
                    \beta \in \parts{n} \\
                    \abs{\beta} = d
                  }}
                  \frac{d!}{%
                    \beta!
                  } \, 
                  \sderiv{\gamma}{d}
                  \prod_{i = 1}^{n}
                    \left(%
                      \sderiv{\phi}{i}
                    \right)^{\beta_i}.
            \label{eq:Bruno_curve.para}
          \end{equation}
        \item \label{item:cor_Bruno_curve.coord}%
          Let $\gamma$ be expressed as
          $\gamma \paren{t} = \paren{x_1 \paren{t}, \dots, x_N \paren{t}}$.
          If $\varPhi \colon \paren{\R^N, \bm{0}} \to \paren{\R^N, \bm{0}}$
          is a coordinate transformation,
          then the curve
          $\hat{\gamma} \coloneqq \varPhi \circ \gamma
            \colon \paren{\R, 0} \to \paren{\R^N, \bzero}$
          satisfies
          \begin{equation}
            \sderiv{\hat{\gamma}}{n}
            = \sum_{\beta \in \parts{n}}
                \sum_{\paren{q_{i j}} \in \mathcal{Q}_{\beta}}
                  \frac{1}{%
                    \displaystyle
                    \prod_{i = 1}^{n}
                      \prod_{j = 1}^{N}
                        q_{i j}!
                  } \,
                  \partial^{\alpha} \varPhi \paren{\bzero} \,
                  \prod_{i = 1}^{n}
                    \prod_{j = 1}^{N}
                      \left(\sderiv{x}{i}_j\right)^{q_{i j}},
            \label{eq:Bruno_curve.coord}
          \end{equation}
          under the notation
          \cref{eq:Bruno_formula.vec.sum,eq:Bruno_formula.vec.notation}
          with $r$ replaced by $N$.
      \end{enumerate}
    \end{cor}

    \begin{proof}
      \begin{enumerate}[label={(\alph*)}]
        \item Substituting $s = 0$
          and $\sfrac{\deriv{\phi}{i} \paren{0}}{i !}
            = \sderiv{\phi}{i}$
          into \cref{eq:Bruno_formula.scalar},
          we get the conclusion.
        \item Substituting $r = N$, $t = 0$, and
          $\deriv{x}{i}_j \paren{0}
            = i! \, \sderiv{x}{i}_j$
          into \cref{eq:Bruno_formula.vector}
          and noting 
          \begin{equation}
            \prod_{i = 1}^{n}
              \prod_{j = 1}^{N}
                \paren{i!}^{q_{i j}}
            = \prod_{i = 1}^{n}
              \paren{i!}^{\sum_{j = 1}^{N} q_{i j}}
            = \prod_{i = 1}^{n}
              \paren{i!}^{\beta_i},
          \end{equation}
          we get the claim.
      \end{enumerate}
    \end{proof}

  \subsection{Multiplicity of a Curve}
  \label{ssec:multiplicity}%

    \begin{defn}[\rcite{Section~1}{Fukui_2017_multiplicities}]
        \label{defn:multiplicity}%
        A curve $\gamma \colon \paren{\bbR, 0}
            \to \paren{\bbR^N, \bzero}$
                 is said to be \emph{of multiplicity $m \in \poZ$}
                  at $t = 0$
                if there exists a curve
                $\overline{\gamma}$ in $\R^N$
                with $\overline{\gamma} \paren{0} \neq \bzero$
                such that $\gamma$ is expressed around $t = 0$ as
                \begin{equation}
                    \gamma \paren{t}
                    = \frac{t^m}{m} \,
                        \overline{\gamma} \paren{t}.
                    \label{eq:multiplicity.defn}
                \end{equation}
    \end{defn}

    \begin{prop}
      \label{prop:multi_and_lin_dep.invariant}%
      \begin{enumerate}[label={(\arabic*)}]
        \item \label{item:prop_inv.multi}%
          The multiplicity of a curve in $\R^N$
          is invariant under \Aqui-equivalence.
        \item \label{item:prop_inv.lin_dep}%
          Let $\gamma \colon \paren{\bbR, 0}
            \to \paren{\bbR^N, \bzero}$
          be of multiplicity $m \in \geZ{2}$.
          Then, for each $n \in \geZ{3}$
          with $m + 1 \leq n \leq 2 m - 1$,
          the property that%
          \begin{equation} \label{eq:inv.span}
            \sderiv{\gamma}{n}
            \in \Span{\sderiv{\gamma}{m}, \dots, \sderiv{\gamma}{n - 1}}
          \end{equation}
          is invariant under \Aqui-equivalence.
      \end{enumerate}
    \end{prop}

    \begin{rem}
      \label{rem:coord_transf.multi_m}%
      For a curve $\gamma$ of multiplicity $m \in \geZ{2}$ and an integer $n \ge m$, when we consider \cref{eq:Bruno_curve.coord}, it suffices to take the sum over indices $\paren{q_{ij}} \in \mathcal{Q}_{\beta}$ such that
      \begin{equation} 
        q_{i j} = 0
        \quatext{for all}
        i = 1, \dots, m - 1
        \quatext{and all}
        j = 1, \dots, N,
      \end{equation}
      since
      $\prod_{i = 1}^{n}
        \prod_{j = 1}^{N}
          \paren{\sderiv{x}{i}_j \paren{0}}^{q_{i j}}$
      vanishes for all other $\paren{q_{ij}}$. If there exists such a $(q_{ij}) \in \mathcal{Q}_{\beta}$, then $\beta$ must satisfy 
      \begin{equation}
        \beta_i
        = \sum_{j = 1}^{N} q_{i j}
        = 0
        \quatext{for all}
        i = 1, \dots, m - 1,
      \end{equation}
      that is, $\beta = (0, \dots, 0, \beta_m, \dots, \beta_n)$.
    \end{rem}

    \begin{proof}[Proof of \cref{prop:multi_and_lin_dep.invariant}]
      We consider (a) parameter changes and (b) coordinate transformations separately.
      Assume that $\gamma \colon \paren{\R, 0} \to \paren{\R^N, \bzero}$
      is of multiplicity $m \in \geZ{2}$.

      (a) Let us consider a parameter change $\phi$
          and the curve $\tilde{\gamma} \coloneqq \gamma \circ \phi$.
          By \cref{cor:Bruno_for_curves} \ref{item:cor_Bruno_curve.para},
          we can compute $\sderiv{\tilde{\gamma}}{k}$
          ($1 \leq k \leq 2 m - 1$) as follows:
          \begin{enumerate}[label={(\Roman*)}]
          \item 
          When $1 \leq k \leq m - 1$, we get $\sderiv{\tilde{\gamma}}{k} = \bzero$, since $\sderiv{\gamma}{d} = \bzero$ for all $d = 1, \dots, k$.
          \item \label{item:diff.of.gamm.tilde}
          When $m \leq k \leq 2 m - 1$, since the only partition $\beta$ of $k \in \poZ$ satisfying $\abs{\beta} = k$ is $\beta = \paren{k, 0, \dots, 0}$, \cref{eq:Bruno_curve.para} implies 
          $\sderiv{\tilde{\gamma}}{k}
            \in \Span{\sderiv{\gamma}{m}, \dots, \sderiv{\gamma}{k - 1}}
            + \paren{\sderiv{\phi}{1}}^{k} \, \sderiv{\gamma}{k}$.
          \end{enumerate}
          Thus, we have \ref{item:prop_inv.multi} for parameter changes.
          To show \ref{item:prop_inv.lin_dep}, we fix $n \in \{m+1, m+2, \dots, 2m-1\}$.
          The above calculation in \ref{item:diff.of.gamm.tilde} gives $\tilde{\gamma}^{[n]} \in \Span{\gamma^{[m]}, ..., \gamma^{[n]}} $ 
          and 
          $\Span{\tilde{\gamma}^{[m]}, ..., \tilde{\gamma}^{[n-1]}} = \Span{\gamma^{[m]}, ..., \gamma^{[n-1]}}$. 
          If \cref{eq:inv.span} holds, then
          $\Span{\gamma^{[m]}, ..., \gamma^{[n]}} \allowbreak = \Span{\gamma^{[m]}, ..., \gamma^{[n-1]}}$.
          They ensure 
          $\tilde{\gamma}^{[n]} \in \Span{\tilde{\gamma}^{[m]}, ..., \tilde{\gamma}^{[n-1]}}$.

          (b) Let $\gamma$ be expressed as
          $\gamma \paren{t} = \paren{x_1 \paren{t}, \dots, x_N \paren{t}}$,
          and let us consider a coordinate transformation $\varPhi$
          and the curve $\hat{\gamma} \coloneqq \varPhi \circ \gamma$.
          By \cref{cor:Bruno_for_curves} \ref{item:cor_Bruno_curve.coord},
          we can compute $\sderiv{\hat{\gamma}}{k}$
          ($1 \leq k \leq 2 m - 1$) as follows:
          \begin{enumerate}[label={(\Roman*)'}]
            \item 
              When $1 \leq k \leq m - 1$, we get $\sderiv{\hat{\gamma}}{k} = \bzero$,
              since $\sderiv{x}{i}_j = 0$ for all $i = 1, \dots, k$ and all $j = 1, \dots, N$.
            \item \label{item:diff.gamm.hat}
              When $m \leq k \leq 2 m - 1$, from \cref{rem:coord_transf.multi_m}, it suffices to take the sum over $\beta \in \mathcal{P}_k$ satisfying
              $2 m > k
                = \sum_{i = m}^{k} i \, \beta_i
                \geq m \, \abs{\beta}$.
              Hence, $\abs{\beta} = 1$, and thus, $\beta = \paren{0, \dots, 0, 1}$.
              Therefore, for each $\paren{q_{i j}} \in \mathcal{Q}_{(0, ..., 0, 1)}$
              there exists $j_0 \in \cparen{1, \dots, N}$
              such that $q_{ij} = \delta_{ik} \, \delta_{jj_0}$ (see \cref{tab:qij_alpha_and_beta}), where $\delta_{i j}$ denotes the Kronecker delta. Consequently, from 
              $\alpha_j = \sum_{i = 1}^{k} q_{i j} = \delta_{j j_0}$ and \cref{eq:Bruno_curve.coord}, we get 
              $\sderiv{\hat{\gamma}}{k} = J_{\varPhi} \paren{\bzero} \, \sderiv{\gamma}{k}$.
          \end{enumerate}
          Hence, we have obtained \ref{item:prop_inv.multi} for coordinate transformations.
          Moreover, for each $m + 1 \leq n \leq 2 m - 1$, if \cref{eq:inv.span} holds, 
          then \ref{item:diff.gamm.hat} implies $\sderiv{\hat{\gamma}}{n}
            \in \Span{\sderiv{\hat{\gamma}}{m}, \dots, \sderiv{\hat{\gamma}}{n - 1}}$.
    \end{proof}

\subsection{Representability of Non-Negative Integers by a Family of Positive Integers}
\label{sec:Frobenius.numbers}
    In the construction of the criteria (specifically in \cref{lem:suff.cond.1:cNR.containment}), 
    we employ notions from elementary number theory.
    It is said that
    ~\cite{Sylvester_1857_partition_of_numbers,Sylvester_1882_subvariants,Sylvester_1884_mathematical_questions}
    contain information about the Frobenius problem.
    In this subsection, 
    let $k \in \geZ{2}$
    and let
    $A \coloneqq \cparen{a_1, \dots, a_k}
      \subset \geZ{2}$.

    We extend the definition of
    the denumerant $\denu{}{a; A}$ given by%
    ~\cite{Sylvester_1857_partition_of_numbers,Liu_Xin_2023_Frobenius_numbers}
    to 
    $\denu{\ge n}{a; A}$,
    with a particular focus on $\denu{\ge 1}{a; A}$ and  $\denu{\ge 2}{a; A}$.

    \begin{defn}
      \label{defn:denumerant}
      Let $n, \, a \in \nnZ$. We define 
      \begin{equation}
        \mathemph{\denu{\geq n}{a; A}}
        \coloneqq
          \# \cparen{%
                \paren{x_1, \dots, x_k}
                \in \paren{\nnZ}^k
              \mvert
                \sum_{i = 1}^{k}
                  a_i \, x_i
                = a
                \quatext{and}
                \sum_{i = 1}^{k}
                  x_i
                \geq n
            },
        \label{eq:defn.denumerant}
      \end{equation}
      where $\#$ denotes the cardinality of a set.

      Let us denote
      \begin{align}
        \mathemph{\cR{n}{A}}
        & \coloneqq
          \cparen{%
              a \in \nnZ
            \mvert
              \denu{\geq n}{a; A}
              > 0
          }
        , \\
        \mathemph{\cNR{n}{A}}
        & \coloneqq
          \nnZ \setminus \cR{n}{A}
        = \cparen{%
              a \in \nnZ
            \mvert
              \denu{\geq n}{a; A}
              = 0
          }
        ,
      \end{align}
      and simply
      $\mathemph{\cR{}{A}}
        \coloneqq \cR{0}{A}$
      and $\mathemph{\cNR{}{A}}
        \coloneqq \cNR{0}{A}$.
      A non-negative integer $a \in \nnZ$
      is said to be \emph{representable} by $A$
      if $a \in \cR{0}{A}$,
      and \emph{not representable} by $A$
      if $a \in \cNR{0}{A}$.
      We also denote $\cR{n}{A}$ and $\cNR{n}{A}$ as $\cR{n}{a_1, \dots, a_k}$ and $\cNR{n}{a_1, \dots, a_k}$, respectively.
    \end{defn}

    Here we introduce some properties of $\cR{n}{A}$ and $\cNR{n}{A}$.
    \begin{prop}
      \label{prop:cR.and.cNR.containment}
      \begin{enumerate}[label={(\arabic*)}]
        \item \label{item:rem.cR.and.cNR:addition}%
          For any $m, n \in \nnZ$,
          $\cR{m + n}{A} = \cR{m}{A} + \cR{n}{A}$.
        \item \label{item:rem.cR.and.cNR:cRn+1.and.cRn}%
          For any $n \in \nnZ$,
          $\cR{n + 1}{A}
            \subset \cR{n}{A}$,
          and thus,
          $\cNR{n}{A}
            \subset \cNR{n + 1}{A}$.
          Moreover, for any $n \in \nnZ$,
          $\cNR{n + 1}{A} \setminus \cNR{n}{A}
            \subset \nnZ$
          is a finite set.
        \item \label{item:rem.cR.and.cNR:cNRn.finite}%
          For any $m, n \in \nnZ$,
          $\cNR{m}{A}$ is a finite set
          if and only if so is $\cNR{n}{A}$.
      \end{enumerate}
    \end{prop}

    \begin{proof}
        (1) For any
          $a \in \cR{m}{A}$ and $b \in \cR{n}{A}$,
          there exist 
          $\paren{x_1, \dots, x_k}, \paren{y_1, \dots, y_k} \in \paren{\Z_{\ge 0}}^k$
          such that $\sum a_i \, x_i = a$, $\sum a_i \, y_i = b$, $\sum x_i \geq m$, and $\sum y_i \geq n$.
          Then,
          $\paren{x_1 + y_1, \dots, x_k + y_k}$ gives $a + b \in \cR{m + n}{A}$.
          Conversely, for any $a \in \cR{m + n}{A}$, there exists
          $\paren{x_1, \dots, x_k} \in \paren{\Z_{\ge 0}}^k$ such that
          $\sum a_i \, x_i = a$ and $\sum x_i \geq m + n$.
          Then, we can take $\paren{y_1, \dots, y_k} \in \paren{\Z_{\ge 0}}^k$ satisfying 
          $0 \leq y_i \leq x_i$ and $\sum y_i = m$.
          Since $\sum a_i y_i \in \cR{m}{A}$
          and $\sum a_i (x_i - y_i) \in \cR{n}{A}$ hold, 
          we get $a \in \cR{m}{A} + \cR{m}{A}$.
          
        (2) The former part is clear from \cref{defn:denumerant}.
          Moreover, since
          $\cNR{n + 1}{A} \setminus \cNR{n}{A} = \cNR{n + 1}{A} \cap \cR{n}{A}$,
          $a \in \cNR{n + 1}{A} \setminus \cNR{n}{A}$
          if and only if there exists
          $\paren{x_1, \dots, x_k} \in \paren{\Z_{\ge 0}}^k$
          such that
          $\sum a_i \, x_i = a$
          and
          $\sum x_i = n$.
          Hence, 
          $\cNR{n + 1}{A} \setminus \cNR{n}{A}$ equals 
          the finite set $\cparen{\sum a_i \, x_i \mvert \sum x_i = n}$.

        (3) When $m = n$, the assertion is clear.
          When $m \neq n$, by symmetry, we may assume $m < n$. If $\cNR{n}{A}$ is finite, then $\# \cNR{m}{A} \le \# \cNR{n}{A} < \infty$ from the former part of \ref{item:rem.cR.and.cNR:cRn+1.and.cRn}. 
          Assume $\cNR{m}{A}$ is finite. 
          By \ref{item:rem.cR.and.cNR:cRn+1.and.cRn}, we have that
          \begin{equation}
            \cNR{n}{A} 
            = \cNR{m}{A} \cup
                \bigcup_{i = m}^{n - 1}
                  \paren{\cNR{i + 1}{A} \setminus \cNR{i}{A}}
          \end{equation}
          is a finite set.
    \end{proof}

    \begin{fact}[\nscite{Liu_Xin_2023_Frobenius_numbers,Subwattanachai_2024_generalized_Frobenius_number}]
      \label{fact:NR.finite.if.gcd1}%
      The set $\cNR{}{A}$ is finite if and only if $\gcd A = 1$.
    \end{fact}

    \begin{rem}
      \label{rem:NRn.finite.if.gcd1}%
      It follows
      from \cref{prop:cR.and.cNR.containment} \ref{item:rem.cR.and.cNR:cNRn.finite}
      and \cref{fact:NR.finite.if.gcd1}
      that for any $n \in \nnZ$,
      $\cNR{n}{A}$ is a finite set
      if and only if $\gcd A = 1$.
    \end{rem}

\section{Sufficient Conditions for a Singularity to Belong to a Given \Apdf-Equivalence Class}
\label{ssec:RN.construct_suff}

    In order to construct sufficient conditions for singularities of curves in $\R^N$ to belong to a given \A-equivalence class, we use the following facts:

    \begin{fact}[The division lemma, see \rscite{Corollary~A.3}{Saji_Umehara_Yamada_2022_singularities}]
        \label{fact:division.lemma}%
        Let $f$ be a $C^{\infty}$-function defined on an open interval $I \subset \R$ containing the origin $0$. If $f(0) = f'(0) = f''(0) = \cdots = f^{(k)}(0) = 0$ for a non-negative integer $k$, then there exists a $C^{\infty}$-function $g$ defined on $I$ such that 
        \begin{equation}
            f(t) = t^{k+1} \, g(t) \quad (t \in I)
        \end{equation}
    \end{fact}
  
    The following fact is induced from Whitney's lemma (\!\!\cite{Whitney_1943_Diff_even_fcn}).

    \begin{fact}[\rcite{Lemma~4.8}{Matsushita_2024_classifications_4_5_cusps_arXiv}]
        \label{lem:Matsushita.4.8}%
        Let $k \in \bbZ_{> 0}$.
        For any $C^{\infty}$-function germ
        $f \colon \paren{\R, 0}
            \to \R$,
        there exists a $2^k$-tuple of
        $C^{\infty}$-function germs
        $g_i \colon \paren{\R, 0}
            \to \R$
        $(i \in \{0, 1, 2, 3, \dots, 2^k - 1\})$
        such that
        \begin{equation}
            f \paren{t}
            = \sum_{i = 0}^{2^k - 1}
                t^i \, g_i \paren{t^{2^k}}
            = g_0 \paren{t^{2^k}}
                + t \, g_1 \paren{t^{2^k}}
                + \dots
                + t^{2^k - 1} \, g_{2^k - 1} \paren{t^{2^k}}.
            \label{eq:Matsushita.lem.4.8}
        \end{equation}
    \end{fact}

    \begin{lem}
        \label{lem:Matsushita.4.8:derivative}
        Under the decomposition \cref{eq:Matsushita.lem.4.8},
        for any $\lambda, \mu \in \Z_{\geq 0}$
        with $0 \leq \lambda \leq 2^k - 1$,
        \begin{equation}
            \sderiv{f}{\lambda + 2^k \mu}
            = \sderiv{g}{\mu}_{\lambda}.
            \label{eq:Matsushita.lem.4.8:derivative}
        \end{equation}
    \end{lem}

    \begin{proof}
        Take $m \in \Z_{\geq \mu + 1}$.
        If $g_i$ has the form
        \begin{equation}
            g_i \paren{s}
            = \sum_{j = 0}^{m} a_{i j} \, s^j + o \paren{s^{m}}
        \end{equation}
        for each $i \in \cparen{0, 1, 2, 3, \dots, 2^k - 1}$,
        then, \cref{eq:Matsushita.lem.4.8} gives
        \begin{align}
            f \paren{t} =
            \sum_{i = 0}^{2^k - 1} 
            \sum_{j = 0}^{m} a_{i j} \, t^{i + 2^k j} + o \paren{t^{2^k m}}.
        \end{align}
        Hence, we have
        $\sderiv{f}{\lambda + 2^k \mu} = a_{\lambda \mu} = \sderiv{g}{\mu}_{\lambda}$.
    \end{proof}

    \begin{defn}
      \label{defn:space.F.of.functions}%
      Let 
      $\gamma = \paren{x_1, \dots, x_N} \colon \paren{\R, 0} \to \paren{\R^N, \bzero}$
      be a curve in $\R^N$.
      Then, we define the sets
      \begin{align}
        \mathemph{\Fcns{1}{i}{\gamma}}
        & \coloneqq
          \cparen{
              f \colon \paren{\R^N, \bzero}
                \to \paren{\R, 0}
            \mvert
              \paren{f \circ \gamma} \paren{t} = t^i
          }, \\
        \mathemph{\Fcns{2}{i}{\gamma}}
        & \coloneqq
          \cparen{%
              f \in \Fcns{1}{i}{\gamma}
            \mvert
              J_f \paren{\bzero}
              = O
              \in M_{1 \times N} \paren{\R}
          }, \\
        \mathemph{\cR{n}{\gamma}}
        & \coloneqq
          \cparen{%
              i \in \poZ
            \mvert
              \Fcns{n}{i}{\gamma}
              \neq \emptyset
          }
        \subset \poZ, \\
        \mathemph{\cNR{n}{\gamma}}
        & \coloneqq
          \poZ \setminus \cR{n}{\gamma}
        = \cparen{%
              i \in \poZ
            \mvert
              \Fcns{n}{i}{\gamma}
              = \emptyset
          }
        \subset \poZ
      \end{align}
      for $n= 1, 2$.
    \end{defn}

    The following lemma is an analogue of \cref{prop:cR.and.cNR.containment}
    \ref{item:rem.cR.and.cNR:addition}
    and
    \ref{item:rem.cR.and.cNR:cRn+1.and.cRn} for curves.

    \begin{lem}
    \label{prop:cR.and.cNR.of.curve}%
      It holds that $\cR{2}{\gamma} \subset \cR{1}{\gamma}$, and thus 
      $\cNR{1}{\gamma} \subset \cNR{2}{\gamma}$.
      Moreover, $\cR{1}{\gamma} + \cR{1}{\gamma} \subset \cR{2}{\gamma}$.
    \end{lem}
    
    \begin{proof}
      The former part is immediate from
      $\Fcns{2}{i}{\gamma} \subset \Fcns{1}{i}{\gamma}$.
      If $i, j \in \cR{1}{\gamma}$,
      then we can take $f_i \in \Fcns{1}{i}{\gamma}$ and $f_j \in \Fcns{1}{j}{\gamma}$.
      Thus, $((f_i \, f_j) \circ \gamma) (t) = t^{i + j}$
      holds with
      $f_i \, f_j \colon \paren{\R^N, \bzero} \to \paren{\R, 0}$
      satisfying $J_{f_i \, f_j} = f_j \, J_{f_i} + f_i \, J_{f_j}$.
      Since
      $f_i \paren{\bzero} = f_j \paren{\bzero} = 0$,
      we get $J_{f_i \, f_j} \paren{\bzero} = O$,
      and hence, $i + j \in \cR{2}{\gamma}$.
    \end{proof}

    We examine the similarity between
    $\cNR{n}{A}$ and $\cNR{n}{\varGamma}$ as follows.

    \begin{prop}
      \label{lem:suff.cond.1:cNR.containment}
      Let $k \in \geZ{2}$ satisfy $2 \leq k \leq N$,
      let
      $A \coloneqq \{a_1, \dots, a_k \} \subset \geZ{2}$,
      and let
      $\varGamma \colon \paren{\R, 0}
        \to \paren{\R^N, \bzero}$
      be a curve of the form%
      \begin{equation}
        \varGamma \paren{t}
        = \paren{X_1 \paren{t}, \dots, X_N \paren{t}}
        = \paren{%
          t^{a_1},
          t^{a_2},
          \dots,
          t^{a_k},
          o \paren{t^{a_k}},
          \dots,
          o \paren{t^{a_k}}
        }
        \in \R^N.
        \label{eq:suff.cond.1:curve.Gamma}
      \end{equation}
      Then, for each $n \in \cparen{1, 2}$,
      the curve $\varGamma$ satisfies
      \begin{equation}
        \cR{n}{A}
        \subset \cR{n}{\varGamma},
        \quatext{and thus,}
        \cNR{n}{\varGamma}
        \subset \cNR{n}{A}.
      \end{equation}
    \end{prop}

    \begin{proof}
      We only prove
      $\cR{n}{A} \subset \cR{n}{\varGamma}$.
      The remaining part follows directly by taking the complement.
      First, let us consider the $n = 1$ case.
          For each $a \in \cR{1}{A}$,
          there exists
          $\paren{c_1, \dots, c_k} \in \paren{\nnZ}^k$
          such that
          $\sum a_i \, c_i = a$ and $\sum c_i \geq 1$
          by \cref{defn:denumerant}.
          Thus, we have
          $(f \circ \varGamma) (t) = t^a$, 
          where
          $f \paren{x_1, \dots, x_N}
            \coloneqq \prod_{i = 1}^{k}
                x_{i}^{c_i}$.
          Hence, we get $f \in \Fcns{1}{a}{\varGamma}$,
          and thus, $a \in \cR{1}{\varGamma}$.
        Next, let us consider the $n = 2$ case.
          Applying 
          \cref{prop:cR.and.cNR.containment}
          \ref{item:rem.cR.and.cNR:addition},
          $\cR{1}{A} \subset \cR{1}{\varGamma}$,
          and
          \cref{prop:cR.and.cNR.of.curve},
          we get
          $\cR{2}{A}
            = \cR{1}{A} + \cR{1}{A}
            \subset \cR{1}{\varGamma} + \cR{1}{\varGamma}
            \subset \cR{2}{\varGamma}$.
    \end{proof}
    
    Now, we construct sufficient conditions for singularities of curves in $\R^N$ to belong to a given \A-equivalence class based on the set $\cNR{2}{\gamma}$.

    \begin{thm}
      \label{thm:general.property.A.equivalence}%
      Let $\varGamma \colon \paren{\R, 0}
        \to \paren{\R^N, \bzero}$
      be a curve in $\R^N$
      with $\cNR{2}{\varGamma}$ finite.
      %
      Then, for any curve
      $M \colon \paren{\R, 0}
        \to \paren{\R^N, \bzero}$
      satisfying
      \begin{equation}
        \sderiv{M}{i} = \bzero
        \quatext{for all}
        i \in \cNR{2}{\varGamma},
        \label{eq:general.property:assumption.1}
      \end{equation}
      the curve $\varGamma + M$ is \cL-equivalent to $\varGamma$.
      In particular, $\varGamma + M$ is \A-equivalent to $\varGamma$.
    \end{thm}

    \begin{proof}
      The assumption and \cref{prop:cR.and.cNR.of.curve}
      give that $\cNR{2}{\varGamma}$ and $\cNR{1}{\varGamma}$ are finite.
      Thus, we set
      \begin{equation}
        l
        \coloneqq
          \min \cparen{\lambda \in \poZ
              \mvert
                2^{\lambda} \in \cR{1}{\varGamma}
            }
      \end{equation}
      and take $f \in \Fcns{1}{2^l}{\varGamma}$.
      Also, for each $r \in \cparen{0, 1, \dots, 2^l - 1}$,
      we set
      \begin{equation}
        j_r
        \coloneqq
          \min \cparen{j \in \nnZ
              \mvert
                r + 2^l j \in \cR{2}{\varGamma}
            }
      \end{equation}
      and take $f_r \in \Fcns{2}{r + 2^l j}{\varGamma}$.
      By \cref{lem:Matsushita.4.8},
      there exist curves
      $M_r \colon \paren{\R, 0}
        \to \paren{\R^N, \bzero}$
      ($r \in \cparen{0, 1, \dots, 2^l - 1}$)
      such that
      \begin{equation}
          M \paren{t}
          = \sum_{r = 0}^{2^l - 1}
            t^{r} \, M_{r} \paren{t^{2^l}}.
          \label{eq:pf:general.property6:eq.1}
      \end{equation}
      Fix $r \in \{0, 1, \dots, 2^l-1\}$.
      Then, by the assumption
      \cref{eq:general.property:assumption.1}
      and \cref{lem:Matsushita.4.8:derivative},
      we have
      $\sderiv{M}{j}_{r} = \sderiv{M}{r + 2^l j} = \bm{0}$ for all
      $j \in \cparen{0, \dots, j_r - 1}$,
      since $r + 2^l j \in \cNR{2}{\varGamma}$.
      Thus, by \cref{fact:division.lemma},
      there exist curves
      $\tilde{M}_{r} \colon \paren{\R, 0} \to \paren{\R^N, \bzero}$
      such that
      $M_r \paren{s} = s^{j_r} \, \tilde{M}_{r} \paren{s}$.
      Substituting these results into
      \cref{eq:pf:general.property6:eq.1},
      we obtain
      \begin{equation}
        M \paren{t}
        = \sum_{r = 0}^{2^l - 1}
          t^{r + 2^l j_r} \,
          \tilde{M}_{r} \paren{t^{2^l}}.
      \end{equation}
      Hence, the $C^{\infty}$-map
      $\psi \colon \paren{\R^N, \bzero} \to \R^N$
      defined by
      \begin{equation}
        \psi \paren{\bm{x}}
        \coloneqq
          \sum_{r = 0}^{2^l - 1}
            f_r \paren{\bm{x}} \,
            \tilde{M}_r \paren{f \paren{\bm{x}}}
      \end{equation}
      satisfies
      \begin{equation}
        M (t)
        = \sum_{r = 0}^{2^l - 1}
          t^{r + 2^l j_r} \, \tilde{M}_{r} (t^{2^l})
        = \sum_{r = 0}^{2^l - 1}
          (f_r \circ \varGamma) (t)
          \cdot (\tilde{M}_{r} \circ f \circ \varGamma) (t)
        = (\psi \circ \varGamma) (t).
      \end{equation}
      Moreover, since $f_r\paren{\bzero} = 0$
      for each $r$, we get
      $\psi \paren{\bzero} = \bzero$,
      and $J_{f_r} \paren{\bzero} = O$ also implies
      \begin{align}
        J_{\psi} \paren{\bzero}
         = \sum_{r = 0}^{2^l - 1}
            \paren{%
              \paren{\tilde{M}_r \circ f} \paren{\bzero} \,
                J_{f_r} \paren{\bzero}
              + f_r \paren{\bzero} \,
                J_{M_r \circ f} \paren{\bzero}
            }
        = O.
      \end{align}
      Now, the $C^{\infty}$-map
      $\varPsi \coloneqq \id[\R^N] + \psi 
        \colon \paren{\R^N, \bzero} \to \paren{\R^N, \bzero}$
      satisfies
      $\varPsi \circ \varGamma
        = \varGamma + \paren{\psi \circ \varGamma}
        = \varGamma + M$
      and $J_{\varPsi} \paren{\bzero}$ is the identity matrix.
      Therefore,
      $\varPsi \colon \paren{\R^N, \bzero}
        \to \paren{\R^N, \bzero}$
      is a coordinate transformation,
      and thus, $\varGamma + M = \varPsi \circ \varGamma$
      is $\mathcal{L}$-equivalent to $\varGamma$.
    \end{proof}

    When $\varGamma$ takes a special form,
    \cref{thm:general.property.A.equivalence}
    can be applied in a specific way as follows.

    \begin{cor}
      \label{thm:A.equivalence.suff.cond.1}%
      Let $k \in \geZ{2}$ and $A$ be as in \cref{lem:suff.cond.1:cNR.containment} and suppose that $\gcd A = 1$.
      Let a curve $\varGamma$ be as in \cref{eq:suff.cond.1:curve.Gamma}.
      Then, for any curve
      $M \colon \paren{\R, 0}
        \to \paren{\R^N, \bzero}$
      satisfying
      \begin{equation}
        \sderiv{M}{i} = \bzero
        \quatext{for all}
        i \in \cNR{2}{A},
        \label{eq:suff.cond.1:assumption.1}
      \end{equation}
      the curve $\varGamma + M$ is \cL-equivalent to $\varGamma$.
      In particular, $\varGamma + M$ is \A-equivalent to $\varGamma$.
    \end{cor}

    \begin{proof}
      Note that
      $\gcd A = 1$
      implies $\# \cNR{2}{A} < \infty$
      as stated in \cref{rem:NRn.finite.if.gcd1}.
      Furthermore, by \cref{lem:suff.cond.1:cNR.containment},
      $\cNR{2}{\varGamma}$ is finite.
      Therefore, \cref{thm:general.property.A.equivalence} gives the conclusion.
    \end{proof}

\section{Examples of the Construction of Criteria for Singularities}
\label{ssec:RN.examples}

  In this section, we provide examples of
  the construction of criteria for several \A-equivalence classes of singularities, following the procedure outlined in \cref{sec:Introduction}.   
  As the criteria in each subsection demonstrate,
  $\paren{3, 4, 5}$-cusps and $\paren{3, 5, 7}$-cusps in \cref{conj:criteria.3.4.5.and.3.4.and.3.5.7.and.3.5}, and $\paren{4, 5, 7}$-cusps, $\paren{4, 5, 6}$-cusps, $\paren{4, 5, 11}$-cusps, and $\paren{4, 5_{\pm 7}, 11}$-cusps in \cref{conj:criteria.4.5.and.other} only appear when $N \geq 3$.
  Also, $\paren{4, 5, 6, 7}$-cusps in \cref{conj:criteria.4.5.and.other} only appear when $N \geq 4$.

    In the construction of sufficient conditions for a singularity to belong to a given \A-equivalence class,
    we use the following elementary lemma.
    In the sequel,
      let $\paren{\be_i}_{1 \leq i \leq N}$
      be the standard basis of $\R^N$.
    \begin{lem}
      \label{lem:lin_change_to_standard_basis}%
      Let $1 \leq k \leq N$
      and let $(\bv_1, \dots, \bv_k) \in (\R^N)^k$
      be a linearly independent
      tuple of vectors.
      Then, there exists
      $T \in \GL{N, \R}$
      such that
      $
        T \bv_i
        = \be_i
        \, 
        (1 \leq i \leq k).
      $
    
    \end{lem}

  \subsection{Singularities of Multiplicity $2$ in \Rpdf[N]}
  \label{sec:Singularities.multiplicity.2}%

    In this subsection, we construct criteria for classifying singularities of multiplicity $2$ in $\R^N$,
    in particular,
    $\paren{2, 3}$-cusps,
    $\paren{2, 5}$-cusps,
    and $\paren{2, 7}$-cusps in $\R^N$.
    Our goal in this subsection is to prove the following theorem:

    \begin{thm}
      \label{thm:criteria.2.3.and.2.5.and.2.7}%
      Let $\gamma \colon \paren{\R, 0} \to \paren{\R^N, \bzero}$
      be a curve in $\R^N$.
      Then, the curve
      $\gamma$ is \Aqui-equivalent to
         $(1)\, \Cusp{2, 3} ( t ) = \paren{t^2, t^3, 0, \dots, 0}$,
         $(2) \, \Cusp{2, 5} ( t ) = \paren{t^2, t^5, 0, \dots, 0}$,
         $(3) \, \Cusp{2, 7} ( t ) = \paren{t^2, t^7, 0, \dots, 0}$
      if and only if
      $\sderiv{\gamma}{1}
        = \bzero
        \neq \sderiv{\gamma}{2}$
      and respectively,
      \begin{enumerate}[label={(\arabic*)}]
        \item \label{item:criteria.2.3.criteria}%
          $\sderiv{\gamma}{3}
            \notin \Span{\sderiv{\gamma}{2}}$,
        \item \label{item:criteria.2.5.criteria}%
          there exists $\lambda \in \R$ such that
          $
            \sderiv{\gamma}{3}
            = \lambda \, \sderiv{\gamma}{2}
          $
          and
          $
            \sderiv{\gamma}{5}
              - 2 \, \lambda \, \sderiv{\gamma}{4}
            \notin \Span{\sderiv{\gamma}{2}},
          $
        \item \label{item:criteria.2.7.criteria}%
          there exist $\lambda, \mu \in \R$ such that
            $\sderiv{\gamma}{3} = \lambda \, \sderiv{\gamma}{2}$,
            $\sderiv{\gamma}{5}
              - 2 \, \lambda \, \sderiv{\gamma}{4}
            = \mu \, \sderiv{\gamma}{2}$,
            and
            $\sderiv{\gamma}{7}
              - 3 \, \lambda \, \sderiv{\gamma}{6}
              + \paren{3 \, \lambda^3 - 2 \, \mu} \, \sderiv{\gamma}{4}
            \notin \Span{\sderiv{\gamma}{2}}$.
      \end{enumerate}
    \end{thm}

    \begin{proof}
        We prove the claim following the procedure we outlined in \cref{sec:Introduction}.
              First, we prove the ``if'' part of the claim.
              
              \textbf{Step 1}.
              We note that for any $n \in \poZ$,
              \begin{equation} \label{eq:NR.2.2,2n+1}
                \cNR{2}{2, 2 n + 1} = \cparen{0, 2} \cup \cparen{1, 3, \dots, 2 n + 1}.
              \end{equation}

            \textbf{Step 2}.
              Let $\gamma$ be of multiplicity $2$ at $t = 0$, 
              let $T_1$ be the invertible matrix defined in \cref{lem:lin_change_to_standard_basis}. 
              Then, the curve 
              $\varGamma \paren{t} = T_1 \, \gamma \paren{t}$
              satisfies
              $\sderiv{\varGamma}{1} = \bzero$ and
              $\sderiv{\varGamma}{2} = \be_1$.
            Hence, the curve $\varGamma$ takes the form
            \begin{equation}
              \label{eq:multi_2.standard_form}%
              \varGamma \paren{t}
                = \left(%
                      t^2 + \sum_{i = 3}^{7} a_{i - 2} \, t^i,
                      0, \dots, 0
                    \right)
                + \sum_{i = 3}^{7} t^i \, \bv_i
                + \smallo{t^7}
            \end{equation}
            with constants $a_{i - 2} \in \R$
            and $\bv_i \in \Span{\be_2, \dots, \be_N}$
            ($3 \leq i \leq 7$).

              Now, we show that, if
                 (1) $\bv_3 \neq \bzero$,
                 (2) $\bv_3 = \bzero$
                  and $\bv_5 - 2 a_1 \, \bv_4 \neq \bzero$, or
                 (3) $\bv_3 = \bzero$,
                  $\bv_5 - 2 a_1 \, \bv_4 = \bzero$,
                  and $\bv_7 - 3 a_1 \, \bv_6 + \paren{3 a_1^3 + 4 a_1 a_2 - 2 a_3} \, \bv_4 \neq \bzero$,
              then the curve $\varGamma$
              expressed in the form of \cref{eq:multi_2.standard_form}
              is \Aqui-equivalent to
                (1) $\Cusp{2, 3}$, 
                (2) $\Cusp{2, 5}$, or
                (3) $\Cusp{2, 7}$, 
              respectively.

              This can be verified as follows: 
              there exist $c_{i} \in \R$ ($i = 1, 3, 5$)
              such that the parameter change
              $\phi \paren{s}
                = s + \sum_{i = 1, 3, 5} c_{i} \, s^{i + 1}$
              yields
              \begin{equation}
                \paren{\varGamma \circ \phi} \paren{s}
                = \paren{%
                      s^2 + \tilde{a}_2 \, s^4 + \tilde{a}_4 \, s^6,
                      0, \dots, 0
                    }
                  + \sum_{i = 3}^{7} s^i \, \tilde{\bv}_i
                  + \smallo{s^7}
              \end{equation}
              with constants $\tilde{a}_2, \tilde{a}_4 \in \R$
              and $\tilde{\bv}_i \in \Span{\be_2, \dots, \be_N}$
              ($3 \leq i \leq 7$).
              Indeed, by a direct substitution, we get
              \begin{align}
                \paren{\varGamma \circ \phi} \paren{s}
                & = \bigl(%
                      s^2
                      + \left(a_1 + 2 c_1\right) s^3
                      + \left(a_2 + 3 a_1 c_1 + c_1^2\right) s^4 \\
                  & \qquad 
                      + \left(a_3 + 4 a_2 c_1 + 3 a_1 c_1^2 + 2 c_3\right) s^5 \\
                  & \qquad 
                      + \left(a_4 + 5 a_3 c_1 + 6 a_2 c_1^2 + a_1 \left(c_1^3 + 3 c_3\right) + 2 c_1 c_3\right) s^6 \\
                  & \qquad 
                      + \left(a_5 + 6 a_4 c_1 + 10 a_3 c_1^2 + 4 a_2 \left(c_1^3 + c_3\right) + 6 a_1 c_1 c_3 + 2 c_5\right) s^7, 
                      0, \dots, 0
                    \bigr) \\
                  & \quad
                   + s^3 \bv_3
                   + s^4 \left(\bv_4 + 3 c_1 \bv_3\right)
                   + s^5 \left(\bv_5 + 4 c_1 \bv_4 + 3 c_1^2 \bv_3\right) \\
                  & \quad
                   + s^6 \left(\bv_6 + 5 c_1 \bv_5 + 6 c_1^2 \bv_4 + \left(c_1^3 + 3 c_3\right) \bv_3\right) \\
                  & \quad
                   + s^7 \left(\bv_7 + 6 c_1 \bv_6 + 10 c_1^2 \bv_5 + 4 \left(c_1^3 + c_3\right) \bv_4 + 6 c_1 c_3 \bv_3\right)
                   + \smallo{s^7}.
              \end{align}
              Therefore, solving the system of equations
              \begin{align}
                  a_1 + 2 c_1
                  & = a_3 + 4 a_2 c_1 + 3 a_1 c_1^2 + 2 c_3 \\
                  & = a_5 + 6 a_4 c_1 + 10 a_3 c_1^2 + 4 a_2 \left(c_1^3 + c_3\right) + 6 a_1 c_1 c_3 + 2 c_5
                  = 0,
              \end{align}
              we get
              \begin{gather}
                c_{1} = - \frac{1}{2} a_1
                \quad
                c_{3} = \frac{1}{8} \paren{- 3 a_1^3 + 8 a_1 a_2 - 4 a_3}, \\
                c_{5} = \frac{1}{16} \left(-9 a_1^5 + 40 a_1^3 a_2 - 32 a_1^2 a_3 - 32 a_1 a_2^2 + 24 a_1 a_4 + 16 a_2 a_3 - 8 a_5\right).
              \end{gather}
              Also, we have
              \begin{gather}
                  \tilde{\bv}_3
                  = \bv_3, 
                \quad
                  \tilde{\bv}_5
                   = \bv_5 + 4 c_1 \bv_4 + 3 c_1^2 \bv_3, \\
                 \tilde{\bv}_7
                  = \bv_7 + 6 c_1 \bv_6 + 10 c_1^2 \bv_5 + 4 \paren{c_1^3 + c_3} \bv_4 + 6 c_1 c_3 \bv_3.
              \end{gather}
              
              (1) If $\bv_3 \neq \bzero$,
                  then $\be_1$ and $\tilde{\bv}_3$ are linearly independent, and hence,
                  by using a matrix $T_2 \in \GL{N,\R}$ obtained in \cref{lem:lin_change_to_standard_basis} from $(\bm{e}_1, \tilde{\bm{v}}_3)$,
                  the curve $T_2 \, \paren{\varGamma \circ \phi} \paren{s}$ satisfies
                  $\sderiv{(T_2 \, (\varGamma \circ \phi))}{1} = \bzero$, 
                  $\sderiv{(T_2 \, (\varGamma \circ \phi))}{2} = \be_1$, 
                  and $\sderiv{(T_2 \, (\varGamma \circ \phi))}{3} = \be_2$.
                  Hence, it takes the form
                    $T_2 \, \paren{\varGamma \circ \phi} \paren{s}
                    = \paren{s^2, s^3, 0, \dots, 0}
                      + \smallo{s^3}$.
                  This curve is \Aqui-equivalent to $\Cusp{2, 3}$ 
                  by \cref{thm:A.equivalence.suff.cond.1} and \eqref{eq:NR.2.2,2n+1}.
                  
                (2) If $\bv_3 = \bzero$
                  and $\bv_5 - 2 a_1 \bv_4 \neq \bzero$,
                  then $\tilde{\bv}_3 = \bzero$
                  and
                  $\tilde{\bv}_5
                    = \bv_5 - 2 a_1 \, \bv_4
                    \neq \bzero$.
                  Thus, $\be_1$ and $\tilde{\bv}_5$ are linearly independent,
                  and hence,
                  using
                  $T_3
                    \in \GL{N, \R}$
                  in \cref{lem:lin_change_to_standard_basis},
                  we obtain
                  \begin{equation}
                    T_3 \, \paren{\varGamma \circ \phi} \paren{s}
                    = \paren{s^2, s^5, 0, \dots, 0}
                      + s^4 \, \bw_4
                      + \smallo{s^5}
                  \end{equation}
                  with a constant $\bw_4 \in \R^N$.
                  This curve is \Aqui-equivalent to $\Cusp{2, 5}$ by
                  \cref{thm:A.equivalence.suff.cond.1}
                  and \eqref{eq:NR.2.2,2n+1}.
                
                (3) If $\bv_3 = \bzero$,
                  $\bv_5 - 2 a_1 \, \bv_4 = \bzero$,
                  and $\bv_7 - 3 a_1 \, \bv_6 + \paren{3 a_1^3 + 4 a_1 a_2 - 2 a_3} \, \bv_4 \neq \bzero$,
                  then $\tilde{\bv}_3 = \tilde{\bv}_5 = \bzero$
                  and
                  $\tilde{\bv}_7
                    = \bv_7 - 3 a_1 \, \bv_6
                        + \paren{3 a_1^3 + 4 a_1 a_2 - 2 a_3} \, \bv_4 
                    \neq \bzero$.
                  Thus, $\be_1$ and $\tilde{\bv}_7$ are linearly independent,
                  and hence,
                  using
                  $T_4
                    \in \GL{N, \R}$
                  in \cref{lem:lin_change_to_standard_basis},
                  we obtain
                  \begin{equation}
                    T_4 \, \paren{\varGamma \circ \phi} \paren{s}
                    = \paren{s^2, s^7, 0, \dots, 0}
                      + s^4 \, \bw_4
                      + s^6 \, \bw_6
                      + \smallo{s^7}
                  \end{equation}
                  with constants $\bw_4, \bw_6 \in \R^N$.
                  This curve is \Aqui-equivalent to $\Cusp{2, 7}$ by
                  \cref{thm:A.equivalence.suff.cond.1} and \eqref{eq:NR.2.2,2n+1}.

          \textbf{Step 3}.
              We note that $\sderiv{\varGamma}{2} = \be_1$ and
              $\sderiv{\varGamma}{i}
                = a_{i - 2} \, \be_1 + \bv_i
                = a_{i - 2} \, \sderiv{\varGamma}{2} + \bv_i$ 
            for each $3 \leq i \leq 7$.

                (1) Since
                  $\sderiv{\varGamma}{3}
                    = a_1 \, \sderiv{\varGamma}{2} + \bv_3$,
                  the condition $\sderiv{\varGamma}{3} \notin \Span{\sderiv{\varGamma}{2}}$
                  means $\bv_3 \neq \bzero$.
                  Hence, by \textbf{Step 2},
                  the curve $\varGamma$ is \Aqui-equivalent to $\Cusp{2, 3}$.
                  
                (2) The condition
                  $\sderiv{\varGamma}{3} = \lambda \, \sderiv{\varGamma}{2}$
                  means $a_1 = \lambda$ and $\bv_3 = \bzero$.
                  Then, 
                  $
                    \sderiv{\varGamma}{5}
                      - 2 \lambda \, \sderiv{\varGamma}{4}
                     = \paren{a_3 - 2 \lambda \, a_2} \, \sderiv{\varGamma}{2}
                      + \paren{\bv_5 - 2 a_1 \, \bv_4}.
                  $
                  Thus, the condition
                  $\sderiv{\varGamma}{5}
                      - 2 \, \lambda \, \sderiv{\varGamma}{4}
                    \notin \Span{\sderiv{\varGamma}{2}}$
                  means $\bv_5 - 2 a_1 \, \bv_4 \neq \bzero$.
                  Hence, by \textbf{Step 2},
                  the curve $\varGamma$ is \Aqui-equivalent to $\Cusp{2, 5}$.
                  
                (3) Note that $a_1 = \lambda$ and $\bv_3 = \bzero$.
                  The condition 
                  $\sderiv{\varGamma}{5}
                      - 2 \lambda \, \sderiv{\gamma}{4}
                    = \mu \, \sderiv{\varGamma}{2}$
                  means $a_3 - 2 \lambda \, a_2 = \mu$
                  and $\bv_5 - 2 a_1 \, \bv_4 = \bzero$.
                  Then,
                  \begin{align}
                    \sderiv{\varGamma}{7}
                    &  - 3 \lambda \, \sderiv{\varGamma}{6}
                       + \paren{3 \lambda^3 - 2 \mu} \, \sderiv{\varGamma}{4} \\
                    & = \paren{%
                            a_5 - 3 \lambda \, a_4
                            + \paren{3 \lambda^3 + 4 \lambda \, a_2 - 2 a_3} \, a_2
                          } \, \sderiv{\varGamma}{2} \\
                    &    \qquad \qquad
                    + \paren{%
                            \bv_7 - 3 a_1 \, \bv_6
                            + \paren{3 a_1^3 + 4 a_1 \, a_2 - 2 a_3} \, \bv_4
                          }.
                  \end{align}
                  Thus, the condition
                  $\sderiv{\varGamma}{7}
                      - 3 \lambda \, \sderiv{\varGamma}{6}
                      + \paren{3 \lambda^3 - 2 \mu} \, \sderiv{\varGamma}{4}
                    \notin \Span{\sderiv{\varGamma}{2}}$
                  means
                  $\bv_7 - 3 a_1 \, \bv_6
                      + \paren{3 a_1^3 + 4 a_1 \, a_2 - 2 a_3} \, \bv_4
                    \neq \bzero$.
                  By \textbf{Step 2},
                  the curve $\varGamma$ is \Aqui-equivalent to $\Cusp{2, 7}$.
                  This completes the proof of the ``if'' part of the theorem.
              
              \textbf{Step 4}.
              Finally, we prove the ``only if'' part of the theorem.
              Each of the standard singularities
              $\Cusp{2, 3}$, $\Cusp{2, 5}$, and $\Cusp{2, 7}$
              satisfies the respective conditions stated in the claim.
              Hence, we only need to verify that the conditions stated in the claim are invariant under \A-equivalence.
              Applying \cref{prop:multi_and_lin_dep.invariant} for $m = 2$ and $n = 3$ directly shows \ref{item:criteria.2.3.criteria} in the claim.
              It remains to verify the invariance of conditions \ref{item:criteria.2.5.criteria} and \ref{item:criteria.2.7.criteria}.
              Let us separately show the invariance under
             (a) parameter changes and (b) coordinate transformations.

                (a) We take a parameter change $\phi$
                  and consider the curve $\tilde{\gamma} \coloneqq \gamma \circ \phi$.
                  Direct computation gives that
                  $
                    \sderiv{\tilde{\gamma}}{2}
                    = \paren{\sderiv{\phi}{1}}^2 \, \sderiv{\gamma}{2}
                  $ 
                   and
                  $ 
                    \sderiv{\tilde{\gamma}}{3}
                    = \paren{\sderiv{\phi}{1}}^3 \, \sderiv{\gamma}{3}
                      + 2 \sderiv{\phi}{1} \, \sderiv{\phi}{2} \, \sderiv{\gamma}{2}
                  $    
                  by \cref{cor:Bruno_for_curves} \ref{item:cor_Bruno_curve.para}.
                  
                  Assume that there exists $\lambda \in \R$ satisfying 
                  \ref{item:criteria.2.5.criteria}.
                  Then,
                  $
                    \tilde{\lambda}
                    \coloneqq \lambda \, \sderiv{\phi}{1}
                      + \sfrac{2 \sderiv{\phi}{2}}{\sderiv{\phi}{1}}
                  $
                  satisfies
                  $\sderiv{\tilde{\gamma}}{3}
                    = \tilde{\lambda} \, \sderiv{\tilde{\gamma}}{2}$.
                Then,
                  \begin{equation}
                    \sderiv{\tilde{\gamma}}{5}
                      - 2 \, \tilde{\lambda} \, \sderiv{\tilde{\gamma}}{4}
                    \in \paren{\sderiv{\phi}{1}}^5 \,
                        \paren{%
                          \sderiv{\gamma}{5}
                          - 2 \, \lambda \, \sderiv{\gamma}{4}
                        }
                      + \Span{\sderiv{\gamma}{2}},
                  \end{equation}
                  which shows that the condition \ref{item:criteria.2.5.criteria}
                  is invariant under parameter changes.
                  
                  On the other hand,
                  if there exist $\lambda, \mu \in \R$ satisfying \ref{item:criteria.2.7.criteria},
                  then taking the above $\tilde{\lambda}$, we know that
                  \begin{align}
                    \tilde{\mu}
                    &\coloneqq
                      \mu \, \paren{\sderiv{\phi}{1}}^3
                      - 6 \lambda^2 \, \sderiv{\phi}{1} \, \sderiv{\phi}{2}
                      - \lambda \, \sderiv{\phi}{3} \\
                    &\hspace{3cm}  
                    + \frac{2 \sderiv{\phi}{4} - 11 \lambda \, \paren{\sderiv{\phi}{2}}^2}{\sderiv{\phi}{1}}
                      - \frac{6 \sderiv{\phi}{2} \, \sderiv{\phi}{3}}{\paren{\sderiv{\phi}{1}}^2}
                      - \frac{4 \paren{\sderiv{\phi}{2}}^3}{\paren{\sderiv{\phi}{1}}^3}
                  \end{align}
                  satisfies
                  $\sderiv{\tilde{\gamma}}{5}
                      - 2 \, \tilde{\lambda} \, \sderiv{\tilde{\gamma}}{4}
                    = \tilde{\mu} \, \sderiv{\tilde{\gamma}}{2}$.
                  Moreover, we observe
                  \begin{align}
                    &\sderiv{\tilde{\gamma}}{7}
                        - 3 \, \tilde{\lambda} \, \sderiv{\tilde{\gamma}}{6}
                        + \paren{3 \, \tilde{\lambda}^3 - 2 \, \tilde{\mu}} \, \sderiv{\tilde{\gamma}}{4} \\
                    & \hspace{2cm}
                    \in \paren{\sderiv{\phi}{1}}^7 \,
                        \paren{%
                          \sderiv{\gamma}{7}
                          - 3 \, \lambda \, \sderiv{\gamma}{6}
                          + \paren{3 \, \lambda^3 - 2 \, \mu} \, \sderiv{\gamma}{4}
                        }
                      + \Span{\sderiv{\gamma}{2}},
                  \end{align}
                  and hence,
                  the condition \ref{item:criteria.2.7.criteria}
                  is invariant under parameter changes.
                  
                (b) We take a coordinate transformation $\varPhi$
                  and consider the curve $\hat{\gamma} \coloneqq \varPhi \circ \gamma$.
                  By
                  \cref{cor:Bruno_for_curves} \ref{item:cor_Bruno_curve.coord},
                  and \cref{rem:coord_transf.multi_m}
                  with $m = 2$, we have
                  $
                    \sderiv{\hat{\gamma}}{2}
                     = J_{\varPhi} \, \sderiv{\gamma}{2},
                     \, 
                    \sderiv{\hat{\gamma}}{3}
                    = J_{\varPhi} \, \sderiv{\gamma}{3}.
                  $
                  Assume that there exists $\lambda \in \R$ such that
                  $\sderiv{\gamma}{3} = \lambda \, \sderiv{\gamma}{2}$.
                  Then, 
                  $
                    \sderiv{\hat{\gamma}}{3}
                    = J_{\varPhi} \, \sderiv{\gamma}{3}
                    = \lambda \, J_{\varPhi} \, \sderiv{\gamma}{2}
                    = \lambda \, \sderiv{\hat{\gamma}}{2}.
                  $

                  Let us show the invariance of the condition \ref{item:criteria.2.5.criteria}
                  under coordinate transformations.
                  We suppose that $\sderiv{\gamma}{5} - 2 \lambda \sderiv{\gamma}{4} 
                  \not \in \Span{\sderiv{\gamma}{2}}$.
                  Since 
                  
                  \begin{align}
                    \sderiv{\hat{\gamma}}{4}
                    & = J_{\varPhi} \, \sderiv{\gamma}{4}
                        + \frac{1}{2}
                          \sum_{i, j = 1}^{N}
                          \pfrac{^2 \varPhi}{x_i \, \partial x_j} \,
                          \sderiv{x}{2}_i \,
                          \sderiv{x}{2}_j, \quad
                    \sderiv{\hat{\gamma}}{5}
                     = J_{\varPhi} \, \sderiv{\gamma}{5}
                        + \lambda
                          \sum_{i, j = 1}^{N}
                            \pfrac{^2 \varPhi}{x_i \, \partial x_j} \,
                            \sderiv{x}{2}_i \,
                            \sderiv{x}{2}_j,
                  \end{align}
                  we obtain
                  $
                    \sderiv{\hat{\gamma}}{5}
                        - 2 \, \lambda \, \sderiv{\hat{\gamma}}{4}
                     = J_{\varPhi}
                        \paren{%
                          \sderiv{\gamma}{5}
                          - 2 \, \lambda \, \sderiv{\gamma}{4}
                        },
                  $ 
                  which implies the assertion.

                  Let us finally show the invariance of the 
                  condition \ref{item:criteria.2.7.criteria} under
                  coordinate transformations.

                  If there exist the above $\lambda$ and $\mu \in \R$ such that
                  $\sderiv{\gamma}{5}
                      - 2 \, \lambda \, \sderiv{\gamma}{4}
                    = \mu \, \sderiv{\gamma}{2}$ and 
                    $\sderiv{\gamma}{7}
                      - 3 \, \lambda \, \sderiv{\gamma}{6}
                      + \paren{3 \, \lambda^3 - 2 \, \mu} \, \sderiv{\gamma}{4}
                    \notin \Span{\sderiv{\gamma}{2}}$, 
                  then
                  \begin{align}
                    \sderiv{\hat{\gamma}}{6}
                    & = J_{\varPhi} \, \sderiv{\gamma}{6}
                       + \sum_{i, j = 1}^{N}
                          \pfrac{^2 \varPhi}{x_i \, \partial x_j} \,
                          \sderiv{x}{2}_i \,
                          \sderiv{x}{4}_j
                        + \frac{\lambda^2}{2}
                          \sum_{i, j = 1}^{N}
                            \pfrac{^2 \varPhi}{x_i \, \partial x_j} \,
                            \sderiv{x}{2}_i \,
                            \sderiv{x}{2}_j \\
                    &   \hspace{4cm} + \frac{1}{6}
                          \sum_{i, j, k = 1}^{N}
                            \pfrac{^3 \varPhi}{x_i \, \partial x_j \, \partial x_k} \,
                            \sderiv{x}{2}_i \,
                            \sderiv{x}{2}_j \,
                            \sderiv{x}{2}_k,\\
                    \sderiv{\hat{\gamma}}{7}
                    & = J_{\varPhi} \, \sderiv{\gamma}{7}
                        + \sum_{i, j = 1}^{N}
                          \pfrac{^2 \varPhi}{x_i \, \partial x_j} \,
                          \paren{%
                            3 \lambda \, \sderiv{x}{2}_i \, \sderiv{x}{4}_j
                            + \mu \, \sderiv{x}{2}_i \, \sderiv{x}{2}_j
                          } \\
                        & \hspace{4cm} + \frac{\lambda}{2}
                            \sum_{i, j, k = 1}^{N}
                              \pfrac{^3 \varPhi}{x_i \, \partial x_j \, \partial x_k} \,
                              \sderiv{x}{2}_i \,
                              \sderiv{x}{2}_j \,
                              \sderiv{x}{2}_k.
                  \end{align}
                  yield 
                  $
                    \sderiv{\hat{\gamma}}{7}
                        - 3 \, \lambda \, \sderiv{\hat{\gamma}}{6}
                        + \paren{3 \, \lambda^3 - 2 \, \mu} \, \sderiv{\hat{\gamma}}{4}
                    = J_{\varPhi}
                        \paren{%
                          \sderiv{\gamma}{7}
                          - 3 \, \lambda \, \sderiv{\gamma}{6}
                          + \paren{3 \, \lambda^3 - 2 \, \mu} \, \sderiv{\gamma}{4}
                        },
                $
                which gives the conclusion and
              completes the proof of \cref{thm:criteria.2.3.and.2.5.and.2.7}.
    \end{proof}

  \subsection{Singularities of Multiplicity $3$ in \Rpdf[N]}
  \label{sec:Singularities.multiplicity.3}%

    This subsection is devoted to
    constructing criteria for classifying singularities of multiplicity $3$,
    in particular,
    $\paren{3, 4}$-cusps,
    $\paren{3, 4, 5}$-cusps,
    $\paren{3, 5}$-cusps,
    $\paren{3, 5, 7}$-cusps,
    $\paren{3, 7}$-cusps,
    $\paren{3, 7_{+8}}$-cusps,
    $\paren{3, 7, 8}$-cusps,
    $\paren{3, 7, 11}$-cusps, and
    $\paren{3, 7_{+8}, 11}$-cusps.

    \begin{thm}
        \label{conj:criteria.3.4.5.and.3.4.and.3.5.7.and.3.5}%
        Let $\gamma \colon \paren{\R, 0} \to \paren{\R^N, \bzero}$
        be a curve in $\R^N$.
        \begin{enumerate}[label={(\arabic*)}]
            \item \label{item:thm.criteria:3.4.5.and.3.4}
                The curve $\gamma$ is \A-equivalent to
                    $\mathrm{(i)} \ \Cusp{3, 4, 5}
                            ( t ) = \paren{t^3, t^4, t^5, 0, \dots, 0}$, 
                    $\mathrm{(ii)} \ \Cusp{3, 4}
                            ( t ) = \paren{t^3, t^4, 0, \dots, 0}$
                if and only if
                $
                    \sderiv{\gamma}{1}
                    = \sderiv{\gamma}{2}
                    = \bzero
                    \neq \sderiv{\gamma}{3}
                $, 
                $
                    \sderiv{\gamma}{4}
                    \notin \Span{\sderiv{\gamma}{3}},
                $
                and respectively, 
                \begin{enumerate}[label={(\roman*)}]
                    \item $\sderiv{\gamma}{5}
                            \notin \Span{\sderiv{\gamma}{3}, \sderiv{\gamma}{4}}$,
                    \item $\sderiv{\gamma}{5}
                            \in \Span{\sderiv{\gamma}{3}, \sderiv{\gamma}{4}}$.
                \end{enumerate}
            \item \label{item:thm.criteria:3.5.7.and.3.5}%
                The curve $\gamma$ is \A-equivalent to
                    $\mathrm{(i)} \ \Cusp{3, 5, 7}(t) = \paren{t^3, t^5, t^7, 0, \dots, 0}$,
                    $\mathrm{(ii)} \ \Cusp{3, 5}(t) = \paren{t^3, t^5, 0, \dots, 0}$
                if and only if there exists $\lambda \in \R$
                such that
                \begin{equation}
                    \sderiv{\gamma}{1}
                    = \sderiv{\gamma}{2}
                    = \bzero
                    \neq \sderiv{\gamma}{3},
                    \quad
                    \sderiv{\gamma}{4}
                    = \lambda \, \sderiv{\gamma}{3},
                    \quad
                    \sderiv{\gamma}{5}
                    \notin \Span{\sderiv{\gamma}{3}},
                    \label{eq:criteria.3.5.7.and.3.5:conditions.1}
                \end{equation}
                and respectively,
                \begin{enumerate}[label={(\roman*)}]
                    \item $\sderiv{\gamma}{7}
                          - 2 \, \lambda \, \sderiv{\gamma}{6}
                        \notin \Span{\sderiv{\gamma}{3}, \sderiv{\gamma}{5}}$,
                    \item $\sderiv{\gamma}{7}
                          - 2 \, \lambda \, \sderiv{\gamma}{6}
                        \in \Span{\sderiv{\gamma}{3}, \sderiv{\gamma}{5}}$.
                \end{enumerate}
        \end{enumerate}
    \end{thm}

    \begin{proof}
        We prove the claim following the procedure we outlined in \cref{sec:Introduction}.
        First, we prove the ``if'' part of (1).
        
        \textbf{Step 1 of (1)}.
        We observe
        \begin{equation} \label{eq:NR.2.3,4.NR.2.3,4,5}
            \cNR{2}{3, 4} = \cNR{2}{3, 4, 5} = \cparen{0, 1, 2, 3, 4, 5}.
        \end{equation}
        
        \textbf{Steps 2 and 3 of (1)}.
                (i) If $\gamma$ satisfies
                  $\sderiv{\gamma}{4}
                    \notin \Span{\sderiv{\gamma}{3}}$
                  and
                  $\sderiv{\gamma}{5}
                    \notin
                      \Span{%
                        \sderiv{\gamma}{3},
                        \sderiv{\gamma}{4}
                      }$,
                  then the curve
                  $\varGamma \paren{t}
                    = T_1 \,
                        \gamma \paren{t}$,
                  obtained by 
                  $T_1 \in \GL{N, \R}$
                  given in \cref{lem:lin_change_to_standard_basis},
                  satisfies
                  $\sderiv{\varGamma}{1}
                    = \sderiv{\varGamma}{2}
                    = \bzero$ 
                  and 
                  $\sderiv{\varGamma}{i}
                    = \be_{i - 2}
                    \quad
                    (i = 3, 4, 5)$.
                  Hence, the curve $\varGamma$ takes the form
                    $
                    \varGamma \paren{t}
                    = \paren{t^3, t^4, t^5, 0, \dots, 0}
                      + \smallo{t^5}$.
                  This curve $\varGamma$ is \Aqui-equivalent to $\Cusp{3, 4, 5}$ by
              \cref{thm:A.equivalence.suff.cond.1} and \eqref{eq:NR.2.3,4.NR.2.3,4,5}.
              
              (ii) If $\gamma$ satisfies
                  $\sderiv{\gamma}{4}
                    \notin \Span{\sderiv{\gamma}{3}}$
                  and
                  $\sderiv{\gamma}{5}
                    \in
                      \Span{%
                        \sderiv{\gamma}{3},
                        \sderiv{\gamma}{4}
                      }$,
                  then the curve
                  $\varGamma \paren{t}
                    = T_2\,\gamma \paren{t}$ 
                  satisfies
                    $\sderiv{\varGamma}{1}
                    = \sderiv{\varGamma}{2}
                    = \bzero$, 
                    $\sderiv{\varGamma}{3}
                    = \be_1$, 
                    $\sderiv{\varGamma}{4}
                    = \be_2$, and 
                    $\sderiv{\varGamma}{5}
                    \in \Span{\be_1, \be_2}$,
                    where $T_2 \in \GL{N,\R}$ is in \cref{lem:lin_change_to_standard_basis}.
                  Hence, the curve $\varGamma$ takes the form
                  \begin{equation}
                    \label{eq:34.standard_form}%
                    \varGamma \paren{t}
                    = \paren{t^3 + a_2 \, t^5, t^4 + b_1 \, t^5, 0, \dots, 0}
                      + \smallo{t^5}
                  \end{equation}
                  with constants $a_2, b_1 \in \R$.
                  Through the parameter change
                  $\phi (s)
                    = s - (b_1s^2) / 4
                        + ((- 16 a_2 - 3 b_1^2) \, s^3) / 48 $
                  and the coordinate transformation 
                    $\varPhi (x_1, x_2, \dots, x_N)
                    = (
                          x_1 + 3 b_1 x_2 / 4,
                          x_2, \dots, x_N
                        )$,
                  we get
                    $\paren{\varPhi \circ \varGamma \circ \phi} \paren{s}
                    = \paren{s^3, s^4, 0, \dots, 0}
                      + \smallo{s^5}$.
                  This curve is \Aqui-equivalent to $\Cusp{3, 4}$ by
                  \cref{thm:A.equivalence.suff.cond.1} and \eqref{eq:NR.2.3,4.NR.2.3,4,5}.

            \textbf{Step 1 of (2)}.
            Next, let us prove the ``if'' part of (2). 
            We observe
            \begin{equation} \label{eq:NR.2.3,5.NR.2.3,5,7}
                \cNR{2}{3, 5} = \cNR{2}{3, 5, 7} = \cparen{0, 1, 2, 3, 4, 5, 7}.
            \end{equation}

            \textbf{Step 2 of (2)}.
            If $\gamma$ satisfies
              $\sderiv{\gamma}{4}
                \in \Span{\sderiv{\gamma}{3}}
                \notni \sderiv{\gamma}{5}$,
              then there exists $\lambda \in \R$
              such that
              $\sderiv{\gamma}{4} = \lambda \, \sderiv{\gamma}{3}$.
              Thus, the curve
              $\varGamma \paren{t}
                = T_3 \,
                    \gamma \paren{t}$
              satisfies
                $\sderiv{\varGamma}{1}
                = \sderiv{\varGamma}{2}
                = \bzero$, 
                $\sderiv{\varGamma}{3}
                = \be_1$, 
                $\sderiv{\varGamma}{4}
                = \lambda \, \be_1$, and 
                $\sderiv{\varGamma}{5}
                = \be_2$, 
                where $T_3 \in \GL{N,\R}$ is in \cref{lem:lin_change_to_standard_basis}.
              Hence, the curve $\varGamma$ takes the form
              \begin{equation}
                \label{eq:35.standard_form}%
                \varGamma \paren{t}
                = \left(%
                      t^3 + \lambda \, t^4 + a_{3} \, t^{6} + a_{4} \, t^{7},
                      t^5 + b_{1} \, t^{6} + b_{2} \, t^{7},
                      0, \dots, 0
                    \right)
                  + t^6 \, \bv_6
                  + t^7 \, \bv_7
                  + \smallo{t^7}
              \end{equation}
              with constants $a_3, a_4, b_1, b_2 \in \R$
              and $\bv_6, \bv_7
                \in \Span{\be_3, \dots, \be_N}$.
                
              \textbf{Step 3 of (2)}.
              In this step, we show that, if
                   (i) $\bv_7 - 2 \lambda \, \bv_6 \neq \bzero$ or
                   (ii) $\bv_7 - 2 \lambda \, \bv_6 = \bzero$,
                then $\varGamma$ is \Aqui-equivalent to 
                   (i) $\Cusp{3, 5, 7}$ or
                   (ii) $\Cusp{3, 5}$
                respectively.

                Suppose that $\varGamma$ is expressed in the form of \cref{eq:35.standard_form}.
              By an argument similar to that in the proof of \cref{thm:criteria.2.3.and.2.5.and.2.7}, there exist  $c_{i} \in \R$ ($i = 1, 2, 4$) and $p_1 \in \R$
              such that the parameter change $\phi \paren{s} = s + c_1 \, s^2 + c_2\, s^3 + c_4\, s^5$ and the coordinate transformation $\varPhi (x_1, x_2, \dots, x_N) = (x_1 + p_1 \, x_2, x_2, \dots, x_N)$ yield
              \begin{equation}
                \label{eq:357_and_35.transformed}%
                \paren{\varPhi \circ \varGamma \circ \phi} \paren{s}
                = \paren{s^3 + \tilde{a}_3 \, s^6, s^5 + \tilde{b}_1 \, s^6, 0, \dots, 0}
                  + s^6 \, \tilde{\bv}_6
                  + s^7 \, \tilde{\bv}_7
                  + \smallo{s^7}
              \end{equation}
              with constants $\tilde{a}_3, \tilde{b}_1 \in \R$ and $\tilde{\bv}_6, \tilde{\bv}_7 \in \Span{\be_3, \dots, \be_N}$.

                (i) If $\tilde{\bv}_7 = \bv_7 - 2 \lambda \, \bv_6 \neq \bzero$, then $(\bm{e}_1, \bm{e}_2, \tilde{\bm{v}}_7)$ is linearly independent.
                Thus, the matrix $T_4 \in \GL{N, \R}$ obtained from this tuple by \cref{lem:lin_change_to_standard_basis} gives
                  \begin{equation}
                    T_4 \, \paren{\varPhi \circ \varGamma \circ \phi} \paren{s}
                    = \paren{%
                          s^3 + \tilde{a}_3 \, s^6,
                          s^5 + \tilde{b}_1 \, s^6,
                          s^7,
                          0, \dots, 0
                        }
                      + T_4 \, \tilde{\bm{v}}_6 + \smallo{s^7},
                  \end{equation}
                  is \Aqui-equivalent to $\Cusp{3, 5, 7}$ by \cref{thm:A.equivalence.suff.cond.1} and \eqref{eq:NR.2.3,5.NR.2.3,5,7}. 
                  
                (ii) If $\tilde{\bv}_7 = \bv_7 - 2 \lambda \, \bv_6 = \bzero$, then the curve in \cref{eq:357_and_35.transformed} is \Aqui-equivalent to $\Cusp{3, 5}$ by \cref{thm:A.equivalence.suff.cond.1} and 
                  \eqref{eq:NR.2.3,5.NR.2.3,5,7}. 

              We note that 
              $\sderiv{\varGamma}{7} - 2 \lambda \, \sderiv{\varGamma}{6}
               = (a_3 - 2 \lambda \, a_4) \, \be_1 + 
                 (b_1 - 2 \lambda \, b_2) \, \be_2 +
                 \bv_6 - 2\lambda \, \be_7$.
            
                (i) The condition
                  $\sderiv{\varGamma}{7} - 2 \, \lambda \, \sderiv{\varGamma}{6}
                    \notin \Span{\sderiv{\varGamma}{3}, \sderiv{\varGamma}{5}}$
                  means $\bv_7 - 2 \lambda \, \bv_6 \neq \bzero$.
                  Hence, by \textbf{Step 2 of (2)},
                  the curve $\varGamma$ is \Aqui-equivalent to $\Cusp{3, 5, 7}$.
                  
                (ii) The condition
                  $\sderiv{\varGamma}{7} - 2 \, \lambda \, \sderiv{\varGamma}{6}
                    \in \Span{\sderiv{\varGamma}{3}, \sderiv{\varGamma}{5}}$
                  means $\bv_7 - 2 \lambda \, \bv_6 = \bzero$.
                  Hence, by \textbf{Step 2 of (2)},
                  the curve $\varGamma$ is \Aqui-equivalent to $\Cusp{3, 5}$.
                  This completes the proof of  the ``if'' part of (2).

            \textbf{Step 4 of (1) and (2)}.
            Finally, we prove the ``only if'' part of (1) and (2). Each of the standard singularities $\Cusp{3, 4, 5}$, $\Cusp{3, 4}$, $\Cusp{3, 5, 7}$, and $\Cusp{3, 5}$ satisfies its respective conditions.
              Hence, we only need to verify that the stated conditions are invariant under \A-equivalence.
              Most of the conditions stated in the claim are immediately seen to be invariant under \A-equivalence from \cref{prop:multi_and_lin_dep.invariant}.
              In particular, the criteria \ref{item:thm.criteria:3.4.5.and.3.4} is clear.
              Thus, under \cref{eq:criteria.3.5.7.and.3.5:conditions.1},
              it remains to verify the invariance of the condition
              \begin{equation}
                \label{eq:criteria.3.5.7.and.3.5:condition}%
                \sderiv{\gamma}{7}
                  - 2 \, \lambda \, \sderiv{\gamma}{6}
                \notin \Span{\sderiv{\gamma}{3}, \sderiv{\gamma}{5}}.
              \end{equation}

              Suppose that $\gamma$ satisfies
              \cref{eq:criteria.3.5.7.and.3.5:conditions.1}.

              Let us separately show the invariance under
              (a) parameter changes and (b) coordinate transformations.

              (a)
                  We take a parameter change $\phi$
                  and consider the curve $\tilde{\gamma} \coloneqq \gamma \circ \phi$.
                  We can compute the Taylor coefficients as follows:
                  \begin{align}
                    \sderiv{\tilde{\gamma}}{3}
                    & = \paren{\sderiv{\phi}{1}}^3 \, \sderiv{\gamma}{3}, \qquad
                    \sderiv{\tilde{\gamma}}{4}
                     = \Bigl( \lambda \paren{\sderiv{\phi}{1}}^4+3 \sderiv{\phi}{2} \paren{\sderiv{\phi}{1}}^2 \Bigr) \, \sderiv{\gamma}{3}, \\
                    \sderiv{\tilde{\gamma}}{5}
                    & = \paren{\sderiv{\phi}{1}}^5 \, \sderiv{\gamma}{5}
                      + \sderiv{\phi}{1} \left(4 \lambda \sderiv{\phi}{2} \paren{\sderiv{\phi}{1}}^2+3 \sderiv{\phi}{3} \sderiv{\phi}{1}+3 \paren{\sderiv{\phi}{2}}^2\right) \, \sderiv{\gamma}{3}, \\
                    \sderiv{\tilde{\gamma}}{6}
                    & = \paren{\sderiv{\phi}{1}}^6 \, \sderiv{\gamma}{5}
                      + 5 \paren{\sderiv{\phi}{1}}^4 \sderiv{\phi}{2} \, \sderiv{\gamma}{5} \\
                    & \qquad
                      + \biggl(
                          4 \lambda \sderiv{\phi}{3} \paren{\sderiv{\phi}{1}}^3+3 \paren{\sderiv{\phi}{1}}^2 \left(2 \lambda \paren{\sderiv{\phi}{2}}^2+\sderiv{\phi}{4}\right) \\
                          & \hspace{15em}
                          +6 \sderiv{\phi}{2} \sderiv{\phi}{3} \sderiv{\phi}{1}+\paren{\sderiv{\phi}{2}}^3
                      \biggr)
                      \, \sderiv{\gamma}{3},
                  \end{align}
                  \begin{align}
                    \sderiv{\tilde{\gamma}}{7}
                    & = \paren{\sderiv{\phi}{1}}^7 \, \sderiv{\gamma}{7}
                      + 6 \paren{\sderiv{\phi}{1}}^5 \sderiv{\phi}{2} \, \sderiv{\gamma}{6}
                      + 5 \paren{\sderiv{\phi}{1}}^3 \left(2 \paren{\sderiv{\phi}{2}}^2+\sderiv{\phi}{1} \sderiv{\phi}{3}\right) \, \sderiv{\gamma}{5} \\
                    & \qquad 
                      + \biggl( 
                          4 \lambda \sderiv{\phi}{4} \paren{\sderiv{\phi}{1}}^3+3 \paren{\sderiv{\phi}{1}}^2 
                          \left(4 \lambda \sderiv{\phi}{2} \sderiv{\phi}{3}+\sderiv{\phi}{5}\right) \\
                                & \qquad \qquad 
                                +\sderiv{\phi}{1} \left(4 \lambda \paren{\sderiv{\phi}{2}}^3+6 \sderiv{\phi}{4} \sderiv{\phi}{2}
                                +3 \paren{\sderiv{\phi}{3}}^2\right)+3 \paren{\sderiv{\phi}{2}}^2 \sderiv{\phi}{3}
                        \biggr)
                         \, \sderiv{\gamma}{3}.
                  \end{align}
                  Since $\sderiv{\phi}{1} \neq 0$,
                  we have
                  $\Span{\sderiv{\tilde{\gamma}}{3}, \sderiv{\tilde{\gamma}}{5}}
                    = \Span{\sderiv{\gamma}{3}, \sderiv{\gamma}{5}}$.
                  Also, since
                  $\sderiv{\gamma}{4} = \lambda \, \sderiv{\gamma}{3}$,
                  $
                    \tilde{\lambda}
                    \coloneqq \lambda \sderiv{\phi}{1} + \sfrac{3 \sderiv{\phi}{2}}{\sderiv{\phi}{1}}
                  $
                  satisfies
                  $\sderiv{\tilde{\gamma}}{4}
                    = \tilde{\lambda} \, \sderiv{\tilde{\gamma}}{3}$.
                  Moreover, we obtain
                  \begin{equation}
                    \sderiv{\tilde{\gamma}}{7}
                      - 2 \, \tilde{\lambda} \, \sderiv{\tilde{\gamma}}{6}
                    \in \paren{\sderiv{\phi}{1}}^7 \,
                        \paren{%
                          \sderiv{\gamma}{7}
                          - 2 \, \lambda \, \sderiv{\gamma}{6}
                        }
                      + \Span{\sderiv{\gamma}{3}, \sderiv{\gamma}{5}},
                  \end{equation}
                  and hence,
                  condition \cref{eq:criteria.3.5.7.and.3.5:condition}
                  is invariant under parameter changes.
                
                (b) We take a coordinate transformation $\varPhi$
                  and consider the curve $\hat{\gamma} \coloneqq \varPhi \circ \gamma$.
                  \cref{cor:Bruno_for_curves} \ref{item:cor_Bruno_curve.coord},
                  and \cref{rem:coord_transf.multi_m}
                  with $m = 3$, give
                  $\sderiv{\hat{\gamma}}{3}
                    = J_{\varPhi} \, \sderiv{\gamma}{3}$,
                  $\sderiv{\hat{\gamma}}{4}
                    = J_{\varPhi} \, \sderiv{\gamma}{4}
                    = \lambda \, J_{\varPhi} \, \sderiv{\gamma}{3}
                    = \lambda \, \sderiv{\hat{\gamma}}{3}$,
                  and
                  $\sderiv{\hat{\gamma}}{5}
                    = J_{\varPhi} \, \sderiv{\gamma}{5}$.
                  Thus,
                  \begin{equation}
                    \label{eq:pf.cri.3.5.7.and.3.5:coord:span.3.5.change}%
                    \Span{\sderiv{\hat{\gamma}}{3}, \sderiv{\hat{\gamma}}{5}}
                    = J_{\varPhi} \, \Span{\sderiv{\gamma}{3}, \sderiv{\gamma}{5}}.
                  \end{equation}
                  Moreover, since
                  $\sderiv{x}{4}_j = \lambda \, \sderiv{x}{3}_j$
                  for $j = 1, \dots, N$,
                  we get
                  \begin{equation}
                    \sderiv{\hat{\gamma}}{6}
                     = J_{\varPhi} \, \sderiv{\gamma}{6}
                        + \frac{1}{2}
                          \sum_{i, j = 1}^{N}
                          \pfrac{^2 \varPhi}{x_i \, \partial x_j} \,
                          \sderiv{x}{3}_i \,
                          \sderiv{x}{3}_j, \quad
                    \sderiv{\hat{\gamma}}{7}
                     = J_{\varPhi} \, \sderiv{\gamma}{7}
                        + \lambda
                          \sum_{i, j = 1}^{N}
                            \pfrac{^2 \varPhi}{x_i \, \partial x_j} \,
                            \sderiv{x}{3}_i \,
                            \sderiv{x}{3}_j.
                  \end{equation}
                  Therefore,
                  $
                    \sderiv{\hat{\gamma}}{7}
                        - 2 \, \lambda \, \sderiv{\hat{\gamma}}{6}
                     = J_{\varPhi}
                        \paren{%
                          \sderiv{\gamma}{7}
                          - 2 \, \lambda \, \sderiv{\gamma}{6}
                        }
                  $
                  gives
                  condition \cref{eq:criteria.3.5.7.and.3.5:condition}
                  is invariant under coordinate transformations
                  by \cref{eq:pf.cri.3.5.7.and.3.5:coord:span.3.5.change},
                  completes the proof of \cref{conj:criteria.3.4.5.and.3.4.and.3.5.7.and.3.5}.
    \end{proof}

    \begin{thm}
        \label{conj:criteria.3.7.etc}%
        Let $\sigma \in \cparen{0, 1}$.
        A curve $\gamma \colon \paren{\R, 0} \to \paren{\R^N, \bzero}$ in $\R^N$ is \A-equivalent to
        \begin{enumerate}[label={(\roman*)}]
            \item $\Cusp{3, 7, 8}  (t)
                = \paren{t^3, t^7, t^8, 0, \dots, 0}$,
            \item $\Cusp{3, 7_{8 \sigma}, 11}  (t)
                = \paren{t^3, t^7 + \sigma \, t^8, t^{11}, 0, \dots, 0}$,
            \item $\Cusp{3, 7_{8 \sigma}} (t)
                = \paren{t^3, t^7 + \sigma \, t^8, 0, \dots, 0}$
        \end{enumerate}
        if and only if there exist $\lambda, \mu \in \R$ such that
        \begin{equation}
            \label{eq:thm.3.7.etc:conditions.no.1-1}%
            \sderiv{\gamma}{1} = \sderiv{\gamma}{2} = \bzero \neq \sderiv{\gamma}{3},
            \quad
            \sderiv{\gamma}{4} = \lambda \, \sderiv{\gamma}{3},
            \quad
            \sderiv{\gamma}{5} = \mu \, \sderiv{\gamma}{3},
            \quad
            \bm{V}_7 \notin \Span{\sderiv{\gamma}{3}}
        \end{equation}
        and respectively,
        \begin{enumerate}[label={(\roman*)}]
            \item \label{item:thm.3.7.etc:type.1}%
                $\bm{V}_8 \notin \Span{\sderiv{\gamma}{3}, \bm{V}_7}$,
            \item\label{item:thm.3.7.etc:type.2}%
                there exist $A_8, B_8 \in \R$
                such that $\bm{V}_8 = A_8 \, \sderiv{\gamma}{3} + B_8 \, \bm{V}_7$,
                $\bm{V}_{11} \notin \Span{\sderiv{\gamma}{3}, \bm{V}_7}$,
                and
                \begin{equation}
                    \label{eq:thm.3.7.etc:conditions.no.1-2}
                    \begin{cases}
                        B_8 - \sfrac{7 \lambda}{3} = 0
                        & (\sigma = 0), \\
                        B_8 - \sfrac{7 \lambda}{3} \neq 0
                        & (\sigma = 1),
                    \end{cases}
                \end{equation}
            \item \label{item:thm.3.7.etc:type.3}%
                there exist $A_8, B_8 \in \R$
                such that $\bm{V}_8 = A_8 \, \sderiv{\gamma}{3} + B_8 \, \bm{V}_7$,
                $\bm{V}_{11} \in \Span{\sderiv{\gamma}{3}, \bm{V}_7}$,
                and \cref{eq:thm.3.7.etc:conditions.no.1-2},
        \end{enumerate}
        where we set
            $\bm{V}_7 \coloneqq \sderiv{\gamma}{7} - 2 \lambda \, \sderiv{\gamma}{6}$,
            $\bm{V}_8 \coloneqq \sderiv{\gamma}{8} - \left(\lambda^2 + 2 \mu \right) \, \sderiv{\gamma}{6}$,
            and
            \begin{align}
                \label{eq:thm.3.7.etc:conditions.no.3-1}
                \bm{V}_{11} \coloneqq
                    \sderiv{\gamma}{11}
                        & - \left( B_8 + \lambda \right) \, \sderiv{\gamma}{10}
                        + 3 \left( \lambda \, B_8 - \mu \right) \, \sderiv{\gamma}{9}
                        - \frac{4}{27} \left(
                            3 \left( 13 \lambda^2 - 6 \mu \right) B_8
                            - 35 \lambda^3
                        \right) \, \bm{V}_8 \\
                    & - \left(
                            2 A_8 + \mu \left( 6 \lambda^2 - \mu \right) B_8 - 7 \lambda \mu^2
                        \right) \, \sderiv{\gamma}{6}
            \end{align}
            in \cref{item:thm.3.7.etc:type.2,item:thm.3.7.etc:type.3}.
    \end{thm}

    \begin{proof}
        We prove the claim following the procedure we outlined in \cref{sec:Introduction}.
        First, we prove the ``if'' part.
        
        \textbf{Step 1}.
          For $\sigma \in \{ 0, 1 \}$,
          we note that 
              $3 \in \cR{1}{\Cusp{3, 7_{8 \sigma}}}$ 
              and
              $13, 14 \in \cR{2}{\Cusp{3, 7_{8 \sigma}}}$,
          since 
          $f_3 \in \Fcns{1}{3}{\Cusp{3, 7_{8 \sigma}}}$
          and 
          $f_{i} \in \Fcns{2}{i}{\Cusp{3, 7_{8 \sigma}}}$
          ($i = 13, 14$),
          where 
          \begin{align}
            f_3 \paren{X_1, X_2, \dots, X_N} & = X_1, \qquad
            f_{13} \paren{X_1, X_2, \dots, X_N}
             = 
            \frac{X_1^2 X_2 - \sigma X_2^2 + 2 \sigma X_1^5}{1 - \sigma X_1}, \\
            f_{14} \paren{X_1, X_2, \dots, X_N}
            & = \frac{X_2^2 - 2 \sigma X_1^5 - \sigma X_1^3 X_2}{1 - \sigma X_1}.
          \end{align}
          Hence, by \cref{prop:cR.and.cNR.of.curve}, we get
          \begin{gather}
            \label{eq:NR.3.7.8}
            \cNR{2}{\Cusp{3, 7, 8}}
            \subset \{ 1, 2, 3, 4, 5, 7, 8 \}, \\
            \label{eq:NR.3.7.etc}
            \cNR{2}{\Cusp{3, 7_{8 \sigma}}}, \,
            \cNR{2}{\Cusp{3, 7_{8 \sigma}, 11}}
             \subset \{ 1, 2, 3, 4, 5, 7, 8, 10, 11 \}.
          \end{gather}

        \textbf{Steps 2 and 3 of \cref{item:thm.3.7.etc:type.1}}.
            If there exist $\lambda, \mu \in \R$ such that \cref{eq:thm.3.7.etc:conditions.no.1-1} and if $\bm{V}_8 \notin \Span{\sderiv{\gamma}{3}, \bm{V}_7}$,
              then the curve $\varGamma \paren{t} = T_1 \, \gamma \paren{t}$ satisfies
            $\sderiv{\varGamma}{1}
                = \sderiv{\varGamma}{2}
                = \bm{0}$, 
                $\sderiv{\varGamma}{3} = \bm{e}_1$,
                $\sderiv{\varGamma}{4} = \lambda \, \bm{e}_1$,
                $\sderiv{\varGamma}{5} = \mu \, \bm{e}_1$,
                $T_1 \, \bm{V}_7 = \bm{e}_2$,
                and $T_1 \, \bm{V}_8 = \bm{e}_3$,
                where $T_1 \in \GL{N, \R}$ 
                is obtained in \cref{lem:lin_change_to_standard_basis}.
              Hence, the curve $\varGamma$ takes the form
              \begin{equation}
                \label{eq:3.7.etc.standard_form}%
                \varGamma \paren{t}
                = \left(%
                      t^3 + \lambda \, t^4 + \mu \, t^5 + \sum_{i = 6}^{11} a_{i} \, t^{i},
                      \sum_{i = 6}^{11} b_{i} \, t^{i},
                      \sum_{i = 6}^{11} c_{i} \, t^{i},
                      0, \dots, 0
                    \right)
                  + \sum_{i = 6}^{11} t^{i} \, \bm{v}_{i}
                  + \smallo{t^{11}}
              \end{equation}
              with constants $a_{i}, b_{i}, c_{i} \in \R$
              and $\bm{v}_{i} \in \Span{\be_4, \dots, \be_N}$
              ($i = 6, \dots, 11$).
              We note that
              \begin{gather}
                  a_7 - 2 \lambda \, a_6 = c_7 - 2 \lambda \, c_6 = a_8 - (\lambda^2 + 2 \mu) \, a_6 = b_8 - (\lambda^2 + 2 \mu) \, b_6 = 0, \\
                  b_7 - 2 \lambda \, b_6 = c_8 - (\lambda^2 + 2 \mu) \, c_6 = 1,
                  \quatext{and} \bm{v}_7 - 2 \lambda \, \bm{v}_6 = \bm{v}_8 - (\lambda^2 + 2 \mu) \, \bm{v}_6 = \bm{0},
              \end{gather}
              since
              $\bm{e}_2 = T_1 \, \bm{V}_7
                = \sderiv{\varGamma}{7} - 2 \, \lambda \, \sderiv{\varGamma}{6}
                = (a_7 - 2 \lambda \, a_6) \, \bm{e}_1
                + (b_7 - 2 \lambda \, b_6) \, \bm{e}_2
                + (c_7 - 2 \lambda \, c_6) \, \bm{e}_3
                + (\bm{v}_7 - 2 \lambda \, \bm{v}_6)$
              and
              $\bm{e}_3 = T_1 \, \bm{V}_8
                = \sderiv{\varGamma}{8} - (\lambda^2 + 2 \mu) \, \sderiv{\varGamma}{6}
                = (a_8 - (\lambda^2 + 2 \mu) \, a_6) \, \bm{e}_1
                + (b_8 - (\lambda^2 + 2 \mu) \, b_6) \, \bm{e}_2
                + (c_8 - (\lambda^2 + 2 \mu) \, c_6) \, \bm{e}_3
                + (\bm{v}_8 - (\lambda^2 + 2 \mu) \, \bm{v}_6)$.
              
              By an argument similar to that in the proof of \cref{thm:criteria.2.3.and.2.5.and.2.7}, there exist $q_1, q_2 \in \R$ and $p_1, p_2, p_3 \in \R$
              such that the parameter change $\phi \paren{s} = s + q_{1} \, s^{2} + q_2 \, s^3$ and the coordinate transformation $\varPhi (x_1, \dots, x_N) = (x_1 + p_1 \, x_2 + p_2 \, x_3, x_2 + p_3 \, x_3, x_3, \dots, x_N)$ yield
              \begin{equation}
                \label{eq:3.7.etc.transformed}%
                    \paren{\varPhi \circ \varGamma \circ \phi} \paren{s}
                    = \left( t^3, t^7, t^8, 0, \dots, 0 \right)
                      + \sum_{i = 6, 9, 10, 11} t^{i} \, \hat{\bm{v}}_{i}
                      + \smallo{t^{11}}
              \end{equation}
              with constants $\hat{\bm{v}}_{i} \in \R^N$ ($i = 6, 9, 10, 11$).
              This curve is \A-equivalent to $\Cusp{3, 7, 8}$ by \cref{thm:A.equivalence.suff.cond.1,eq:NR.3.7.8}.

        \textbf{Steps 2 and 3 of \cref{item:thm.3.7.etc:type.2,item:thm.3.7.etc:type.3}}.
            If there exist $\lambda, \mu, A_8, B_8 \in \R$ such that \cref{eq:thm.3.7.etc:conditions.no.1-1} and
            $\bm{V}_8 = A_8 \, \sderiv{\gamma}{3} + B_8 \, \bm{V}_7$,
              then the curve $\varGamma \paren{t} = T_2 \, \gamma \paren{t}$ satisfies
            $\sderiv{\varGamma}{1}
                = \sderiv{\varGamma}{2}
                = \bm{0}$, 
                $\sderiv{\varGamma}{3} = \bm{e}_1$,
                $\sderiv{\varGamma}{4} = \lambda \, \bm{e}_1$,
                $\sderiv{\varGamma}{5} = \mu \, \bm{e}_1$,
                $T_2 \, \bm{V}_7 = \bm{e}_2$,
                and $T_2 \, \bm{V}_8 = A_8 \, \bm{e}_1 + B_8 \, \bm{e}_2$,
                where $T_2 \in \GL{N, \R}$ 
                is obtained in \cref{lem:lin_change_to_standard_basis}.
              Hence, the curve $\varGamma$ takes the form
              \begin{equation}
                \label{eq:3.7.etc.standard_form}%
                \varGamma \paren{t}
                = \left(%
                      t^3 + \lambda \, t^4 + \mu \, t^5 + \sum_{i = 6}^{11} a_{i} \, t^{i},
                      \sum_{i = 6}^{11} b_{i} \, t^{i},
                      0, \dots, 0
                    \right)
                  + \sum_{i = 6}^{11} t^{i} \, \bm{v}_{i}
                  + \smallo{t^{11}}
              \end{equation}
              with constants $a_{i}, b_{i} \in \R$
              and $\bm{v}_{i} \in \Span{\be_3, \dots, \be_N}$
              ($i = 6, \dots, 11$).
              We note that
              \begin{equation}
                  a_7 - 2 \lambda \, a_6 = 0,
                  \quad b_7 - 2 \lambda \, b_6 = 1,
                  \quatext{and} \bm{v}_7 - 2 \lambda \, \bm{v}_6 = \bm{0},
              \end{equation}
              since $\bm{e}_2 = T_2 \, \bm{V}_7 = (a_7 - 2 \lambda \, a_6) \, \bm{e}_1 + (b_7 - 2 \lambda \, b_6) \, \bm{e}_2 + (\bm{v}_7 - 2 \lambda \, \bm{v}_6)$.
              Also, since $A_8 \, \bm{e}_1 + B_8 \, \bm{e}_2
                = T_2 \, \bm{V}_8
                = (a_8 - (\lambda^2 + 2 \mu) \, a_6) \, \bm{e}_1
                + (b_8 - (\lambda^2 + 2 \mu) \, b_6) \, \bm{e}_2
                + (\bm{v}_8 - (\lambda^2 + 2 \mu) \, \bm{v}_6)$,
              we get
              \begin{equation}
                  A_8 = a_8 - (\lambda^2 + 2 \mu) \, a_6,
                  \quad B_8 = b_8 - (\lambda^2 + 2 \mu) \, b_6,
                  \quatext{and} \bm{v}_8 - (\lambda^2 + 2 \mu) \, \bm{v}_6 = \bm{0}.
              \end{equation}
              
              By an argument similar to that in the proof of \cref{thm:criteria.2.3.and.2.5.and.2.7}, there exist $q_{i} \in \R$ ($i = 1, 2, 4, 5, 8$) and $p_1, \dots, p_4 \in \R$ and $\bm{w}_1 \in \Span{\be_3, \dots, \be_N}$
              such that the parameter change $\phi \paren{s} = s + \sum_{i = 1, 2, 4, 5, 8} q_{i} \, s^{i + 1}$ and the coordinate transformation $\varPhi (x_1, x_2, \dots, x_N) = (x_1 + p_1 \, x_2 + p_2 \, x_1 \, x_2, x_2 + p_3 \, x_1^2 + p_4 \, x_1 \, x_2, x_3, \dots, x_N) + x_1 \, x_2 \, \bm{w}_1$ yield
              \begin{align}
                \label{eq:3.7.etc.transformed}%
                    \paren{\varPhi \circ \varGamma \circ \phi} \paren{s}
                    & = \left(%
                          s^3 + \hat{a}_6 \, s^6 + \hat{a}_9 \, s^9,
                          s^7 + \hat{b}_8 \, s^8 + \hat{b}_9 \, s^9,
                          0, \dots, 0
                        \right) \\
                    & \qquad
                      + \sum_{i = 6, 8, 9, 11} s^{i} \, \hat{\bm{v}}_{i}
                      + \smallo{s^{11}}
              \end{align}
              with constants $\hat{a}_6, \hat{a}_9, \hat{b}_8, \hat{b}_9 \in \R$ and $\hat{\bm{v}}_{i} \in \Span{\be_3, \dots, \be_N}$ ($i = 6, 8, 9, 11$).
              Now, we get
              \begin{gather}
                  \hat{b}_8 = b_8 - (\lambda^2 + 2 \mu) \, b_6 - \frac{7}{3} \lambda
                  = B_8 - \frac{7}{3} \lambda,
                  \qquad \hat{\bm{v}}_8 = \bm{v}_8 - (\lambda^2 + 2 \mu) \, \bm{v}_6 = \bm{0}, \\
                  \hat{\bm{v}}_{11}
                  = \bm{v}_{11}
                        - \left( B_8 + \lambda \right) \, \bm{v}_{10}
                        + 3 \left( \lambda \, B_8 - \mu \right) \, \bm{v}_{9}
                        - \left(
                            2 A_8 + \mu \left( 6 \lambda^2 - \mu \right) B_8 - 7 \lambda \mu^2
                          \right) \, \bm{v}_{6}.
              \end{gather}
              Hence, since $T_2 \, \bm{V}_8 \in \Span{\bm{e}_1, \bm{e}_2}$ and $T_2 \, \bm{v}_i \in \hat{\bm{v}}_i + \Span{\bm{e}_1, \bm{e}_2}$ ($i = 6, 9, 10, 11$), we obtain
              \begin{align}
                  T_2 \, \bm{V}_{11}
                  & \in T_2 \,
                        \left(
                            \bm{V}_{11} + \frac{4}{27} \left(
                                3 \left( 13 \lambda^2 - 6 \mu \right) B_8
                                - 35 \lambda^3
                            \right) \, \bm{V}_8
                        \right)
                    + \Span{\bm{e}_1, \bm{e}_2} \\
                  & = \hat{\bm{v}}_{11} + \Span{\bm{e}_1, \bm{e}_2}.
              \end{align}
              We set
              \begin{gather}
                \varPsi \paren{x_1, x_2, \dots, x_N}
                \coloneqq
                  \begin{cases}
                    \paren{x_1, x_2, \dots, x_N}
                        & (\hat{b}_8 = 0), \\
                    \paren{\hat{b}_8^3 \, x_1, \hat{b}_8^7 \, x_2, \hat{b}_8^{11} \, x_3, x_4, \dots, x_N}
                        & (\hat{b}_8 \neq 0),
                  \end{cases}
                  \quad \text{and} \\
                  \psi \paren{u} 
                  \coloneqq
                  \begin{cases}
                    u & (\hat{b}_8 = 0), \\
                    {u} / {\hat{b}_8} & (\hat{b}_8 \neq 0).
                  \end{cases}
              \end{gather}
          
              (\rnum{2}) If $\bm{V}_{11} \notin \Span{\sderiv{\gamma}{3}, \bm{V}_7}$,
                  then $\hat{\bv}_{11} \neq \bzero$. 
                  Thus, the curve
                  $\hat{\varGamma} \paren{s} \coloneqq T_3 \, \paren{\varPhi \circ \varGamma \circ \phi} \paren{s}$
                  satisfies
                  $\sderiv{\hat{\varGamma}}{1}
                    = \sderiv{\hat{\varGamma}}{2}
                    = \sderiv{\hat{\varGamma}}{4}
                    = \sderiv{\hat{\varGamma}}{5}
                    = \sderiv{\hat{\varGamma}}{10}
                    = \bm{0}$,
                  $\sderiv{\hat{\varGamma}}{3} = \bm{e}_1$,
                  $\sderiv{\hat{\varGamma}}{7} = \bm{e}_2$,
                  and $\sderiv{\hat{\varGamma}}{11} = \bm{e}_3$,
                  where $T_3 \in \GL{N, \R}$ is as in \cref{lem:lin_change_to_standard_basis} from the tuple $(\bm{e}_1, \bm{e}_2, \hat{\bm{v}}_{11})$.
                  Hence, it takes the form
                  \begin{equation}
                      \hat{\varGamma} \paren{s}
                      = \paren{s^3, s^7 + \hat{b}_8 \, s^8, s^{11}, 0, \dots, 0}
                          + s^{6} \, \hat{\bm{w}}_{6}
                          + s^{9} \, \hat{\bm{w}}_{9}
                          + \smallo{s^{11}}
                  \end{equation}
                  with $\hat{\bm{w}}_{6}, \hat{\bm{w}}_{9} \in \R^N$.
                  Hence, the curve $\varPsi \circ \hat{\varGamma} \circ \psi$ takes the form
                  \begin{equation}
                      \paren{\varPsi \circ \hat{\varGamma} \circ \psi} \paren{u}
                      = \paren{u^3, u^7 + \sigma \, u^8, u^{11}, 0, \dots, 0}
                          + u^{6} \, \tilde{\bm{w}}_{6}
                          + u^{9} \, \tilde{\bm{w}}_{9}
                          + \smallo{u^{11}}
                  \end{equation}
                  with constants $\tilde{\bm{w}}_{6}, \tilde{\bm{w}}_{9} \in \R^N$ and
                  \begin{equation}
                      \label{eq:3.7.etc:pf.sigma}
                      \sigma =
                      \begin{cases}
                          0 & (\hat{b}_8 = 0), \\
                          1 & (\hat{b}_8 \neq 0).
                      \end{cases}
                  \end{equation}
                  This curve is \A-equivalent to $\Cusp{3, 7_{8 \sigma}, 11}$ by \cref{thm:general.property.A.equivalence} and \eqref{eq:NR.3.7.etc}.
              
              (\rnum{3}) If $\bm{V}_{11} \in \Span{\sderiv{\gamma}{3}, \bm{V}_7}$,
                  then $\hat{\bv}_{11} = \bzero$. 
                  Thus, the curve
                  $\hat{\varGamma} \paren{s} \coloneqq T_3 \, \paren{\varPhi \circ \varGamma \circ \phi} \paren{s}$ takes the form
                  \begin{equation}
                      \hat{\varGamma} \paren{s}
                      = \paren{s^3, s^7 + \hat{b}_8 \, s^8, 0, \dots, 0}
                          + s^{6} \, \hat{\bm{w}}_{6}
                          + s^{9} \, \hat{\bm{w}}_{9}
                          + \smallo{s^{11}}
                  \end{equation}
                  with $\hat{\bm{w}}_{6}, \hat{\bm{w}}_{9} \in \R^N$, 
                  where $T_3 \in \GL{N, \R}$ is as in \cref{lem:lin_change_to_standard_basis}.
                  Hence, the curve $\varPsi \circ \hat{\varGamma} \circ \psi$ takes the form
                  \begin{equation}
                      \paren{\varPsi \circ \hat{\varGamma} \circ \psi} \paren{u}
                      = \paren{u^3, u^7 + \sigma \, u^8, 0, \dots, 0}
                          + u^{6} \, \tilde{\bm{w}}_{6}
                          + u^{9} \, \tilde{\bm{w}}_{9}
                          + \smallo{u^{11}}
                  \end{equation}
                  with constants $\tilde{\bm{w}}_{6}, \tilde{\bm{w}}_{9} \in \R^N$ and $\sigma \in \{ 0, 1 \}$ as in \cref{eq:3.7.etc:pf.sigma}.
                  This curve is \A-equivalent to $\Cusp{3, 7_{8 \sigma}}$ by \cref{thm:general.property.A.equivalence} and \eqref{eq:NR.3.7.etc}.
                  This completes the proof of the ``if'' part of the theorem.
        
          \textbf{Step 4}.
            Finally, we prove the ``only if'' part. Each of the standard singularities $\Cusp{3, 7, 8}$, $\Cusp{3, 7_{8 \sigma}, 11}$, and $\Cusp{3, 7_{8 \sigma}}$ satisfies its respective condition.
              Hence, we only need to verify that the stated conditions are invariant under \A-equivalence.

              Let us separately show the invariance under
              (a) parameter changes and (b) coordinate transformations.

              (a) Suppose that a curve $\gamma$ satisfies \cref{eq:thm.3.7.etc:conditions.no.1-1}.
                  We take a parameter change $\phi$ and consider the curve $\tilde{\gamma} \coloneqq \gamma \circ \phi$.
                  As we computed in the proof of \cref{conj:criteria.3.4.5.and.3.4.and.3.5.7.and.3.5},
                  we have
                  $\sderiv{\tilde{\gamma}}{1} = \sderiv{\tilde{\gamma}}{2} = \bm{0}$,
                  $\sderiv{\tilde{\gamma}}{3} = \paren{\sderiv{\phi}{1}}^3 \, \sderiv{\gamma}{3} \neq \bm{0}$.
                  Also, since
                  $\sderiv{\gamma}{4} = \lambda \, \sderiv{\gamma}{3}$
                  and $\sderiv{\gamma}{5} = \mu \, \sderiv{\gamma}{3}$,
                  \begin{equation}
                    \tilde{\lambda}
                    \coloneqq \lambda \sderiv{\phi}{1} + \frac{3 \sderiv{\phi}{2}}{\sderiv{\phi}{1}}
                    \quatext{and}
                    \tilde{\mu}
                    \coloneqq \mu (\sderiv{\phi}{1})^2 + 4 \lambda \sderiv{\phi}{2} + \frac{3 \sderiv{\phi}{3}}{\sderiv{\phi}{1}} + 
                    \frac{3 (\sderiv{\phi}{2})^2}{(\sderiv{\phi}{1})^2}
                  \end{equation}
                  satisfy
                  $\sderiv{\tilde{\gamma}}{4} = \tilde{\lambda} \, \sderiv{\tilde{\gamma}}{3}$
                  and $\sderiv{\tilde{\gamma}}{5} = \tilde{\mu} \, \sderiv{\tilde{\gamma}}{3}$.
                  Moreover, $\tilde{\bm{V}}_7 \coloneqq 
                    \sderiv{\tilde{\gamma}}{7}
                      - 2 \, \tilde{\lambda} \, \sderiv{\tilde{\gamma}}{6}$
                  satisfies
                  \begin{equation}
                    \tilde{\bm{V}}_7
                    \in \paren{\sderiv{\phi}{1}}^7 \, \bm{V}_7
                      + \Span{\sderiv{\gamma}{3}},
                  \end{equation}
                  and hence, $\tilde{\gamma}$ also satisfies \cref{eq:thm.3.7.etc:conditions.no.1-1} and we get $\Span{\sderiv{\tilde{\gamma}}{3}, \tilde{\bm{V}}_7} = \Span{\sderiv{\gamma}{3}, \bm{V}_7}$.
                  Additionally, we see
                  $\tilde{\bm{V}}_8
                    \in (\sderiv{\phi}{1})^8 \, \bm{V}_8 + \Span{\sderiv{\gamma}{3}, \bm{V}_7}$,
                  where $\tilde{\bm{V}}_8 \coloneqq \sderiv{\tilde{\gamma}}{8} - (\tilde{\lambda}^2 + 2 \tilde{\mu}) \, \sderiv{\tilde{\gamma}}{6}$.

                  For (i), 
                  if $\bm{V}_8 \notin \Span{\sderiv{\gamma}{3}, \bm{V}_7}$, then $\tilde{\bm{V}}_8 \notin \Span{\sderiv{\tilde{\gamma}}{3}, \tilde{\bm{V}}_7}$.
                  
                  For (ii) and (iii), if there exist $A_8, B_8 \in \R$ such that $\bm{V}_8 = A_8 \, \sderiv{\gamma}{3} + B_8 \, \bm{V}_7$,
                  then there exist $\tilde{A}_8, \tilde{B}_8 \in \R$ such that $\tilde{\bm{V}}_8 = \tilde{A}_8 \, \sderiv{\tilde{\gamma}}{3} + \tilde{B}_8 \, \tilde{\bm{V}}_7$.
                  Then, we have $\tilde{B}_8 = B_8 \sderiv{\phi}{1} + {7 \sderiv{\phi}{2}} / {\sderiv{\phi}{1}}$, and thus, we get $\tilde{B}_8 - {7 \tilde{\lambda}} / {3} = (B_8 - {7 \lambda} / {3}) \, \sderiv{\phi}{1}$. Hence, the condition \cref{eq:thm.3.7.etc:conditions.no.1-2} is invariant under parameter changes.
                  Furthermore, by computing $\sderiv{\tilde{\gamma}}{i}$ ($i = 9, 10, 11$) similarly, we obtain
                  $
                      \tilde{\bm{V}}_{11}
                      \in (\sderiv{\phi}{1})^{11} \, \bm{V}_{11} + \Span{\sderiv{\gamma}{3}, \bm{V}_7},
                  $ 
                  where we denote by $\tilde{\bm{V}}_{11}$ the value of $\bm{V}_{11}$ for $\tilde{\gamma}$ with $\tilde{\lambda}$, $\tilde{\mu}$, $\tilde{A}_8$, and $\tilde{B}_8$. 
                  Therefore, all the stated conditions are invariant under parameter changes.

                (b) Suppose that $\gamma$ satisfies \cref{eq:thm.3.7.etc:conditions.no.1-1}.
                  We take a coordinate transformation $\varPhi$ and consider the curve $\hat{\gamma} \coloneqq \varPhi \circ \gamma$.
                  \cref{cor:Bruno_for_curves} \ref{item:cor_Bruno_curve.coord} and \cref{rem:coord_transf.multi_m} with $m = 3$ give
                  $\sderiv{\hat{\gamma}}{3}
                    = J_{\varPhi} \, \sderiv{\gamma}{3}$
                  $\sderiv{\hat{\gamma}}{4}
                    = J_{\varPhi} \, \sderiv{\gamma}{4}
                    = \lambda \, \sderiv{\hat{\gamma}}{3}$,
                  and
                  $\sderiv{\hat{\gamma}}{5}
                    = J_{\varPhi} \, \sderiv{\gamma}{5}
                    = \mu \, \sderiv{\hat{\gamma}}{3}$.
                  Moreover, since
                  $\sderiv{x}{4}_j = \lambda \, \sderiv{x}{3}_j$
                  and $\sderiv{x}{5}_j = \mu \, \sderiv{x}{3}_j$
                  for $j = 1, \dots, N$,
                  we get
                  \begin{gather}
                    \sderiv{\hat{\gamma}}{6}
                    = J_{\varPhi} \, \sderiv{\gamma}{6}
                        + \frac{1}{2}
                          \sum_{i, j = 1}^{N}
                          \pfrac{^2 \varPhi}{x_i \, \partial x_j} \,
                          \sderiv{x}{3}_i \, \sderiv{x}{3}_j,
                    \quad
                    \sderiv{\hat{\gamma}}{7}
                    = J_{\varPhi} \, \sderiv{\gamma}{7}
                        + \lambda
                          \sum_{i, j = 1}^{N}
                            \pfrac{^2 \varPhi}{x_i \, \partial x_j} \,
                            \sderiv{x}{3}_i \, \sderiv{x}{3}_j, \\
                    \sderiv{\hat{\gamma}}{8}
                    = J_{\varPhi} \, \sderiv{\gamma}{8}
                        + \left( \frac{1}{2} \lambda^2 + \mu \right)
                          \sum_{i, j = 1}^{N}
                            \pfrac{^2 \varPhi}{x_i \, \partial x_j} \,
                            \sderiv{x}{3}_i \, \sderiv{x}{3}_j,
                  \end{gather}
                  where $\gamma (t) = (x_1 (t), \dots, x_N (t))$.
                  Therefore, we get
                  $\hat{\bm{V}}_7 = J_{\varPhi} \bm{V}_7$
                  and $\hat{\bm{V}}_8 = J_{\varPhi} \bm{V}_8$,
                  where we denote by $\hat{\bm{V}}_{i}$ the values of $\bm{V}_{i}$ for $\hat{\gamma}$ ($i = 7, 8, 11$).
                  Hence, the conditions \cref{eq:thm.3.7.etc:conditions.no.1-1} and $\bm{V}_8 \notin \Span{\sderiv{\gamma}{3}, \bm{V}_7}$ are invariant under coordinate transformations.

                  Next, we suppose that there exist $A_8, B_8 \in \R$ such that $\bm{V}_8 = A_8 \, \sderiv{\gamma}{3} + B_8 \, \bm{V}_7$.
                  Then, we have $\hat{\bm{V}}_8 = A_8 \, \sderiv{\hat{\gamma}}{3} + B_8 \, \hat{\bm{V}}_7$.
                  Furthermore, we have
                  \begin{align}
                    \sderiv{\hat{\gamma}}{8}
                    & = B_8 \, J_{\varPhi} \, \sderiv{\gamma}{7}
                        + \left( \lambda^2 + 2 \mu - 2 \lambda B_8 \right) J_{\varPhi} \, \sderiv{\gamma}{6}
                        + A_8 \, J_{\varPhi} \, \sderiv{\gamma}{3} \\
                    & \hspace{12em}
                        + \left( \frac{1}{2} \lambda^2 + \mu \right)
                          \sum_{i, j = 1}^{N}
                            \pfrac{^2 \varPhi}{x_i \, \partial x_j} \,
                            \sderiv{x}{3}_i \, \sderiv{x}{3}_j, \\
                    \sderiv{\hat{\gamma}}{9}
                    & = J_{\varPhi} \, \sderiv{\gamma}{9}
                        + \sum_{i, j = 1}^{N}
                            \pfrac{^2 \varPhi}{x_i \, \partial x_j} \,
                            \sderiv{x}{3}_i \, \sderiv{x}{6}_j
                        + \lambda \, \mu
                          \sum_{i, j = 1}^{N}
                            \pfrac{^2 \varPhi}{x_i \, \partial x_j} \,
                            \sderiv{x}{3}_i \, \sderiv{x}{3}_j \\
                    & \hspace{12em}
                        + \frac{1}{6}
                          \sum_{i, j, k = 1}^{N}
                            \pfrac{^3 \varPhi}{x_i \, \partial x_j \, \partial x_k} \,
                            \sderiv{x}{3}_i \, \sderiv{x}{3}_j \, \sderiv{x}{3}_k, \\
                    \sderiv{\hat{\gamma}}{10}
                    & = J_{\varPhi} \, \sderiv{\gamma}{10}
                        + \sum_{i, j = 1}^{N}
                            \pfrac{^2 \varPhi}{x_i \, \partial x_j} \,
                            \sderiv{x}{3}_i \, \sderiv{x}{7}_j
                        + \lambda
                          \sum_{i, j = 1}^{N}
                            \pfrac{^2 \varPhi}{x_i \, \partial x_j} \,
                            \sderiv{x}{3}_i \, \sderiv{x}{6}_j \\
                    & \hspace{3em}
                        + \frac{1}{2} \mu^2
                          \sum_{i, j = 1}^{N}
                            \pfrac{^2 \varPhi}{x_i \, \partial x_j} \,
                            \sderiv{x}{3}_i \, \sderiv{x}{3}_j
                        + \frac{1}{2} \lambda
                          \sum_{i, j, k = 1}^{N}
                            \pfrac{^3 \varPhi}{x_i \, \partial x_j \, \partial x_k} \,
                            \sderiv{x}{3}_i \, \sderiv{x}{3}_j \, \sderiv{x}{3}_k, \\
                    \sderiv{\hat{\gamma}}{11}
                    & = J_{\varPhi} \, \sderiv{\gamma}{11}
                        + \left( B_8 + \lambda \right)
                          \sum_{i, j = 1}^{N}
                            \pfrac{^2 \varPhi}{x_i \, \partial x_j} \,
                            \sderiv{x}{3}_i \, \sderiv{x}{7}_j \\
                    & \hspace{3em}
                        + \left( \lambda^2 + 3 \mu - 2 \lambda \, B_8 \right)
                          \sum_{i, j = 1}^{N}
                            \pfrac{^2 \varPhi}{x_i \, \partial x_j} \,
                            \sderiv{x}{3}_i \, \sderiv{x}{6}_j
                        + A_8
                          \sum_{i, j = 1}^{N}
                            \pfrac{^2 \varPhi}{x_i \, \partial x_j} \,
                            \sderiv{x}{3}_i \, \sderiv{x}{3}_j \\
                    & \hspace{3em}
                        + \frac{1}{2} \left( \lambda^2 + \mu \right)
                          \sum_{i, j, k = 1}^{N}
                            \pfrac{^3 \varPhi}{x_i \, \partial x_j \, \partial x_k} \,
                            \sderiv{x}{3}_i \, \sderiv{x}{3}_j \, \sderiv{x}{3}_k,
                  \end{align}
                  and hence, $\hat{\bm{V}}_{11} = J_{\varPhi} \bm{V}_{11}$.
                  Therefore, all the stated conditions are invariant under coordinate transformations.
                  This completes the proof of \cref{conj:criteria.4.5.and.other}.
    \end{proof}

  \subsection{Singularities of Multiplicity $4$ in \Rpdf[N]}
  \label{sec:Singularities.multiplicity.4}%

    In this subsection,
    we construct criteria for classifying singularities of multiplicity $4$,
    in particular,
    $\paren{4, 5}$-cusps,
    $\paren{4, 5_{\pm 7}}$-cusps
    $\paren{4, 5, 6}$-cusps,
    $\paren{4, 5, 7}$-cusps,
    $\paren{4, 5, 11}$-cusps,    
    $\paren{4, 5_{\pm 7}, 11}$-cusps, and
    $\paren{4, 5, 6, 7}$-cusps.
    
    \begin{thm}
        \label{conj:criteria.4.5.and.other}%
        A curve $\gamma \colon \paren{\R, 0} \to \paren{\R^N, \bzero}$
        in $\R^N$ is \A-equivalent to
        \begin{enumerate}[label={(\arabic*)}]
            \item $\Cusp{4, 5, 6, 7} (t)
                = \paren{t^4, t^5, t^6, t^7, 0, \dots, 0}$,
            \item $\Cusp{4, 5, 6}  (t)
                = \paren{t^4, t^5, t^6, 0, \dots, 0}$,
            \item $\Cusp{4, 5, 7}  (t)
                = \paren{t^4, t^5, t^7, 0, \dots, 0}$,
            \item 
                \begin{enumerate}[label={(\roman*)}]
                    \item $\Cusp{4, 5_{7 \sigma}, 11} (t) = \paren{t^4, t^5 + \sigma \, t^7, t^{11}, 0, \dots, 0} \quad
                        (\sigma \in \cparen{0, \pm 1})$,
                    \item $\Cusp{4, 5_{7 \sigma}}(t) = \paren{t^4, t^5 + \sigma \, t^7, 0, \dots, 0} \quad
                        (\sigma \in \cparen{0, \pm 1})$,
                \end{enumerate}
        \end{enumerate}
        if and only if
        \begin{equation}
            \sderiv{\gamma}{1}
            = \sderiv{\gamma}{2}
            = \sderiv{\gamma}{3}
            = \bzero
            \neq \sderiv{\gamma}{4},
            \quad
            \sderiv{\gamma}{5}
            \notin \Span{\sderiv{\gamma}{4}},
            \label{eq:criteria.45like}
        \end{equation}
        and respectively,
        \begin{enumerate}[label={(\arabic*)}]
            \item %
                $\sderiv{\gamma}{6}
                    \notin \Span{\sderiv{\gamma}{4}, \sderiv{\gamma}{5}}$,
                and
                $\sderiv{\gamma}{7}
                    \notin \Span{\sderiv{\gamma}{4}, \sderiv{\gamma}{5}, \sderiv{\gamma}{6}}$,
            \item %
                $\sderiv{\gamma}{6}
                    \notin \Span{\sderiv{\gamma}{4}, \sderiv{\gamma}{5}}$,
                and
                $\sderiv{\gamma}{7}
                    \in \Span{\sderiv{\gamma}{4}, \sderiv{\gamma}{5}, \sderiv{\gamma}{6}}$,
            \item %
                $\sderiv{\gamma}{6}
                    \in \Span{\sderiv{\gamma}{4}, \sderiv{\gamma}{5}}
                    \notni \sderiv{\gamma}{7}$,
            \item %
                there exist
                $\lambda_2, \lambda_3, \mu_1, \mu_2
                    \in \R$
                such that
                \begin{equation}
                    \label{eq:thm.criteria.4.5.like:defn.lambda.mu}
                    \sderiv{\gamma}{6}
                    = \lambda_2 \, \sderiv{\gamma}{4}
                        + \mu_1 \, \sderiv{\gamma}{5},
                    \quad
                    \sderiv{\gamma}{7}
                    = \lambda_3 \, \sderiv{\gamma}{4}
                        + \mu_2 \, \sderiv{\gamma}{5},
                    \quad
                    \sigma_{\paren{4, 5_{\pm 7}}}
                    = \mathemph{\sigma},
                \end{equation}
                and respectively,
                \begin{enumerate*}[label={(\roman*)}]
                    \item $\V_{\paren{4, 5, 11}}
                        \notin \Span{\sderiv{\gamma}{4}, \sderiv{\gamma}{5}}$,
                    \item $\V_{\paren{4, 5, 11}}
                        \in \Span{\sderiv{\gamma}{4}, \sderiv{\gamma}{5}}$,
                \end{enumerate*}
          where
          \begin{equation}
            \label{eq:defn.45etc.Vector4511}%
            \V_{\paren{4, 5, 11}}
            \coloneqq
              \sderiv{\gamma}{11}
              - 2 \, \mu_1 \, \sderiv{\gamma}{10}
              + \paren{-\lambda_2 + 2 \, \mu_1^2 - \mu_2} \, \sderiv{\gamma}{9}
              + 2 \, \paren{2 \, \lambda_2 \, \mu_1 - \lambda_3} \, \sderiv{\gamma}{8}
          \end{equation}
          and
            $\sigma_{(4, 5_{\pm 7})}
            \coloneqq \sgn (\mu_2 - 5\lambda_2 / 4 - 11\mu_1^2 / 10)
            \in \cparen{0, \pm 1}$.
        \end{enumerate}
    \end{thm}

    \begin{proof}
        We prove the claim following the procedure we outlined in \cref{sec:Introduction}.
        First, we show the ``if'' part of the claim.
        
        \textbf{Step 1}.
        For $\sigma \in \{ 0, \pm 1 \}$,
        we note that 
              $4 \in \cR{1}{\Cusp{4, 5_{7 \sigma}}}$ 
              and
              $13, 10, 15 \in \cR{2}{\Cusp{4, 5_{7 \sigma}}}$,
          since 
          $f_4 \in \Fcns{1}{4}{\Cusp{4, 5_{7 \sigma}}}$
          and 
          $f_{i} \in \Fcns{2}{i}{\Cusp{4, 5_{7 \sigma}}}$
          ($i = 13, 10, 15$),
          where 
          \begin{align}
            f_4 \paren{X_1, X_2, \dots, X_N} & = X_1, \\
            f_{13} \paren{X_1, X_2, \dots, X_N}
            & = 
            \left(X_1^2 - \frac{\sigma \paren{X_2^2 - 3 \sigma X_1^3 - \sigma X_1^4}}{1 - \sigma^2 X_1^2}\right) X_2, \\
            f_{10} \paren{X_1, X_2, \dots, X_N}
            & = \frac{X_2^2 - 2 \sigma X_1^3}{1 + \sigma^2 X_1}, \\
            f_{15} \paren{X_1, X_2, \dots, X_N}
            & = \frac{\paren{X_2^2 - 3 \sigma X_1^3 - \sigma X_1^4} X_2}{1 - \sigma^2 X_1^2}.
          \end{align}
          Hence, by \cref{prop:cR.and.cNR.of.curve}, we get
          \begin{gather} \label{eq:NR.mult.4}
              \cNR{2}{\Cusp{4, 5, 6, 7}}, \,
                \cNR{2}{\Cusp{4, 5, 6}}, \,
                \cNR{2}{\Cusp{4, 5, 7}}
                 \subset \cparen{1, 2, 3, 4, 5, 6, 7}, \\
                \cNR{2}{\Cusp{4, 5, 11}}, \,
                \cNR{2}{\Cusp{4, 5}}
                 \subset \cparen{1, 2, 3, 4, 5, 6, 7, 11}, \quad \text{and}  \\
                \cNR{2}{\Cusp{4, 5_{\pm 7}, 11}}, \,
                \cNR{2}{\Cusp{4, 5_{\pm 7}}}
                 \subset \cparen{1, 2, 3, 4, 5, 6, 7, 9, 11}.
          \end{gather}
            
          \textbf{Steps 2 and 3}.
            (1) If $\gamma$ satisfies
            \begin{equation}
              \sderiv{\gamma}{6}
                \notin \Span{\sderiv{\gamma}{4}, \sderiv{\gamma}{5}}
              \quad \text{and} \quad \sderiv{\gamma}{7}
                \notin \Span{\sderiv{\gamma}{4}, \sderiv{\gamma}{5}, \sderiv{\gamma}{6}},
            \end{equation}
              then the curve
              $\varGamma \paren{t}
                = T_1 \, \gamma \paren{t}$,
              obtained by  $T_1 \in \GL{N, \R}$ defined in \cref{lem:lin_change_to_standard_basis},
              satisfies
                $\sderiv{\varGamma}{1}
                = \sderiv{\varGamma}{2}
                = \sderiv{\varGamma}{3}
                = \bzero
                \quatext{and}
                \sderiv{\varGamma}{i}
                = \be_{i - 3}
                \quad
                (4 \leq i \leq 7)$.
              Hence, the curve $\varGamma$ takes the form 
                $\varGamma \paren{t}
                = \paren{t^4, t^5, t^6, t^7, 0, \dots, 0}
                  + \smallo{t^7}$.
              This curve $\varGamma$ is \A-equivalent to $\Cusp{4, 5, 6, 7}$ by
              \cref{thm:A.equivalence.suff.cond.1} and \eqref{eq:NR.mult.4}.
            
            (2) If $\gamma$ satisfies
              $\sderiv{\gamma}{6}
                \notin \Span{\sderiv{\gamma}{4}, \sderiv{\gamma}{5}}$
              and $\sderiv{\gamma}{7}
                \in \Span{\sderiv{\gamma}{4}, \sderiv{\gamma}{5}, \sderiv{\gamma}{6}}$,
              then the curve
              $\varGamma \paren{t}
                = T_2 \, \gamma \paren{t}$
              satisfies
                $\sderiv{\varGamma}{1}
                = \sderiv{\varGamma}{2}
                = \sderiv{\varGamma}{3}
                = \bzero$, and 
                $\sderiv{\varGamma}{i}
                = \be_{i - 3}$ 
                $(i = 4, 5, 6)$, and
                $\sderiv{\varGamma}{7}
                \in \Span{\be_1, \be_2, \be_3}$,
                where $T_2 \in \GL{N,\R}$ is given in \cref{lem:lin_change_to_standard_basis}.
              Hence, the curve $\varGamma$ takes the form
              \begin{equation}
                \label{eq:456.standard_form}%
                \varGamma \paren{t}
                = \paren{t^4 + a \, t^7, t^5 + b \, t^7, t^6 + c \, t^7, 0, \dots, 0}
                  + \smallo{t^7}
              \end{equation}
              with constants $a, b, c \in \R$.
              By an argument similar to that in the proof of \cref{thm:criteria.2.3.and.2.5.and.2.7}, there exist $c_1, c_2, c_3, p_1, p_2, p_3 \in \R$ such that the parameter change
                $
                \phi \paren{s}
                = s + c_{1} \, s^2 + c_{2} \, s^3 + c_{3} \, s^4
              $
              and the coordinate transformation
              $
                \varPhi \paren{x_1, x_2, \dots, x_N}
                = \paren{%
                      x_1 + p_1 \, x_2 + p_2 \, x_3, \,
                      x_2 + p_3 \, x_3, \,
                      x_3, \, \dots, \, x_N
                    }
              $
              give
              $
                \paren{\varPhi \circ \varGamma \circ \phi} \paren{s}
                = \paren{s^4, s^5, s^6, 0, \dots, 0}
                  + \smallo{s^7}.
              $
              This curve is \A-equivalent to $\Cusp{4, 5, 6}$ by
              \cref{thm:A.equivalence.suff.cond.1} and \eqref{eq:NR.mult.4}.
            
            (3) If $\gamma$ satisfies
              $\sderiv{\gamma}{6}
                \in \Span{\sderiv{\gamma}{4}, \sderiv{\gamma}{6}}
                \notni \sderiv{\gamma}{7}$,
              then, by using a matrix $T_3 \in \GL{N,\R}$ in \cref{lem:lin_change_to_standard_basis},
              $\varGamma \paren{t}
                = T_3 \, \gamma \paren{t}$
              satisfies
                $\sderiv{\varGamma}{1}
                = \sderiv{\varGamma}{2}
                = \sderiv{\varGamma}{3}
                = \bzero$, 
                $\sderiv{\varGamma}{4} = \be_1$, 
                $\sderiv{\varGamma}{5} = \be_2$, 
                $\sderiv{\varGamma}{6} \in \Span{\be_1, \be_2}$, and
                $\sderiv{\varGamma}{7} = \be_3$.
              Hence, the curve $\varGamma$ takes the form
              \begin{equation}
                \label{eq:457.standard_form}%
                \varGamma \paren{t}
                = \paren{t^4 + a \, t^6, t^5 + b \, t^6, t^7, 0, \dots, 0}
                  + \smallo{t^7}
              \end{equation}
              with constants $a, b \in \R$.
              There exist $c_1, c_2, p_1, p_2, p_3 \in \R$ such that the parameter change
              $
                \phi \paren{s}
                = s + c_{1} \, s^2 + c_{2} \, s^3
              $
              and the coordinate transformation
              $
                \varPhi \paren{x_1, x_2, \dots, x_N}
                = \paren{%
                      x_1 + p_1 \, x_2 + p_2 \, x_3, \,
                      x_2 + p_3 \, x_3, \,
                      x_3, \, \dots, \, x_N
                    }
              $
              yield
              $
                \paren{\varPhi \circ \varGamma \circ \phi} \paren{s}
                = \paren{s^4, s^5, s^7, 0, \dots, 0}
                  + \smallo{s^7}.
              $
              This curve is \A-equivalent to $\Cusp{4, 5, 7}$ by
              \cref{thm:A.equivalence.suff.cond.1} and \eqref{eq:NR.mult.4}.
            
            (4) If there exist $\lambda_2, \lambda_3, \mu_1, \mu_2 \in \R$ such that
                $\sderiv{\gamma}{i} = \lambda_{i - 4} \, \sderiv{\gamma}{4} + \mu_{i - 5} \, \sderiv{\gamma}{5}$ 
                $(i = 6, 7)$,
              then the curve
              $\varGamma \paren{t} = T_4 \, \gamma \paren{t}$
              satisfies
                $\sderiv{\varGamma}{1}
                = \sderiv{\varGamma}{2}
                = \sderiv{\varGamma}{3}
                = \bm{0}$, 
                $\sderiv{\varGamma}{4} = \bm{e}_1$,
                $\sderiv{\varGamma}{5} = \bm{e}_2$,
                $\sderiv{\varGamma}{i}
                = \lambda_{i - 4} \, \be_1
                  + \mu_{i - 5} \, \be_2$ 
                $(i = 6, 7)$,
                where $T_4 \in \GL{N, \R}$ 
                is obtained in \cref{lem:lin_change_to_standard_basis}.
              Hence, the curve $\varGamma$ takes the form
              \begin{equation}
                \label{eq:45etc.standard_form}%
                \varGamma \paren{t}
                = \left(%
                      t^4 + \sum_{i = 6}^{11} a_{i - 4} \, t^{i},
                      t^5 + \sum_{i = 6}^{11} b_{i - 5} \, t^{i},
                      0, \dots, 0
                    \right)
                  + \sum_{i = 8}^{11} t^i \, \bv_i
                  + \smallo{t^{11}}
              \end{equation}
              with constants $a_2, \dots, a_7, b_1, \dots, b_6 \in \R$
              and $\bv_8, \dots, \bv_{11}
                \in \Span{\be_3, \dots, \be_N}$,
              where $a_2 = \lambda_2$, $a_3 = \lambda_3$,
              $b_1 = \mu_1$, and $b_2 = \mu_2$.
          By an argument similar to that in the proof of \cref{thm:criteria.2.3.and.2.5.and.2.7}, there exist $c_1, c_2, c_3, c_6, c_7, p_1, p_2, p_3 \in \R$ and $\bm{w}_9 \in \Span{\be_3, \dots, \be_N}$ such that the parameter change
            $
            \phi \paren{s}
            = s + \sum_{i = 1, 2, 3, 6, 7} c_{i} \, s^{i + 1}
          $
          and the coordinate transformation
          \begin{equation}
            \varPhi \paren{x_1, x_2, \dots, x_N}
            = \paren{%
                  x_1 + p_1 \, x_2 + p_2 \, x_1 x_2, \,
                  x_2 + p_3 \, x_1 x_2, \,
                  x_3, \, \dots, \, x_N
                }
                + x_1 x_2 \, \bm{w}_9
          \end{equation}
          yield
          \begin{equation}
            \paren{\varPhi \circ \varGamma \circ \phi} \paren{s}
            = \paren{s^4, s^5 + B_2 \, s^7, 0, \dots, 0}
              + s^{8} \, \bm{w}_{8}
              + s^{10} \, \bm{w}_{10}
              + s^{11} \, \bm{w}_{11}
              + \smallo{s^{11}}
          \end{equation}
          with $B_2 \in \R$, $\bm{w}_8, \bm{w}_{10} \in \R^N$, and $\bm{w}_{11} \in \Span{\be_3, \dots, \be_N}$.
          Then, we get
          $
              B_2 = \mu_2 - 5 \lambda_2/4 - 11 \mu_1^2/10
          $
          and
          \begin{equation}
              \bm{w}_{11}
              = \bm{v}_{11}
                  - 2 \, \mu_1 \, \bm{v}_{10}
                  + \paren{-\lambda_2 + 2 \, \mu_1^2 - \mu_2} \, \bm{v}_{9}
                  + 2 \, \paren{2 \, \lambda_2 \, \mu_1 - \lambda_3} \, \bm{v}_{8}.
          \end{equation}
          We note that $T_4 \, \V_{(4, 5, 11)} \in \bm{w}_{11} + \Span{\bm{e}_1, \bm{e}_2}$.
          We set $p \coloneqq \sqrt{\abs{B_2}}$, 
          \begin{gather}
            \varPsi \paren{x_1, x_2, \dots, x_N}
            \coloneqq
              \begin{cases}
                \paren{x_1, x_2, \dots, x_N}
                    & (B_2 = 0), \\
                \paren{p^4 \, x_1, p^5 \, x_2, p^{11} \, x_3, x_4, \dots, x_N}
                    & (B_2 \neq 0),
              \end{cases}
              \quad \text{and} \\
              \psi \paren{u} 
              \coloneqq
              \begin{cases}
                u & (B_2 = 0), \\
                \sfrac{u}{p} & (B_2 \neq 0).
              \end{cases}
          \end{gather}
          We consider two cases: (\rnum{1}) and (\rnum{2}).
          
              (\rnum{1}) If $\V_{\paren{4, 5, 11}}
                  \notin \Span{\sderiv{\gamma}{4}, \sderiv{\gamma}{5}}$,
                  then $\bm{w}_{11} \neq \bm{0}$.
                  Thus, the curve
                  $\hat{\varGamma} \paren{s} \coloneqq T_5 \, \paren{\varPhi \circ \varGamma \circ \phi} \paren{s}$ takes the form
                  \begin{equation}
                      \hat{\varGamma} \paren{s}
                      = \paren{s^4, s^5 + B_2 \, s^7, s^{11}, 0, \dots, 0}
                          + s^{8} \, \hat{\bm{w}}_{8}
                          + s^{10} \, \hat{\bm{w}}_{10}
                          + \smallo{s^{11}}
                  \end{equation}
                  with $\hat{\bm{w}}_{8}, \hat{\bm{w}}_{10} \in \R^N$, 
                  where $T_5 \in \GL{N, \R}$ is as in \cref{lem:lin_change_to_standard_basis}.
                  Hence, the curve $\paren{\varPsi \circ \hat{\varGamma} \circ \psi} \paren{u}$ takes the form
                  \begin{equation}
                      \paren{\varPsi \circ \hat{\varGamma} \circ \psi} \paren{u}
                      = \paren{u^4, u^5 + \paren{\sgn B_2} \, u^7, u^{11}, 0, \dots, 0}
                          + u^{8} \, \tilde{\bm{w}}_{8}
                          + u^{10} \, \tilde{\bm{w}}_{10}
                          + \smallo{u^{11}}
                  \end{equation}
                  with constants $\tilde{\bm{w}}_{8}, \tilde{\bm{w}}_{10} \in \R^N$.
                  This curve is \A-equivalent to the curve $t \mapsto \paren{t^4, t^5 + \sigma_{\paren{4, 5_{\pm 7}}} \, t^7, t^{11}, 0, \dots, 0}$ by \cref{thm:A.equivalence.suff.cond.1} and \eqref{eq:NR.mult.4}, 
                  since $\sgn B_2 = \sigma_{\paren{4, 5_{\pm 7}}}$.
              
              (\rnum{2}) If $\V_{\paren{4, 5, 11}}
                  \in \Span{\sderiv{\gamma}{4}, \sderiv{\gamma}{5}}$,
                  then $\bm{w}_{11} = \bm{0}$.
                  Thus, the curve
                  $\paren{
                  \varPsi \circ \varPhi \circ \varGamma \circ \phi \circ \psi
                  } \paren{u}$ 
                  takes the form
                  \begin{equation}
                      \paren{\varPsi \circ \varPhi \circ \varGamma \circ \phi \circ \psi} \paren{u}
                      = \paren{u^4, u^5 + \paren{\sgn B_2} \, u^7, \dots, 0}
                          + u^{8} \, \tilde{\bm{w}}_{8}
                          + u^{10} \, \tilde{\bm{w}}_{10}
                          + \smallo{u^{11}}
                  \end{equation}
                  with constants $\tilde{\bm{w}}_{8}, \tilde{\bm{w}}_{10} \in \R^N$.
                  This curve is \A-equivalent to the curve $t \mapsto \paren{t^4, t^5 + \sigma_{\paren{4, 5_{\pm 7}}} \, t^7, 0, \dots, 0}$ by 
                  \cref{thm:A.equivalence.suff.cond.1} and \eqref{eq:NR.mult.4}.
                  This completes the proof of the ``if'' part of the theorem.
        
          \textbf{Step 4}.
          Finally, we prove the ``only if'' part of the claim.
          Each of the standard singularities
          $\Cusp{4, 5, 6, 7}$, $\Cusp{4, 5, 6}$, $\Cusp{4, 5, 7}$,
          $\Cusp{4, 5, 11}$, $\Cusp{4, 5_{\pm 7}, 11}$,
          $\Cusp{4, 5}$, and $\Cusp{4, 5_{\pm 7}}$
          satisfies its respective conditions.
          Hence, we only need to verify that the stated conditions are invariant under \A-equivalence.
          Most of the conditions stated in \cref{conj:criteria.4.5.and.other} the claim are immediately seen to be invariant under \A-equivalence from \cref{prop:multi_and_lin_dep.invariant}.
          It remains to verify the invariance of the value of $\sigma_{\paren{4, 5_{\pm 7}}}$ and the condition $\V_{\paren{4, 5, 11}} \notin \Span{\sderiv{\gamma}{4}, \sderiv{\gamma}{5}}$ under \cref{eq:defn.45etc.Vector4511,eq:criteria.45like,eq:thm.criteria.4.5.like:defn.lambda.mu}.
          Suppose that $\gamma$ satisfies
          \cref{eq:criteria.45like,eq:thm.criteria.4.5.like:defn.lambda.mu}.
          Let us separately show the invariance under
          (a) parameter changes and (b) coordinate transformations.

        (a)   We take a parameter change $\phi$
              and consider the curve $\tilde{\gamma} \coloneqq \gamma \circ \phi$.
              We can compute as follows: 
              \begin{align}
                \sderiv{\tilde{\gamma}}{4}
                & = \paren{\sderiv{\phi}{1}}^4 \, \sderiv{\gamma}{4},
                \qquad
                \sderiv{\tilde{\gamma}}{5}
                = \paren{\sderiv{\phi}{1}}^5 \, \sderiv{\gamma}{5}
                  + 4 \sderiv{\phi}{2} \paren{\sderiv{\phi}{1}}^3 \, \sderiv{\gamma}{4}, \\
                \sderiv{\tilde{\gamma}}{6}
                & = \paren{\sderiv{\phi}{1}}^4\left(\mu_1 \paren{\sderiv{\phi}{1}}^2+5 \sderiv{\phi}{2}\right) \, \sderiv{\gamma}{5} \\
                & \qquad
                  + \paren{\sderiv{\phi}{1}}^2 \left(\lambda_2 \paren{\sderiv{\phi}{1}}^4+4 \sderiv{\phi}{3} \sderiv{\phi}{1}+6 \paren{\sderiv{\phi}{2}}^2\right) \, \sderiv{\gamma}{4}, \\
                \sderiv{\tilde{\gamma}}{7}
                & = \paren{\sderiv{\phi}{1}}^3 \left(\mu_2 \paren{\sderiv{\phi}{1}}^4+6 \mu_1 \sderiv{\phi}{2} \paren{\sderiv{\phi}{1}}^2+5 \sderiv{\phi}{3} \sderiv{\phi}{1}+10 \paren{\sderiv{\phi}{2}}^2\right) \, \sderiv{\gamma}{5} \\
                & \qquad
                    + \sderiv{\phi}{1}
                      \biggl(
                          \lambda_3 \paren{\sderiv{\phi}{1}}^6+6 \lambda_2 \sderiv{\phi}{2} \paren{\sderiv{\phi}{1}}^4 \\
                  & \qquad \qquad \qquad
                          +4 \left(\paren{\sderiv{\phi}{2}}^3+3 \sderiv{\phi}{1} \sderiv{\phi}{3}\sderiv{\phi}{2}+\paren{\sderiv{\phi}{1}}^2 \sderiv{\phi}{4}\right)
                      \biggr)
                      \, \sderiv{\gamma}{4}.
              \end{align}
              Since $\sderiv{\phi}{1} \neq 0$,
              we have
              $\Span{\sderiv{\tilde{\gamma}}{4}, \sderiv{\tilde{\gamma}}{5}}
                = \Span{\sderiv{\gamma}{4}, \sderiv{\gamma}{5}}$.
              Also, since
              $\sderiv{\gamma}{i}
                = \lambda_{i - 4} \, \sderiv{\gamma}{4}
                  + \mu_{i - 5} \, \sderiv{\gamma}{5}$
              for $i = 6, 7$,
              \begin{align}
                \tilde{\lambda}_2
                & \coloneqq \lambda_2 \paren{\sderiv{\phi}{1}}^2-4 \mu_1 \sderiv{\phi}{2}+\frac{2 \left(2 \sderiv{\phi}{1} \sderiv{\phi}{3}-7 \paren{\sderiv{\phi}{2}}^2\right)}{\paren{\sderiv{\phi}{1}}^2}, \\
                \tilde{\lambda}_3
                & \coloneqq \lambda_3 \paren{\sderiv{\phi}{1}}^3+6 \lambda_2 \sderiv{\phi}{2} \sderiv{\phi}{1}+\frac{4 \sderiv{\phi}{4}}{\sderiv{\phi}{1}} \\
                & \qquad
                  -\frac{4 \sderiv{\phi}{2} \left(\mu_2 \paren{\sderiv{\phi}{1}}^4+6 \mu_1 \sderiv{\phi}{2} \paren{\sderiv{\phi}{1}}^2+2 \sderiv{\phi}{3} \sderiv{\phi}{1}+9 \paren{\sderiv{\phi}{2}}^2\right)}{\paren{\sderiv{\phi}{1}}^3}, \\
                \tilde{\mu}_1
                & \coloneqq \mu_1 \sderiv{\phi}{1}+\frac{5 \sderiv{\phi}{2}}{\sderiv{\phi}{1}}, \qquad
                \tilde{\mu}_2
                 \coloneqq \mu_2 \paren{\sderiv{\phi}{1}}^2+6 \mu_1 \sderiv{\phi}{2}+\frac{5 \left(2 \paren{\sderiv{\phi}{2}}^2+\sderiv{\phi}{1} \sderiv{\phi}{3}\right)}{\paren{\sderiv{\phi}{1}}^2}
              \end{align}
              satisfy
              $\sderiv{\tilde{\gamma}}{i}
                = \tilde{\lambda}_{i - 4} \, \sderiv{\tilde{\gamma}}{4}
                  + \tilde{\mu}_{i - 5} \, \sderiv{\tilde{\gamma}}{5}$
              for $i = 6, 7$.
              
              Moreover, by computation, we obtain
              \begin{equation}
                \tilde{\mu}_2 - \frac{5}{4} \, \tilde{\lambda}_2 - \frac{11}{10} \, \tilde{\mu}_1^2
                = \paren{\sderiv{\phi}{1}}^2 \,
                    \left( \mu_2 - \frac{5}{4} \, \lambda_2 - \frac{11}{10} \, \mu_1^2 \right),
              \end{equation}
              and hence, the value of
              $\sigma_{\paren{4, 5_{\pm 7}}} \in \cparen{0, \pm 1}$
              is invariant under parameter changes.

              Furthermore, by computing $\sderiv{\tilde{\gamma}}{i}$ ($i = 8, 9, 10, 11$) similarly,
              we obtain
              \begin{equation}
                \tilde{\V}_{\paren{4, 5, 11}}
                \in \paren{\sderiv{\phi}{1}}^{14} \, \V_{\paren{4, 5, 11}}
                  + \Span{\sderiv{\gamma}{4}, \sderiv{\gamma}{5}},
              \end{equation}
              where we denote by $\tilde{\V}_{\paren{4, 5, 11}}$ the value of $\V_{\paren{4, 5, 11}}$ for $\tilde{\gamma}$.
              Hence, the condition
              $\V_{\paren{4, 5, 11}}
                \notin \Span{\sderiv{\gamma}{4}, \sderiv{\gamma}{5}}$
              is invariant under parameter changes.
            
            (b) We take a coordinate transformation $\varPhi$
              and consider the curve $\hat{\gamma} \coloneqq \varPhi \circ \gamma$.
              By 
              \cref{cor:Bruno_for_curves} \ref{item:cor_Bruno_curve.coord},
              and \cref{rem:coord_transf.multi_m}
              with $m = 4$, we have
              $
                \sderiv{\hat{\gamma}}{4}
                = J_{\varPhi} \, \sderiv{\gamma}{4}$,
              $\sderiv{\hat{\gamma}}{5}
                = J_{\varPhi} \, \sderiv{\gamma}{5},
              $
              and
              $
                \sderiv{\hat{\gamma}}{i}
                = \lambda_{i - 4} \, \sderiv{\hat{\gamma}}{4}
                  + \mu_{i - 5} \, \sderiv{\hat{\gamma}}{5}
              $
              for $i = 6, 7$.
              Hence, the value of
              $\sigma_{\paren{4, 5_{\pm 7}}} \in \cparen{0, \pm 1}$
              is invariant under coordinate transformations.
              Moreover, since
              $\sderiv{x}{i}_j
                = \lambda_{i - 4} \, \sderiv{x}{4}_j
                  + \mu_{i - 5} \, \sderiv{x}{5}_j$
              for $i = 6, 7$ and $j = 1, \dots, N$,
              we get
              \begin{align}
                \sderiv{\hat{\gamma}}{8}
                & = J_{\varPhi} \, \sderiv{\gamma}{8}
                    + \frac{1}{2}
                      \sum_{i, j = 1}^{N}
                      \pfrac{^2 \varPhi}{x_i \, \partial x_j} \,
                      \sderiv{x}{4}_i \,
                      \sderiv{x}{4}_j,
                \qquad
                \sderiv{\hat{\gamma}}{9}
                = J_{\varPhi} \, \sderiv{\gamma}{9}
                    + \sum_{i, j = 1}^{N}
                        \pfrac{^2 \varPhi}{x_i \, \partial x_j} \,
                        \sderiv{x}{4}_i \,
                        \sderiv{x}{5}_j, \\
                \sderiv{\hat{\gamma}}{10}
                & = J_{\varPhi} \, \sderiv{\gamma}{10}
                   + \sum_{i, j = 1}^{N}
                      \pfrac{^2 \varPhi}{x_i \, \partial x_j} \,
                      \left(%
                        \lambda_2 \, \sderiv{x}{4}_i \, \sderiv{x}{4}_j
                        + \mu_1 \, \sderiv{x}{4}_i \, \sderiv{x}{5}_j
                        + \frac{1}{2} \sderiv{x}{5}_i \, \sderiv{x}{5}_j
                      \right), \\
                \sderiv{\hat{\gamma}}{11}
                & = J_{\varPhi} \, \sderiv{\gamma}{11}
                  + \sum_{i, j = 1}^{N}
                      \pfrac{^2 \varPhi}{x_i \, \partial x_j} \,
                      \left(%
                        \lambda_3 \, \sderiv{x}{4}_i \, \sderiv{x}{4}_j
                        + \paren{\lambda_2 + \mu_2} \, \sderiv{x}{4}_i \, \sderiv{x}{5}_j
                        + \mu_1 \, \sderiv{x}{5}_i \, \sderiv{x}{5}_j
                      \right),
              \end{align}
              where $\gamma (t) = (x_1 (t), \dots, x_N (t))$.
            Therefore,
            $
              \hat{\V}_{\paren{4, 5, 11}}
              = J_{\varPhi} \, \V_{\paren{4, 5, 11}},
            $
            where we denote by $\hat{\V}_{\paren{4, 5, 11}}$ the value of $\V_{\paren{4, 5, 11}}$ for $\hat{\gamma}$.
            Hence, the condition
            $\V_{\paren{4, 5, 11}}
              \notin \Span{\sderiv{\gamma}{4}, \sderiv{\gamma}{5}}$
            is invariant under coordinate transformations.
            Hence, this completes the proof of \cref{conj:criteria.4.5.and.other}.
    \end{proof}

    \section{Curvatures of Curves in $\Rpdf[N]$ of Finite Multiplicities}
    \label{sec:curvature}
    \subsection{Review of Curvatures of Regular Curves in \Rpdf[N]}
     In this subsection, we recall some facts about regular curves in $\R^N$.
     Let $I \subset \R$ be an open interval and 
     $\gamma(t) : I \to \R^N$ be a regular curve in $\R^N$ parametrized by an arbitrary parameter $t$.
     Throughout this subsection, we assume that $\gamma^{(1)}(t), ..., \gamma^{(N-1)}(t)$ are linearly independent for all $t \in I$. 

     Let $e_1, ... , e_{N-1} : I \to \R^N$ be the vector fields along $\gamma$ obtained by applying an orthonormalization to $\gamma^{(1)},..., \gamma^{(N-1)}$, i.e., $e_1 ,..., e_{N-1}$ are defined as follows: 
     \begin{align}
         E_1
          &\coloneqq \gamma^{(1)}, \qquad
         E_2
          \coloneqq \gamma^{(2)} 
           - \left(\gamma^{(2)} \cdot \frac{E_1}{|E_1|}\right) \frac{E_1}{|E_1|}, \\
         E_3 
          &\coloneqq \gamma^{(3)} 
           - \left(\gamma^{(3)} \cdot \frac{E_1}{|E_1|}\right) \frac{E_1}{|E_1|} 
           - \left(\gamma^{(3)} \cdot \frac{E_2}{|E_2|}\right)\frac{E_2}{|E_2|}, \\
         &\hspace{4cm} \vdots \\
         E_{N-1} 
          &\coloneqq \gamma^{(N-1)}
           - \left(\gamma^{(N-1)} \cdot \frac{E_1}{|E_1|}\right)\frac{E_1}{|E_1|} 
           - \cdots 
           - \left(\gamma^{(N-1)} \cdot \frac{E_{N-2}}{|E_{N-2}|}\right)\frac{E_{N-2}}{|E_{N-2}|}, \\
         e_i &\coloneqq \frac{E_i}{|E_i|} \quad (i = 1, ..., N-1).
     \end{align}
     Moreover, we define $e_N : I \to \R^N$ by
     \begin{equation}
         e_N \coloneqq e_1 \times \cdots \times e_{N-1},
     \end{equation}
     where $e_1 \times \cdots \times e_{N-1}$ denotes the cross product of $e_1, ...., e_{N-1} \in \R^N$.
     Let $s = s(t)$ be the arclength parameter of $\gamma$.
     Then, we know that there uniquely exist smooth functions $\kappa_1 (s), ..., \kappa_{N-2} (s) : I \to \poR$ and $\kappa_{N-1} (s) : I \to \R$ such that 
     \begin{equation}
         \label{eq:Frenet}
        \frac{d}{ds}
          (e_1, ..., e_{N})
        = (e_1, ..., e_{N})
         \begin{pmatrix}
            0 & -\kappa_{1} & 0 & \cdots & 0 \\
            \kappa_{1} & 0 & -\kappa_{2} & \ddots & \vdots \\
            0 &  \kappa_{2} & 0 & \ddots & 0 \\
            \vdots & \ddots & \ddots & \ddots & -\kappa_{N - 1} \\
            0 & \cdots & 0 & \kappa_{N - 1} & 0
          \end{pmatrix}
          .
    \end{equation}
        
    Let $V_i(t) : I \to \R$ ($i = 0, ..., N$) be smooth functions defined by
    \begin{equation}
    \label{eq:def.of.Vi}
        V_{i} \paren{t}
        \coloneqq
          \begin{cases}
            1 & (i = 0 ), \\
            \det \left(
                \trans{\bigl( \deriv{\gamma}{1} \paren{t}, \dots, \deriv{\gamma}{i} \paren{t} \bigr)}
                 \,
                  \bigl( \deriv{\gamma}{1} \paren{t}, \dots, \deriv{\gamma}{i} \paren{t} \bigr)
              \right)^{1/2} & (i \ne 0, N), \\
            \det \paren{\deriv{\gamma}{1} \paren{t}, \dots, \deriv{\gamma}{N} \paren{t}} & (i = N).
          \end{cases}
      \end{equation}
      We note that $V_i (t) > 0$ for each $i = 1, ..., N-1$ since $\gamma^{(1)}(t), ..., \gamma^{(N-1)}(t)$ are linearly independent.
      Gluck showed the following relation between $\kappa_1, \dots, \kappa_{N-1}$ and 
      $V_1, \dots, V_N$. 
      For a detailed proof, see \cite{Fukui_2017_multiplicities}.

    \begin{fact}[\!\!{\cite{Gluck_2024_higher_curvatures},\cite{Fukui_2017_multiplicities}}]
    \label{thm:Fukui.curv.}
        The functions $\kappa_{i}$ $(i = 1,... ,N - 1)$ 
        in \eqref{eq:Frenet} are given by
        \begin{equation}
            \kappa_i (s (t)) = \frac{V_{i - 1}(t) \, V_{i + 1}(t)}{V_{1}(t) \, {V_{i}(t)}^2}.
        \end{equation}
    \end{fact}
    
     By \cref{thm:Fukui.curv.} and direct calculation, the invariance of $\kappa_i$ under parameter changes and isometries of $\R^N$ is summarized in \cref{tab:transf.rule.curv.fcn.}.
     For each $i = 1, ..., N-1$, the function $\kappa_i$ is called the \emph{$i$-th curvature function} of $\gamma$.

     \begin{table}[htbp]
        \centering
        \begin{tabular}{|l||r|r|} \hline
        & \makecell{$\kappa_{1}, ..., \kappa_{N-2}$} & \makecell{$\kappa_{N-1}$} \\ \hhline{|=||=|=|}
      \makecell{orientation-preserving \\ parameter changes} & \makecell{unchanging} & \makecell{unchanging} \\ \hline
      \makecell{orientation-reversing \\ parameter changes} & \makecell{unchanging} & \makecell{$\times (-1)^{{\frac{N (N + 1)}{2}}}$} \\ \hline
      \makecell{orientation-preserving \\ isometries of $\R^N$} & \makecell{unchanging} & \makecell{unchanging} \\ \hline
      \makecell{orientation-reversing \\ isometries of $\R^N$} & \makecell{unchanging} & \makecell{$\times (-1)$} \\ \hline
     \end{tabular}
        \caption{Transformation rules for $\kappa_{i}$}
        \label{tab:transf.rule.curv.fcn.}
    \end{table}
    
    The following fact is known as the fundamental theorem of regular curves in $\R^N$.
    
    \begin{fact}[see~\rcite{Theorem~8}{Sulanke_2020_funda_thm_curve}]
    \label{fact:fundamental.reg.curve}
        Let $N \in \geZ{2}$ and let $I \subset \R$ be an open interval. If $\kappa_1, ..., \kappa_{N-2} : I \to \poR$ and $\kappa_{N-1} : I \to \R$ are smooth functions, then there exists a regular curve $\gamma : I \to \R^N$ satisfying the following conditions:
        \begin{enumerate}[label={(\arabic*)}]
            \item 
            the curve $\gamma$ is parametrized by arclength.
            \item
            the vectors $\gamma^{(1)}, ..., \gamma^{(N-1)}$ are linearly independent on $I$.
            \item 
            the functions $\kappa_i$ coincides with $i$-th curvature function of $\gamma$ $(i = 1, ..., N-1)$.
        \end{enumerate}
    Moreover, such a regular curve is unique up to an orientation-preserving isometry of $\R^N$. 
    \end{fact}

    \subsection{Curvatures of Curves with a Singularity of Finite Multiplicity} 

    In this subsection, we define the curvatures of curves in $\R^N$ with 
    finite multiplicity and state their geometric properties.
    
    Let $\epsilon > 0$, and let 
    $\gamma(t) \colon (-\epsilon, \epsilon) \to \R^N$ be of multiplicity $m$ at $t = 0$. 
    \begin{rem}
        \label{rem:assumption(A)(B)}
        Throughout this subsection, we assume
        that
        \begin{enumerate}[label={(\Alph*)}] 
            \item \label{item:assumption_A}
              $\gamma |_{(-\epsilon,0) \cup (0,\epsilon)}$ is regular, and
            \item \label{item:assumption_B}
              $\gamma^{(1)}(t), ..., \gamma^{(N-1)}(t)$ are linearly independent for all $t \in (-\epsilon,0) \cup (0,\epsilon)$.
        \end{enumerate}
    \end{rem}
    The assumption \ref{item:assumption_A} implies that there exists a smooth map 
    $\bm{T} : (-\epsilon, \epsilon) \to \R^N \setminus \{\bzero\}$ such that $\gamma'(t) = t^{m-1} \, \bm{T}(t)$.

    \begin{defn}
        We call the regular curve $\hat{\gamma}(t) : (-\epsilon,\epsilon) \to \R^N$ defined by 
        \begin{equation}
          \hat{\gamma}(t) \coloneqq \int_0^t\bm{T}(v) \, dv
        \end{equation}
        the \emph{associated regular curve} of $\gamma(t)$.
    \end{defn}
            
    Fukui \cite{Fukui_2017_multiplicities} showed the existence of a parameter with good properties for a curve with a singularity of finite multiplicity.
    \begin{fact}[\nscite{Fukui_2017_multiplicities}]
        \label{fact:existence.of.1/m.arclength}
        There exists an orientation-preserving parameter change $u = u(t)$ mapping $0$ to $0$ such that
        \begin{equation}
          \left|\frac{d\gamma}{du}\right| = |u|^{m-1}.
        \end{equation}
    \end{fact}
    
    \cref{fact:existence.of.1/m.arclength} implies the existence of the parameter $\tau$ satisfying 
    \begin{equation}
    \label{eq:1_m_arclength_para}
      \left|\frac{d\gamma}{d\tau}\right| = m|\tau|^{m-1}.
    \end{equation}
    This parameter $\tau$ satisfies $|\tau| = |s(\tau)|^{1/m}$ by using the signed arclength function 
        \begin{equation} \label{eq:singed_arclengrth}
            s(t) \coloneqq \int_0^t |\gamma'(v)| \, dv.   
        \end{equation} 
    Roughly speaking, $\tau$ is the $m$-th root of the arclength.
    The parameter in \eqref{eq:1_m_arclength_para} is called 
    the \emph{$1/m$-arclength parameter} of $\gamma$
    \cite{Martins_Saji_Santos_Teramoto_2024_bdd_geom_inv}.
    In the sequel,  
    we denote the $1/m$-arclength parameter by $\tau$ and a general parameter by $t$.
    We note that $\tau$ is the $1/m$-arclength parameter of $\gamma$ if and only if
        \begin{equation}
            \left|\frac{d\,\hat{\gamma}}{d\tau}\right|=m
        \end{equation}
    holds for the associated regular curve $\hat{\gamma}(\tau)$.
    
    The assumptions \ref{item:assumption_A} and \ref{item:assumption_B} ensure the existence of the curvature functions $\kappa_1, \dots, \kappa_{N-1} \colon (-\epsilon,0) \cup (0,\epsilon) \to \R$ of the regular curve $\gamma|_{(-\epsilon,0) \cup (0,\epsilon)}$. 

    \begin{defn}
        \label{defn:m.nom.curv.fcn}
        For each $i = 1, \dots, N -1$, we define the \emph{$i$-th $m$-normalized curvature functions} $\mu_{m, i} \colon (-\epsilon, 0) \cup (0, \epsilon) \to \R$ of $\gamma$ by
        \begin{equation}
            \mu_{m,i}(t) \coloneqq
            \begin{cases}
                \kappa_{i} \paren{t} \, \abs{s \paren{t}}^{\frac{m - 1}{m}} & (i = 1, ..., N-2), \\
                \paren{\sgn t}^{N \paren{m - 1}} \, \kappa_{N - 1} \paren{t} \, \abs{s \paren{t}}^{\frac{m - 1}{m}} & (i = N-1),
            \end{cases}
        \end{equation}
        where $s(t)$ is the signed arclength function in \eqref{eq:singed_arclengrth}.
    \end{defn}
    \noindent
    By direct calculation, the invariance of the $m$-normalized curvature functions under parameter changes and isometries of $\R^N$ is summarized in \cref{tab:transf.rule.kappa}.

     \begin{table}[htbp]
        \centering
        \begin{tabular}{|l||r|r|} \hline
        & \makecell{$\mu_{m,1}, ..., \mu_{m,N-2}$} & \makecell{$\mu_{m,N-1}$} \\ \hhline{|=||=|=|}
      \makecell{orientation-preserving \\ parameter changes} & \makecell{unchanging} & \makecell{unchanging} \\ \hline
      \makecell{orientation-reversing \\ parameter changes} & \makecell{unchanging} & \makecell{$\times (-1)^{{\frac{N (2 m + N - 1)}{2}}}$} \\ \hline
      \makecell{orientation-preserving \\ isometries of $\R^N$} & \makecell{unchanging} & \makecell{unchanging} \\ \hline
      \makecell{orientation-reversing \\ isometries of $\R^N$} & \makecell{unchanging} & \makecell{$\times (-1)$} \\ \hline
     \end{tabular}
        \caption{Transformation rules for $\mu_{m,i}$}
        \label{tab:transf.rule.kappa}
    \end{table}
    
    In view of the assumption \ref{item:assumption_B} in \cref{rem:assumption(A)(B)}, since 
        \begin{equation}
        \label{eq:transf.gamma.to.hatgamma}
             \left(\gamma^{(1)}(t), ..., \gamma^{(N-1)}(t)\right) 
            = \left(\hat{\gamma}^{(1)}(t), ..., \hat{\gamma}^{(N-1)}(t)\right) 
            \begin{pmatrix}
                t^{m-1} & \cdots & * \\
                    & \ddots & \vdots \\
                  0 &      & t^{m-1}    
            \end{pmatrix},
        \end{equation}
    we have that 
    $\hat{\gamma}^{(1)}(t), ..., \hat{\gamma}^{(N-1)}(t)$ 
    are linearly independent for all $t \in (-\epsilon,0) \cup (0,\epsilon)$.
    Hence, the curvature functions $\hat{\kappa}_1, \dots, \hat{\kappa}_{N-1}$ of the associated regular curve $\hat{\gamma}$ are defined on $(-\epsilon,0) \cup (0,\epsilon)$.
    The $m$-normalized curvature functions can be regarded as the curvature functions of the associated regular curve in the following sense:
    \begin{prop}
    \label{prop:m.nom.curv.and.dual.curv.}
        Let $\hat{\kappa}_{i}(t)$ $(i = 1, ..., N-1)$ be the 
        curvature functions of the associated regular curve $\hat{\gamma}(t)$. 
        Then, the smooth function $g \colon (-\epsilon, \epsilon) \to \R_{>0}$ given by
        \begin{equation} \label{eq:g}
        g(t) \coloneqq 
        \int_0^{\epsilon / 2} \left(\frac{2}{\epsilon}\right)^{m} v^{m-1} \, \left|\hat{\gamma}'\left(\frac{2tv}{\epsilon}\right)\right| \, dv
        \end{equation}
        satisfies  
        \begin{equation}
            g(0) = \frac{|\hat{\gamma}'(0)|}{m} \quad \text{and} \quad
            \mu_{m,i}(t) = \hat{\kappa}_{i}(t) \, g(t)^{\frac{m-1}{m}} 
            \quad (i = 1, ..., N-1)
        \end{equation}
        for all $t \ne 0$. Furthermore, if $\gamma$ is parameterized by the $1/m$-arclength parameter $\tau$, then
        \begin{equation}
        \label{eq:m.nom.curv.fcn.and.dual.curv.}
            \mu_{m,i}(\tau) = \hat{\kappa}_{i}(\tau) \quad
        (i = 1, ..., N-1)
        \end{equation}
        hold for all $\tau \ne 0$.
    \end{prop}
    \noindent 
    To prove \cref{prop:m.nom.curv.and.dual.curv.}, we use the following lemma.
    \begin{lem}[\rcite{Lemma A.4}{Saji_Umehara_Yamada_2022_singularities}]
    \label{fact:division.lem.general.order}
    Let $\phi \colon (-\epsilon, \epsilon) \to \R$ be a smooth function and let $\alpha \ge 0$. Then, there exists a smooth function $g \colon (-\epsilon,\epsilon) \to \R$ satisfying
    \begin{equation}
         g(0) = \frac{\phi(0)}{\alpha + 1} \quad \text{and} \quad
         \int_0^t|v|^{\alpha} \, \phi(v) \, dv = \paren{\sgn t} \, |t|^{\alpha + 1} \, g(t).
    \end{equation}
    \end{lem}
    
    \begin{proof}
        If we set 
        \begin{equation}
            g(t) \coloneqq \int_0^{\epsilon / 2} 
            \biggr(\frac{2}{\epsilon}\biggl)^{\alpha + 1} \, v^{\alpha} \, \phi \biggl (\frac{2tv}{\epsilon} \biggr) \, dv,
        \end{equation}
        then the following calculations give the conclusion.
        \begin{align}
            \int_0^t|v|^{\alpha} \, \phi(v) \, dv 
            &= \int_0^{\epsilon / 2} \frac{2t}{\epsilon} \left|\frac{2tv}{\epsilon}\right|^{\alpha} \, \phi\biggl(\frac{2tv}{\epsilon}\biggr) \, dv 
            = \paren{\sgn t} \, |t|^{\alpha + 1} \, g(t),\\
            g(0)
            &= \left(\frac{2}{\epsilon}\right)^{\alpha + 1} \, \phi(0) \int_0^{\epsilon/2} v^{\alpha} \, dv = \frac{\phi(0)}{\alpha + 1}.
        \end{align}
    \end{proof}

    \begin{proof}[Proof of \cref{prop:m.nom.curv.and.dual.curv.}]
        Let $\kappa_1(t), ..., \kappa_{N-1}(t)$ ($t \ne 0$) be the curvature functions 
         of $\gamma$. Then \cref{thm:Fukui.curv.}, \cref{eq:def.of.Vi,eq:transf.gamma.to.hatgamma} give
        \begin{equation}
            \kappa_{i}(t) = 
            \begin{cases}
                \dfrac{\hat{\kappa}_i(t)}{|t|^{m-1}} & (i = 1, ..., N-2), \\
                \paren{\sgn t}^{\vphantom{\big(}N(m-1)} \, 
                \dfrac{\hat{\kappa}_{N-1}(t)}{|t|^{m-1}} & (i = N-1).
            \end{cases}
        \end{equation}
        By applying \cref{fact:division.lem.general.order} to $|\hat{\gamma}'| \colon (-\epsilon, \epsilon) \to \R_{>0}$, the smooth function $g \colon (-\epsilon, \epsilon) \to \R_{>0}$ in \cref{eq:g} satisfies
        $g(0) = |\hat{\gamma}'(0)|/m$ and 
        $s(t) = \paren{\sgn t} \, |t|^m \, g(t)$.
        \cref{defn:m.nom.curv.fcn} gives the conclusion. Moreover, if $\gamma$ is parameterized by the $1/m$-arclength parameter $\tau$, then $|d\hat{\gamma}/d\tau| = m$, i.e., $g(\tau) = 1$. 
        Thus we get \cref{eq:m.nom.curv.fcn.and.dual.curv.}.
    \end{proof}

    \begin{rem} \label{rem:week.conditions.for.smoothness}
        Fix $i = 1, \dots, N-1$, and suppose that $\hat{\kappa}_i$ and $\mu_{m,i}$ extend to $t = 0$. 
        Then, by \cref{prop:m.nom.curv.and.dual.curv.}, $\hat{\kappa}_i$ is smooth at $t = 0$ if and only if $\mu_{m,i}$ is smooth at $t = 0$. 
        Combining this with
        \begin{equation} \label{eq:comparison.at.0}
            \hat{\gamma}^{\paren{i}}(0) 
            = \frac{(i - 1)!}{(m + i - 2)!} \, \gamma^{\paren{m + i - 1}}(0),
        \end{equation}
        which is obtained by repeatedly differentiating both sides of $\gamma^{\prime}(t) = t^{m-1} \, \hat{\gamma}^{\prime}(t)$, 
        we deduce that if $\gamma^{(m)}(0), \dots, \gamma^{(m+N-2)}(0)$ are linearly independent,
        then $\mu_{m, i}$ is smooth at $t = 0$.
        
        In particular, for $N = 2$, 
        the first $m$-normalized curvature function 
        $\mu_{m, 1}$ is always smooth at $t = 0$. 
        Specifically, it coincides with the normalized curvature function defined in 
        \cite{Shiba_Umehara_2012_curvature_functions} for $m = 2$, and with its generalization to plane curves with singularities of finite multiplicity introduced in 
        \cite{Martins_Saji_Santos_Teramoto_2024_bdd_geom_inv}
        for $m \ge 2$. 
        Furthermore, $\mu_{m, 1}$ is related to the smooth curvature 
        function $\kappa$ given by Fukui and Hoshino 
        \cite{Fukui_Hoshino_curvature_criteria} for such curves by the equation 
        $\mu_{m, 1}(\tau) = (m!)^{1 / m} \, \kappa ((m!)^{1 / m} \, \tau) / m$.

        If $\gamma$ is analytic, the smoothness of $\mu_{m,i}$ at $t = 0$ follows more straightforwardly by observing how it extends to $t = 0$ (see \ref{item:smoothness.of.norm.curv.} in \cref{thm:smoothness.of.nom.curv.fcn}).
    \end{rem}
    
    Let $v_{m,i} \in \R$ ($i = 0, ..., N$) be constants defined by
    \begin{equation}
        v_{m,i}
        \coloneqq
          \begin{cases}
            1 & (i = 0), \\
            \det \left( 
                {}^t
                  \bigl( \deriv{\gamma}{m} \paren{0}, ..., \deriv{\gamma}{m+i-1} \paren{0} \bigr)
                 \,
                  \bigl( \deriv{\gamma}{m} \paren{0}, ..., \deriv{\gamma}{m+i-1} \paren{0} \bigr)
              \right)
            ^{\sfrac{1}{2}} & (i \ne 0, N), \\
            \det \paren{\deriv{\gamma}{m} \paren{0}, ..., \deriv{\gamma}{m+N-1} \paren{0}} & (i = N).
          \end{cases}
    \end{equation}
    We note that for $i = 1, ..., N-1$, the linear independence of 
    $\gamma^{(m)}(0), ..., \gamma^{(m+i-1)}(0)$ is
    equivalent to $v_{m,i} > 0$. 
    
    \begin{thm}
     \label{thm:smoothness.of.nom.curv.fcn}
     \begin{enumerate}[label={(\Roman*)}]
       \item \label{item:existence.of.sing.curv.}
        For each $i = 1, ..., N-1$, 
        if $\gamma^{(m)}(0), ..., \gamma^{(m+i-1)}(0)$ are linearly independent, 
        then $\mu_{m,i}$ extends continuously to $t = 0$. 
        In particular, it holds that
        \begin{equation}
          \lim_{t \to 0} \mu_{m,i}(t)
          = \frac{i \, (m!)^{1/m}}{m \, \paren{m + i - 1}} \, 
          \frac{v_{m,i-1} \, v_{m,i+1}}{v_{m,1}^{1/m} \, v_{m,i}^2}.
        \end{equation}
        
        \item \label{item:smoothness.of.norm.curv.}
        Let the curve $\gamma$ be analytic.
        \begin{enumerate}[label={(\arabic*)}]
            \item \label{item:smoothness.N-2}
             For each $i = 1, ..., N-2$, if $\mu_{m,i}$ extends continuously to $t = 0$ with 
             $\mu_{m,i}(0) > 0$, then $\mu_{m,i}$ is smooth at $t = 0$.
             
             \item \label{item:smoothness.2,3}
             If $N = 2, 3$ and $\mu_{m,N-1}$ extends continuously to $t = 0$, then $\mu_{m,N-1}$ is smooth at $t = 0$.
             
             \item \label{item:smoothness.ge.4}
             If $N \ge 4$ and $\mu_{m,N-1}$ extends continuously to $t = 0$ with $\mu_{m,N-1}(0) > 0$, then $\mu_{m,N-1}$ is smooth at $t = 0$.

        \end{enumerate}
    \end{enumerate}
    \end{thm}
    
    \begin{rem}
        
        The converse of \ref{item:existence.of.sing.curv.} does not hold in general. In fact, for $\gamma(t) \coloneqq (t^2, t^4, t^5)$, the second $2$-normalized curvature function has a limit at $t=0$, whereas $\gamma''(0) = (2, 0, 0)$ and $\gamma'''(0) = (0, 0, 0)$ are linearly dependent.

        In addition, the assumption $\mu_{m,N-1}(0) > 0$ in \ref{item:smoothness.ge.4} cannot be dropped. Indeed, the $(N-1)$-th $2$-normalized curvature function of $\gamma(t) \coloneqq (t^2, t^3, ..., t^{N-2}, t^N, t^{N+1}, t^{N+3})$ is expressed as $\mu_{2,N-1}(t) = |t| \phi(t)$ by using some smooth function $\phi$.
    \end{rem}
    
    \begin{proof}[Proof of \cref{thm:smoothness.of.nom.curv.fcn}]
        First, we show \ref{item:existence.of.sing.curv.}. 
        Fix $i = 1, \dots N-1$ and 
        assume that $\gamma^{\paren{m}}(0), \dots, \gamma^{\paren{m+i-1}}(0)$ are linearly independent. Then, \eqref{eq:comparison.at.0} yields 
        \begin{equation}
            \hat{V}_i(0) = \frac{0! \, 1! \cdots (i-1)!}{(m-1)! \, m! \cdots (m+i-2)!}\,v_{m,i} \ne 0,
        \end{equation}
        where $\hat{V}_i$ is the function defined by replacing $\gamma$ with $\hat{\gamma}$ in \cref{eq:def.of.Vi}. This implies
        \begin{equation}
            \lim_{t \rightarrow 0} \hat{\kappa}_i(t) 
            = \frac{\hat{V}_{i-1}(0) \, \hat{V}_{i+1}(0)}{\hat{V}_{1}(0) \, \hat{V}_i(0)}
            = \frac{i (m-1)!}{m + i -1} \frac{v_{m,i-1 } \, v_{m,i+1}}{v_{m,1} \, v_{m,i}^2}.
        \end{equation}
        Therefore, \cref{prop:m.nom.curv.and.dual.curv.} gives the conclusion.

        To show \ref{item:smoothness.N-2}, we fix $i = 1, \dots, N-2$ and assume $\mu_{m,i}$ extends continuously to $t = 0$ with $\mu_{m,i}(0) > 0$.
        Since $\gamma$ is analytic, so is $\hat{V}_i(t)^2$.
        Moreover, because $\hat{V}_i(t)^2$ is positive for all $t \ne 0$, 
        there exist non-negative integers $l_{i-1}, l_i, l_{i+1}$ and  
        smooth functions $\psi_{i-1}, \psi_i, \psi_{i+1} \colon (-\epsilon, \epsilon) \to \R_{>0}$ such that $\hat{V}_j(t)^2 = t^{2l_j}\psi_j(t)$ ($j = i-1, i, i+1$).
        Hence, we have 
        \begin{equation}
            \hat{\kappa}_i(t) = |t|^{l_{i-1}+l_{i+1}-2l_i} \frac{\psi_{i-1}(t) \, \psi_{i+1}(t)}{\hat{V}_1(t) \, \psi_i(t)}.
        \end{equation}
        If $l_{i-1}+l_{i+1}-2l_i < 0$, then $\mu_{m,i}(t) \rightarrow \infty \ (t \rightarrow 0)$. This contradicts the existence of a continuous extension to $t=0$. On the other hand, if $l_{i-1}+l_{i+1}-2l_i > 0$ then $\mu_{m,i}(t) \rightarrow 0 \ (t \rightarrow 0)$. This contradicts $\mu_{m,i}(0) > 0$. Therefore, $l_{i-1}+l_{i+1}-2l_i = 0$, and $\mu_{m,i}(t)$ is smooth at $t=0$. 

        Next, we show \ref{item:smoothness.2,3}. 
        For the case $N = 2$, we discuss the smoothness of $\mu_{m, N-1}$ in \cref{rem:week.conditions.for.smoothness}.
       Let $N =3$ and let $\mu_{m,N-1}(t)$ extend continuously to $t =0$. 
       By 
       the analyticity of $\hat{V}_2(t)^2$ and $\hat{V}_3(t)$,
       there exist non-negative integers $l_2, l_3$, and smooth functions $\psi_2 \colon (-\epsilon, \epsilon) \to \R_{>0}$, $\psi_3 \colon (-\epsilon, \epsilon) \to \R$ such that
        \begin{equation}
            \hat{V}_2(t)^2 = t^{2l_2} \, \psi_2(t), \quad 
            \hat{V}_3(t) = t^{l_3} \, \psi_3(t), \quad \psi_3(0) \ne 0.
        \end{equation}
        Hence the $(N-1)$-th curvature function $\hat{\kappa}_{N-1}(t)$ of $\hat{\gamma}(t)$ takes the form
        \begin{equation}
            \hat{\kappa}_{N-1}(t) = \frac{\hat{V}_3(t)}{\hat{V}_2(t)^2} = t^{l_3-2l_2}\frac{\psi_3(t)}{\psi_2(t)}.
        \end{equation}
        The same argument as in the proof of \ref{item:smoothness.N-2} yields $l_3 - 2l_2 \ge 0$, and thus $\mu_{m,N-1}(t)$ is smooth at $t=0$.

        Let us show \ref{item:smoothness.ge.4}. 
        The analyticity of $\hat{V}_{N - 2}(t)^2, \hat{V}_{N-1}(t)^2$, and $\hat{V}_N(t)$
        implies that there exist non-negative integers $l_{i}$ ($i= N-2, N-1, N$) and smooth functions $\psi_i \colon (-\epsilon, \epsilon) \to \R_{>0}$ ($i= N-2, N-1$), $\psi_N \colon (-\epsilon, \epsilon) \to \R$ such that
        \begin{align}
            \hat{V}_i(t)^2 &= t^{2l_i}\, \psi_i(t) \quad (i = N-2, N-1), \\
            \hat{V}_N(t) & = t^{l_N} \, \psi_N(t), \quad \psi_N(0) \ne 0.
        \end{align}
        Using the smooth function $g$ obtained by \cref{prop:m.nom.curv.and.dual.curv.}, we can write $\mu_{m,N-1}$ as 
        \begin{equation}
            \mu_{m, N-1}(t) = \paren{\sgn t}^{l_{N-2}} \, t^{l_{N-2}+l_N-2l_{N-1}} \, \frac{\sqrt{\psi_{N-2}(t)} \, \psi_N(t)}{\hat{V}_1(t) \, \psi_{N-1}(t)} \, g(t)^{\frac{m-1}{m}}.
        \end{equation}
        The same argument as in the proof of \ref{item:smoothness.N-2} yields $l_{N-2} + l_N - 2l_{N-1} = 0$. If $l_{N-2}$ is an odd number, then 
        \begin{equation}
            \lim_{t \rightarrow +0} \mu_{m,N-1}(t) \ne \lim_{t \rightarrow -0} \mu_{m,N-1}(t).
        \end{equation}
        Therefore, $l_{N-2}$ must be even, and thus $\mu_{m,N-1}(t)$ is smooth at $t=0$.
    \end{proof}

     \begin{defn}
        \label{defn:kappa.m.i}
         For fixed $i = 1, ..., N-1$, let $\gamma^{(m)}(0), ..., \gamma^{(m+i-1)}(0)$ be linearly independent. Then we call
         \begin{equation}
             \sigma_{m,i} \coloneqq
             \frac{v_{m,i-1} \, v_{m,i+1}}{v_{m,1}^{1/m} \, v_{m,i}^2}
         \end{equation}
         the \emph{$i$-th $m$-singular curvature}.
     \end{defn}
     
    \noindent
    By \ref{item:existence.of.sing.curv.} in \cref{thm:smoothness.of.nom.curv.fcn}, we know that the same invariance under parameter changes and isometries of $\R^N$ as those of the $m$-normalized curvature functions holds for $\sigma_{m,1}, ..., \sigma_{m,N-1}$.

    When $N = m = 2$, the first $2$-singular curvature $\sigma_{2, 1}$ is equal to the cuspidal curvature introduced in \cite{Umehara_2011_simplification}. 
    Moreover, the statement of \ref{item:existence.of.sing.curv.} in \cref{thm:smoothness.of.nom.curv.fcn} is a generalization of the relation between the normalized curvature function 
    and the cuspidal curvature given in \cite{Shiba_Umehara_2012_curvature_functions}.
    In the case of $N = 2$ and $m \ge 2$, the first $m$-singular curvature $\sigma_{m, 1}$ of the $(m, m + 1)$-cusp coincides with the $(m, m + 1)$-cuspidal curvature in 
    \cite{Martins_Saji_Santos_Teramoto_2024_bdd_geom_inv}.

    The existence and uniqueness of curves with a singularity of finite multiplicity are shown in \nscite{Fukui_2017_multiplicities}. Reinterpreting this result in terms of the $m$-normalized curvature functions and the $1/m$-arclength parameter, we obtain the following.
    \begin{thm}
    \label{thm:existence.uniqueness.with.sing.}
        Let $N \in \geZ{2}$ and $\epsilon >0$. 
        If $\mu_1(\tau), \dots, \mu_{N-2}(\tau) \colon (-\epsilon,\epsilon) \to \poR$ and $\mu_{N-1}(\tau) : (-\epsilon,\epsilon) \to \R$ are smooth functions, then there exists a curve $\gamma(\tau) : (-\epsilon,\epsilon) \to \R^N$ of multiplicity $m$ at $\tau = 0$ satisfying the following:
        \begin{itemize}
            \item $\tau$ is the $1/m$-arclength parameter of $\gamma$.
            \item $\gamma^{(1)}(\tau), ..., \gamma^{(N-1)}(\tau)$ are linearly independent for all $\tau \in (-\epsilon,0) \cup (0, \epsilon)$.
            \item The $m$-normalized curvature functions of $\gamma$ are smooth 
            at $\tau = 0$, and
            $\mu_{i}$ coincides with the $i$-th 
            $m$-normalized curvature function 
            $(i = 1, ..., N-1)$.
        \end{itemize}
        Moreover, such a curve is unique up to an orientation-preserving isometry of $\R^N$.
    \end{thm}

    \begin{rem}
        Setting $N = 2$ in \cref{thm:existence.uniqueness.with.sing.} reduces to Proposition 3.2 in \cite{Martins_Saji_Santos_Teramoto_2024_bdd_geom_inv}. 
        In this case, the resulting curve $\gamma(\tau)$ is given by
        \begin{equation}
            \gamma(\tau) = 
            \int_0^{\tau} m \, v^{m-1} \paren{\cos\theta(v), \sin\theta(v)} \, dv, \quad
            \theta(\tau) = 
            \int_0^{\tau} m \, \mu_1(v) \, dv.
        \end{equation}
        Furthermore, the unit normal vector field of $\gamma(\tau)$ is given by
        $(-\sin \theta(\tau), \cos\theta(\tau))$.
    \end{rem}
    
    \begin{proof}[Proof of \cref{thm:existence.uniqueness.with.sing.}]
        By \cref{fact:fundamental.reg.curve}, there exists a regular curve $\gamma_1(\tau) : (-\epsilon, \epsilon) \to \R^N$ such that
        \begin{itemize}
            \item $\tau$ is the arclength parameter of $\gamma_1$,
            \item $\gamma_1^{(1)}(\tau), ..., \gamma_1^{(N-1)}(\tau)$ are linearly independent for all $\tau \in (-\epsilon, \epsilon)$, and
            \item $m\mu_{i}(\tau)$ is the $i$-th curvature function of $\gamma_1(\tau)$ ($i = 1, ..., N-1$).
        \end{itemize}
        We set $\gamma_2 = m \gamma_1$. Then $\gamma_2$ is a regular curve such that
        \begin{itemize}
            \item $\left|\dfrac{d\gamma_2}{d\tau}\right| = m$,
            \item $\gamma_2^{(1)}(\tau), ..., \gamma_2^{(N-1)}(\tau)$ are linearly independent for all $\tau \in (-\epsilon, \epsilon)$, and
            \item $\mu_{i}(\tau)$ is the $i$-th curvature function of $\gamma_2(\tau)$ ($i = 1, ..., N-1$).
        \end{itemize}
        Let $\bm{T} \coloneqq d\gamma_2 / d\tau$. 
        The desired curve $\gamma(\tau) : (-\epsilon, \epsilon) \to \R^N$ is given by
        \begin{equation}
            \gamma(\tau) \coloneqq \int_0^{\tau}v^{m-1}\, \bm{T}(v) \, dv.
        \end{equation}
        It is clear that $\gamma(\tau)$ is a curve of multiplicity $m$ at $\tau = 0$ parameterized by the $1/m$-arclength parameter $\tau$, and its associated regular curve is $\gamma_2$. 
        Thus,
        the linear independence of $\gamma_2^{(1)}(0), ..., \gamma_2^{(N-1)}(0)$ and \cref{rem:week.conditions.for.smoothness}
        guarantee the smoothness of the $m$-normalized curvatures $\mu_{m,1}(\tau), ..., \mu_{m,N-1}(\tau)$ of $\gamma(\tau)$ at $\tau = 0$.
        Moreover, \cref{prop:m.nom.curv.and.dual.curv.} implies that 
        \begin{equation}
            \mu_{m,i}(\tau) = \mu_{i}(\tau) \quad (i = 1, ..., N-1)
        \end{equation}
        for all $\tau \ne 0$. Since both $\mu_{i,m}$ and $\mu_{i}$ are continuous, they also coincide at $\tau = 0$.
        The uniqueness of $\gamma$ follows from that of $\gamma_1$.
    \end{proof}

\bibliographystyle{alpha}   

\begin{thebibliography}{99}
  \bibitem{BG82_simple_sing}
    J.~W. Bruce and T.~Gaffney, Simple singularities of mappings ${\bf C},0\rightarrow {\bf C}\sp{2},0$, J. London Math. Soc. (2) \textbf{26} (1982), 465--474.
  \bibitem{Bruce_1987_determinacy_and_unipotency}
    J.~W.~Bruce, A.~A.~du~Plessis, and C.~T.~C.~Wall, Determinacy and unipotency, Inventiones mathematicae \textbf{88} (1987), 521--554.
  \bibitem{Fukui_2017_multiplicities}
    T.~Fukui, Local differential geometry of singular curves with finite multiplicities, Saitama Mathematical Journal \textbf{31} (2017), 79--88.
  \bibitem{Fukui_Hoshino_curvature_criteria}
    T.~Fukui and S.~Hoshino,
    Curvature criteria of $\mathcal A$-simple singularities 
    $\mathbb R, 0 \to \mathbb R^2, 0$ and their parallel curves,
    2025, arXiv:\href{https://arxiv.org/abs/2512.23293}{2512.23293}.
  \bibitem{Gluck_2024_higher_curvatures}
    H.~Gluck, Higher Curvatures of Curves in Euclidean Space, The American Mathematical Monthly \textbf{73} (1966), 699--704.
  \bibitem{Hattori_Honda_Morimoto_2024_Bours_theorem_singularities_arXiv}
    Y.~Hattori, A.~Honda and T.~Morimoto, Bour's theorem for helicoidal surfaces with singularities, Differential Geom. Appl. \textbf{99} (2025).
  \bibitem{Honda_Saji_25_invariant}
    A.~Honda and K.~Saji, Geometric invariants of $5/2$-cuspidal edges, Kodai Math. J. \textbf{42} (2019), 496--525.
  \bibitem{Liu_Xin_2023_Frobenius_numbers}
    F.~Liu and G.~Xin, A combinatorial approach to Frobenius numbers of some special sequences, Adv. in Appl. Math. \textbf{158} (2024).
  \bibitem{Matsushita_2024_classifications_4_5_cusps_arXiv}
    Y.~Matsushita, Classifications of cusps appearing on plane curves, 2024, arXiv:\href{https://arxiv.org/abs/2402.12166}{2402.12166}.
  \bibitem{Mishkov_2000_Generalization_Bruno}
    R.~L.~Mishkov, Generalization of the formula of Faa di Bruno for a composite function with a vector argument, International Journal of Mathematics and Mathematical Sciences \textbf{24} (2000), 481--491.
  \bibitem{Martins_Saji_Santos_Teramoto_2019_singular_surfaces}
    L.~F.~Martins, K.~Saji, S.~P.~dos~Santos, and K.~Teramoto, Singular surfaces of revolution with prescribed unbounded mean curvature, Anais da Academia Brasileira de Ci{\^e}ncias \textbf{91} (2019).
  \bibitem{Martins_Saji_Santos_Teramoto_2024_bdd_geom_inv}
    L.~F.~Martins, K.~Saji, S.~P.~dos~Santos, and K.~Teramoto, Boundedness of geometric invariants near a singularity which is a suspension of a singular curve, Rev. Un. Mat. Argentina \textbf{67} (2024), 475--502.
  \bibitem{Okano_2001_Bruno_applications}
      T.~Okano, Y.~Okuto, A.~Shimizu, Y.~Niikura, Y.~Hashimoto, and H.~Yamada, The formula of {Faà di Bruno} and its applications I, Annual review (Nagoya City University), in Japanese, \textbf{5} (2001), 35--44.
  \bibitem{Porteous_2001_curves_and_surfaces}
    I.~R.~Porteous, Geometric differentiation: for the intelligence of curves and surfaces, Cambridge University Press, 1994.
  \bibitem{Shiba_Umehara_2012_curvature_functions}
    S.~Shiba and M.~Umehara, The behavior of curvature functions at cusps and inflection points, Differential Geometry and its Applications \textbf{30} (2012), 285--299.
  \bibitem{Subwattanachai_2024_generalized_Frobenius_number}
    K.~Subwattanachai, Generalized {Frobenius} number of three variables, 2024, arXiv:\href{https://arxiv.org/abs/2309.09149v2}{2309.09149v2}.
  \bibitem{Sulanke_2020_funda_thm_curve}
    R.~Sulanke, The fundamental theorem for curves in the $n$-dimensional Euclidean space, 2020, \url{https://www2.mathematik.hu-berlin.de/~sulanke/diffgeo/euklid/ECTh.pdf}.
  \bibitem{Sylvester_1857_partition_of_numbers}
    J.~J.~Sylvester, On the partition of numbers, Quarterly Journal of Pure and Applied Mathematics \textbf{1} (1857), 141--152.
  \bibitem{Sylvester_1882_subvariants}
    J.~J.~Sylvester, On subvariants, i.e., semi-invariants to binary quantics of an unlimited order, American Journal of Mathematics \textbf{5} (1882), 119--136.
  \bibitem{Sylvester_1884_mathematical_questions}
    J.~J.~Sylvester, Mathematical questions with their solutions, Educational Times \textbf{41} (1884), 21.
  \bibitem{Umehara_2011_simplification}
    M.~Umehara, A simplification of the proof of Bol's conjecture on sextactic points, Proceedings of the Japan Academy, Series A, Mathematical Sciences \textbf{87} (2011), 10--12.
  \bibitem{Saji_Umehara_Yamada_2022_singularities}
    M.~Umehara, K.~Saji and K.~Yamada, Differential geometry of curves and surfaces with singularities, translated from the 2017 Japanese original by Wayne Rossman, 
    Series in Algebraic and Differential Geometry, 1, World Sci. Publ., Hackensack, NJ, 2022.
  \bibitem{Whitney_1943_Diff_even_fcn}
    H.~Whitney, Differentiable even functions, Duke Mathematical Journal \textbf{10} (1943), 159--160.
  \bibitem{Zhang_Pei_oprik}
    C.~Zhang and D.~Pei, Evolutes of $(n, m)$-cusp curves and application in optical system, Optik, \textbf{162} (2018), 42--53.
\end{thebibliography}

\end{document}